\documentclass[letterpaper,12pt,extrafontsizes,oneside]{memoir}
\DoubleSpacing
\usepackage{amsmath,amsthm,amsfonts,amssymb,mathrsfs}
\usepackage{graphicx}
\usepackage{epstopdf}
\usepackage{multirow}
\usepackage{subcaption}
\usepackage[binary-units]{siunitx}
\usepackage[numbers,sort&compress]{natbib}
\usepackage[hidelinks=true]{hyperref}
\usepackage[utf8x]{inputenc}
\usepackage{tikz-cd}
\usepackage{babel}
\usepackage{microtype}
\usetikzlibrary{babel}
\let\amsamp=&
\makeatletter

\setlrmarginsandblock{1in}{1in}{*}
\setulmarginsandblock{1in}{1in}{*}

\maxtocdepth{subsection}

\setsecnumdepth{subsection}

\chapterstyle{ger}

\checkandfixthelayout
\newtheorem{theorem}{Theorem}[section]
\newtheorem*{theorem*}{Theorem}
\newtheorem*{definition*}{Definition}
\newtheorem{proposition}[theorem]{Proposition}
\newtheorem{lemma}[theorem]{Lemma}
\newtheorem{corollary}[theorem]{Corollary}

\theoremstyle{definition}
\newtheorem{definition}[theorem]{Definition}
\newtheorem{question}[theorem]{Question}
\newtheorem{remark}[theorem]{Remark}
\newtheorem{warning}[theorem]{Warning}
\newtheorem{example}[theorem]{Example}

\newtheorem{construction}[theorem]{Construction}
\newtheorem{claim}[theorem]{Claim}

\DeclareMathOperator{\im}{Im}
\newcommand{\otimesl}{\otimes^{\boldsymbol{\mathrm{L}}}}
\newcommand{\bb}[1]{\mathbb{#1}}
\newcommand{\cl}[1]{\mathrm{cl}\left(#1\right)}
\newcommand{\simpmod}[1]{{#1\mathrm{-mod}}_{\Delta}}
\newcommand{\ideal}[1]{\mathfrak{#1}}
\renewcommand{\to}{\rightarrow}

\renewcommand{\over}[2]{\stackrel{#1}{#2}}
\renewcommand{\cal}[1]{\mathcal{#1}}
\newcommand{\scr}[1]{\mathscr{#1}}
\renewcommand{\mod}[1]{{#1}\mathrm{-mod}}
\newcommand{\alg}[1]{{#1}\mathrm{-alg}}

\DeclareMathOperator{\coeq}{coeq}
\DeclareMathOperator{\eq}{eq}
\DeclareMathOperator{\hocoeq}{hocoeq}
\DeclareMathOperator*{\hocolim}{hocolim}
\DeclareMathOperator*{\colim}{colim}
\let\lim\relax
\DeclareMathOperator*{\lim}{lim}
\DeclareMathOperator{\hocoker}{hocoker}
\DeclareMathOperator{\Spec}{Spec}

\DeclareMathOperator{\Ext}{Ext}
\DeclareMathOperator{\Hom}{Hom}
\DeclareMathOperator{\coker}{coker}
\DeclareMathOperator{\Tor}{Tor}

\DeclareMathOperator{\ob}{Ob}
\DeclareMathOperator{\Mor}{Mor}
\DeclareMathOperator{\rad}{rad}

\DeclareMathOperator{\spec}{Spec}
\DeclareMathOperator{\fun}{Fun}
\DeclareMathOperator{\Sym}{Sym}
\DeclareMathOperator{\spf}{Spf}
\DeclareMathOperator{\hofib}{hofib}
\DeclareMathOperator{\Hot}{Hot}
\DeclareMathOperator{\DK}{DK}
\DeclareMathOperator{\R}{R}
\title{The Miracle of Flatness in Algebraic Geometry}
\author{Ivan Zelich}
\renewcommand{\maketitlehookc}{\vspace*{\fill}\centering Submitted in partial fulfillment of the \\
requirements for the degree of\\
Doctor of Philosophy\\
in the Graduate School of Arts and Sciences\\
\vspace{1cm}
Columbia University}
\date{2026}
\aliaspagestyle{title}{empty}
\begin{document}
\maketitle
\vspace*{\fill}
\newcommand{\thecopyrightdate}{2026}
\begin{center}
  \copyright\ \thecopyrightdate \\
  \theauthor \\
  All rights reserved
\end{center}
\thispagestyle{empty}
\renewcommand{\abstractnamefont}{\normalfont}
\renewcommand{\abstractname}{\uppercase{Abstract}}
\begin{abstract}
  \begin{center}
    \thetitle\\
    \theauthor
  \end{center}
  \vspace{1cm}
\pagestyle{empty}

Algebraic Geometry was revolutionized in the 1960s by many incredible techniques introduced by Grothendieck and his collaborators. Arguably 
the most fundamental machinery was the introduction of sheaf cohomology, which in turn relied on the developments of homotopical algebra. 
Here, the notion of flatness becomes a crucial technical condition, especially when considering the interaction between base-change and cohomology, where the flatness condition 
is imposed on a morphism of schemes $f: X \to Y$ to rid $X$ of any pathological behavior for the resulting family of fibers 
$X_y$ for $y \in Y$. It would be unreasonable, however, to expect that the flatness condition on $f$ is generally sufficient to guarantee good behavior of the family of fibers $X_y$, 
because flatness is a module-theoretic condition that can generally exhibit quite strange behavior without additional finiteness hypotheses. Broadly speaking,
the pursuit of understanding this dichotomy is an underlying theme for this thesis.\\
\indent Semirings, studied in Chapter 1, provided new aspects of this dichotomy to wrestle with. One of our main results is that 
the most natural notion of homotopy theory for semirings is not enough to capture the properties of flatness required to guarantee good base-change properties.
\begin{theorem}[See Corollary~\ref{cor:main}]
For any commutative monoid $M \in \mod{\bb{N}}$, $M \otimesl_{\bb{N}} \bb{Z}$ is a discrete module. On the other hand, 
the semiring map $\bb{N} \to \bb{Z}$ is not flat, in the sense that the functor $-\otimes_{\bb{N}} \bb{Z}:\mod{\bb{N}} \to \mod{\bb{N}}$ does not preserve finite limits (see Corollary~\ref{cor:NRnotflat}).
\end{theorem}
While this result is indeed in stark contrast to what happens in classical commutative algebra, where flatness of a ring map $R \to S$ can be detected via the derived functors of $-\otimesl_R S$, 
the semiring situation does bear some similarities to the classical situation.
\begin{theorem}[See Theorem~\ref{theorem:main}]
A flat, finitely presented epimorphism of semirings $f: A \to B$ induces an isomorphism of $\text{Spec}(B)$ onto a quasi-compact open of $\text{Spec}(A)$ in the Zariski-topology.
\end{theorem}
We were informed after having proved Theorem~\ref{theorem:main} that Florian Murty obtained the same result in greater generality via
a different method (see \cite{Martyflat}).\\
\indent Our main interest in studying the flatness of semirings is that here the notion seems to capture some strong positivity properties of semirings. The following is 
a question posed to the author by James Borger.
\begin{question}
Let $A$ be a finitely presented flat $\bb{R}_{+}$-algebra. 
If $A$ is non-trivial, does it necessarily have an $\bb{R}$-point (or even better, an $\bb{R}_{+}$-point)?
\end{question}
If we asked the above question for the groupification $\bb{R}$ of $\bb{R}_{+}$, then the answer would be no, as there are many non-trivial flat and finitely presented $\bb{R}$-algebras that do not have any $\bb{R}$-points.
In our opinion, the question seems to suggest that because $\bb{R}_{+}$ is a totally ordered semiring with no elements admitting additive inverses, the flatness condition on $A$ is strong enough to guarantee the existence of a $\bb{R}_{+}$-point.
We attempted to address this question by understanding when quotients of the form $\bb{R}_{+}[x_1,...,x_n]/(f \sim g)$ are flat over $\bb{R}_{+}$, but we were not able to obtain 
much progress. In private communication, Johan de Jong informed the author of the following result.
\begin{theorem}[de Jong]
The semiring $\bb{R}_{+}[x]/(ax \sim x^2 + 1)$ is flat over $\bb{R}_{+}$ if and only if $a \ge 2$. 
\end{theorem}

On the other hand, in Chapter 2, we will show that from the perspective of \textit{descendability}, flat ring maps between classical rings can behave 
arbitrarily badly. The notion of a descendable morphism is a technical condition introduced by Akhil Mathew \cite{DescendMathew} 
building upon the work of Paul Balmer \cite{DescendBalmer}. Informally, if a map $f: A \to B$ is descendable, then $A$-linear stable $\infty$-categories can
be understood via $B$-linear stable $\infty$-categories together with `descent' data (see \ref{thm:luriedescend}). Our main result is that 
faithfully flat ring maps are not necessarily descendable; see Corollary~\ref{corollary:mainindiv}, Theorem~\ref{theorem:maincounter}, and the work of Aoki \cite{aoki}. 
This answers a question by Bhatt and Scholze (\cite[11.24]{BhattScholzeProj}), Akhil Mathew (\cite[Prop. 3.32]{DescendMathew}), and Jacob Lurie (\cite[D.3.3.4]{LurieSAG}). 
Our strategy was to first understand the relationship between non-vanishing cup-products and cardinality in the module-theoretic setting.
\begin{theorem}[See Theorem~\ref{theorem:indivisiblecup}]
For a commutative ring $R$, if
$R$ contains an $n$-indivisible sequence (see Definition~\ref{definition:indivisible}), then 
there exists a flat module $M$ over $R$ and a class $\eta \in \Ext^1_R(M, R')$ such that 
\[\eta^{\otimesl_R n} \in \Ext^n_R(M^{\otimesl_R n}, R'^{\otimesl_R n}) \neq 0\]
where $R'$ is a free $R$-module.
\end{theorem}
We then form a quite general method to convert module theoretic non-vanishing cup-products to non-descendable ring maps, and consequently obtain many examples of non-descendable faithfully flat ring maps (See Corollary~\ref{corollary:mainindiv}).
Our method in particular allows us to construct examples between $p$-boolean rings, which was a question posed to the author by Juan Esteban Rodríguez Camargo. By slightly refining our argument, we were also able to construct a faithfully flat cover of $\spec{k[x_1,x_2,...]}$, where $k$ is an algebraically closed field, that is not descendable (see Theorem~\ref{theorem:maincounter}).
This example is quite striking, as $\spec{k[x_1,x_2,...]}$, while not Noetherian, is still a reasonably nice scheme from the perspective of algebraic geometry.\\

To wrap up this thesis, we homed in on a conjecture posed by Grothendieck and Dieudonné. In \cite[21.12.14]{EGAIV}, 
the authors conjectured that for a locally of finite type map $f: X \to Y$ of excellent, locally Noetherian schemes, with $X$ normal and $Y$ regular, the 
ramification locus of $f$ is pure of codimension $1$ (See Theorem~\ref{thm:nagata-purity}). Having already shown that for any open subscheme $V \subset X$ such that $V \to X$ is an affine morphism of schemes,
 $X \setminus V$ is pure of codimension $1$ (\cite[21.12.7]{EGAIV}), the authors then conjectured that
if $V$ is the maximal open subscheme $V \subset X$ where $f$ is unramified (and hence \'{e}tale), then $V \to X$ is an affine morphism of schemes (see (v) \cite[21.12.14]{EGAIV}).
The main goal of this chapter was to answer their conjecture in the affirmative, see Theorem~\ref{thm:affine-mixed}. We will in particular demonstrate the following:
\begin{theorem}[See Theorem~\ref{thm:cohpuregen}]
Let $(A,\ideal{m}_A,k)$ be a regular local ring and $V \to \spec{A}$ an \'{e}tale morphism that is cohomologically pure in codimension $1$. Then $V$ is an affine scheme.
\end{theorem}
Our definition of `cohomologically pure in codimension $1$' (see Definition~\ref{def:cohpure}) could be seen as a cohomological analogue of the situation that $V$ arises 
as an open subscheme of $\spec{B}$, for a normal local ring $(B, \ideal{m}_B)$ that is finite over $A$, such that there is a factoring
$ V\to \spec{B} \setminus \{\ideal{m}_B\} \to \spec{B}$ where the first map is affine morphism of schemes. We only realized later that the authors in \cite[21.12]{EGAIV} were interested in characterizing rings $B$ for which such a $V$ is then 
forced to be an affine scheme, and consequently formulated a conjecture (iv) \cite[21.12.14]{EGAIV}.\\
\indent If their conjecture was true, 
then as noted in (v) \cite[21.12.14]{EGAIV}, the affineness of the maximal \'{e}tale locus would be equivalent to the purity of the ramification locus. 
By essentially paraphrasing the proof of purity of the ramification locus given in \cite[0ECD]{sp}, 
Theorem~\ref{thm:affine-mixed} would then follow by induction on the dimension of $Y$, where the case of dimension $2$ would crucially use 
the \textit{miracle flatness theorem} (see \cite[00R4]{sp}), from which this thesis derives its name.
\end{abstract}
\thispagestyle{empty}
\frontmatter
\tableofcontents*
\chapter{Acknowledgments}

First and foremost, I'd like to thank my advisor, Aise Johan de Jong, for his support and guidance throughout my PhD. 
Speaking with him always led to an incredibly stimulating intellectual discussion, and I am very grateful for the many hours he spent helping me develop my ideas. 
I came into this PhD with a preconceived notion of the mathematics that I liked and disliked, 
and have Johan to thank for helping me break out of that mindset and develop an appreciation for many areas of mathematics,
 because one never knows when a particular technique from a seemingly unrelated area of mathematics will be useful for solving a problem.
Johan also taught me the importance of being precise and careful with my mathematical writing, a skill which I have also 
found useful beyond mathematics. With no doubt, I will be drawing on these skills for the rest of my life, and I am very grateful to Johan for helping me develop them.\\
\indent I would also like to thank my former advisor, James Borger. Even after my Master's thesis, James continued to be a great source of support and encouragement, 
and my time under his supervision shaped my mathematical interests and approach to mathematics in many ways. Most of the material on semirings in this thesis
was inspired by collaboration with James, and I am still in awe of James's insights and creativity in developing beautiful mathematical theory.\\
\indent Next, I would like to thank the whole Columbia Math Department support staff for their incredible work in keeping the department running smoothly. 
In particular, I would like to thank Nathan Schweer for his help with many administrative issues I faced during my PhD, as well as his support and encouragement.\\
\indent I would also like to thank Juan Esteban Rodríguez Camargo and Hanlin Cai for many helpful discussions. 
Juan was very generous in helping me understand the literature on $\infty$-categories, descendability, and with proving Proposition~\ref{proposition:symformula}. 
Hanlin was very helpful in discussing ideas concerning Chapter~\ref{chapter:affine}. In particular, Hanlin frequently suggested
that considering the absolute integral closure of a Noetherian local ring may be useful, and he was the first to rephrase the equicharacteristic 
proof of Theorem~\ref{thm:affine-mixed} in terms of the Riemann-Hilbert functor of Bhatt and Lurie. One of the many threads of our discussion 
was the `potential strategy' to prove Theorem~\ref{thm:affine-mixed} by tilting to equicharacteristic outlined in the introduction of Chapter~\ref{chapter:affine}. 
It would've been incredibly satisfying if these arguments had worked, but unfortunately we have not been able to do so.\\
\indent Last, but certainly not least, I'd like to thank my fiancée Liberty for her love and support throughout my whole PhD. Life in New York City
is both wonderful and challenging, and I am very grateful to have had Liberty by my side to share in the wonderful moments and to support me through the challenging ones.\\
\indent My family, especially my parents, have also been a great source of support and encouragement throughout my life, and I couldn't have 
made it this far without them. Thank you to my sister Eve for coming up with the name `indivisible sequences'.\\
\indent Thank you to the Fulbright program for providing me with the funding to pursue my PhD at Columbia University.
\mainmatter
\chapter{Semirings}

This chapter is dedicated to the study of semirings, which are algebraic structures that generalize rings by 
allowing for the absence of additive inverses. While semirings have been studied as early as the 1950s, 
the formal foundations for the algebraic geometry of semirings was largely developed only recently in \cite{borgersemi}.
The algebraic geometry for semirings presents many new challenges. For example, the notion of closed subschemes
is more subtle (see Remark~\ref{rem:badclosed} and Example~\ref{exmp:badclosed}).

\subsubsection*{Notation and conventions} The relevant module-theoretic framework for semirings has been established in \cite{borgersemi}.  
Let $\bb{N}$ denote the set of natural numbers including $0$ endowed with the usual addition and multiplication operators thus making
it a semiring. Then $\mod{\bb{N}}$ is the category of commutative monoids, and $\alg{\bb{N}}$ is the category of semirings. The usual
abstract module theory endows for any semiring $A$ its category $\mod{A}$ of $A$-modules.
Note that $\mod{A}$ is a symmetric monoidal category with tensor product $\otimes_A$ and unit object $A$, 
and furthermore, for any two modules $M,N \in \mod{A}$, $\Hom_{\bb{N}}(M,N)$ is an internal hom for $(\mod{A},\otimes_{A})$.\\
\indent For a semiring $A$, a module $M \in \mod{A}$ is said to be \textit{free} if it is isomorphic to a direct sum of copies of $A$.
A module $M$ is said to be \textit{flat} if the functor $-\otimes_A M:\mod{A} \to \mod{A}$ preserves finite limits. Later, it will be 
also useful to refer to this notion of flatness as \textit{categorical flatness}.

\begin{example}
Here are some examples of semirings that will serve as coefficient rings for our semiring schemes:
\begin{itemize}
\item $\bb{R}_{+}$, the semiring of non-negative real numbers with the usual addition and multiplication.
\item The Boolean semiring $\bb{B}$ is defined as the set $\{0,1\}$ with addition and multiplication defined as follows:
\[0+0=0, \quad 0+1=1, \quad 1+1=1\]
and
\[0\cdot 0=0, \quad 0\cdot 1=0, \quad 1\cdot 1=1.\]
\end{itemize}
\end{example}
For a natural number $n \in \bb{N}$ greater than $1$, let $[n]=\{1,2,...,n\}$.

\section{Preliminary results and motivation}
The semirings $\bb{N}$ and $\bb{R}_{+}$ naturally encode a notion of \textit{positivity} that flat modules over them inherit. 
We make this precise in the following definitions and propositions.
\begin{definition}
Let $M \in \mod{\bb{N}}$ be a commutative monoid. $M$ is said to be:
\begin{enumerate}
\item[(i)] \textit{Cancellative} if for all $m,m',m'' \in M$, $m+m'=m+m''$ implies $m'=m''$.
\item[(ii)] \textit{Negative-free} if for all $m,m' \in M$, $m+m'=0$ implies $m=0$ and $m'=0$.
\item[(iii)] \textit{Integral} if it is both cancellative and negative-free. 
\end{enumerate} 
\end{definition}
\begin{proposition}\label{prop:flatpositivity}
Let $A$ be a semiring, and $M$ a flat $A$-module. If $A$ is cancellative (resp. negative-free, integral), then so is $M$.
\end{proposition}
\begin{proof}
A commutative monoid $N$ is negative free if and only if:
\[0=\text{eq}(+,0: N\oplus N \to N)\]
where $+$ is the sum-map and $0$ is the zero map. Similarly, $N$ is cancellative if and only if the map
\[\text{id} \oplus \mathrm{diag}: N \oplus N \to N \oplus N \oplus N\]
taking an element $n\oplus n' \mapsto n \oplus n' \oplus n'$ realizes an isomorphism
\[N\oplus N = \text{eq}(f,g: N \oplus N \oplus N \to N)\]
where $f,g$ are defined as follows. For $n,n',n'' \in N$:
\[f(n \oplus n' \oplus n'') = (n+n', n'')\]
\[g(n \oplus n' \oplus n'') = (n+n'',n')\]
All propositions follow by considering the relevant equalizer diagram for $A$, which, 
after applying $\otimes_A M$, remains an equalizer diagram.
\end{proof}
\begin{corollary}\label{cor:NRnotflat}
The semiring maps $\bb{N} \to \bb{Z}$ and $\bb{R}_{+} \to \bb{R}$ are not flat.
\end{corollary}
Integral commutative monoids admit a natural partial order defined as follows:
\begin{definition}
Let $M$ be a commutative monoid. For $m,m' \in M$, we say that $m \le m'$ if there exists a $m'' \in M$ such that
\[m + m'' = m'.\] 
For an integral commutative monoid we will denote this comparison relation by $\le_M$.
\end{definition}
\begin{proposition}\label{prop:poset}
Let $M$ be an integral commutative monoid. Then $\le_M$ has the following properties:
\begin{enumerate}
\item[(i)] (Reflexivity) For all $m \in M$, $m \le_M m$.
\item[(ii)] (Transitivity) For all $m,m',m'' \in M$, if $m \le_M m'$ and $m' \le_M m''$, then $m \le_M m''$.
\item[(iii)] (Antisymmetry) For all $m,m' \in M$, if $m \le_M m'$ and $m' \le_M m$, then $m=m'$.
\item[(iv)] (Compatibility with addition) For all $m_1,m_2,m_1',m_2' \in M$, if $m_1 \le_M m_1'$ and $m_2 \le_M m_2'$, 
then $m_1 + m_2 \le_M m_1' + m_2'$.
\item[(v)] (Compatibility with scalar multiplication) For all $a \in A$ and $m,m' \in M$, if $m \le_M m'$, then $a m \le_M a m'$.
\item[(vi)] For all $m,m',m'' \in M$, if $m + m'' \le_M m' + m''$, then $m \le_M m'$.
\end{enumerate}
Therefore, $(M, \le_M)$ defines a partially ordered commutative monoid.
\end{proposition}
\begin{proof}
Out of the properties (i)-(v), only (iii) requires proof. Assume that $m \le_M m'$ and $m' \le_M m$. 
Then there exists $m_1,m_2 \in M$ such that $m= m' + m_1$ and $m' = m + m_2$. Thus $m=m+m_1+m_2$. 
Therefore $m_1+m_2=0 \implies m_1=m_2=0$, as desired.\\
\indent For (vi), assume that $m + m'' \le_M m' + m''$. Then there exists $m''' \in M$ such that $m + m'' + m''' = m' + m''$. 
By cancellation, we have $m + m''' = m'$, as desired, so $m \le_M m'$.
\end{proof}
On the other hand, categorical flatness is strong enough to tell us something about the existence of roots of polynomials.
\begin{proposition}
Let $\bb{R}_{+} \to A_{+}$ be a finite, flat map of semirings such that $A_+$ is reduced. Then $A_+$ has an $\bb{R}$-point, i.e. 
a map $A_{+} \to \bb{R}$ in $\alg{\bb{R}_{+}}$. 
\end{proposition}
\begin{proof}
It suffices to show that the groupification $A=A_{+} \otimes_{\bb{R}_{+}} \bb{R}$ admits a map to $\bb{R}$. Since $A_+$ is flat over $\bb{R}_{+}$,
$A_+ \to A$ is injective, and since $A_{+}$ is also reduced and finite over $\bb{R}_{+}$, we see that $A$ is finite and reduced itself, and is therefore 
a finite product of finite field extensions of $\bb{R}$, hence isomorphic to a finite product of copies of $\bb{R}$ and $\bb{C}$. Thus, in order to construct a map $A_{+} \to \bb{R}$, 
it suffices to construct a map $A \to \bb{R}$, and for that, it suffices to show that $A$ has at least one factor isomorphic to $\bb{R}$.
Assume for the sake of contradiction that $A$ has no $\bb{R}$-factor. Then $A$ is a product of copies of $\bb{C}$.\\
\indent Let $A_+$ be generated as an $\bb{R}_{+}$-module by elements $x_1, \ldots, x_n \in A_{+}$. For any $a \in A_+$, multiplication by 
$a$ defines an $\bb{R}_{+}$-linear map
\[m_a: A_{+} \to A_{+}.\]
Hence, we may view $m_a$ as an $n \times n$ matrix with entries in $\bb{R}_{+}$. Let 
\[S = \{ s = (s_1,...,s_n) \in \bb{R}_+^{n} \mid |s|_1 =\sum_{i \in [n]} s_i = 1\}\]
be the $(n-1)$-simplex, and consider the map $f: S \to S$ defined by
\[f(s) = \frac{m_a(s)}{|m_a(s)|_1}, \; \forall s \in S.\]
By the Brouwer fixed point theorem,
there exists a fixed point $r \in S$ such that $f(r) = r$,
and so we conclude that $m_a$ admits an eigenvector $r$ with a positive real eigenvalue $\lambda = |m_a(r)|_1$ 
(one could also appeal to the Perron-Frobenius theorem).\\
\indent Since $A$ is a product of copies of $\bb{C}$, there exists an element $a \in A$ whose eigenvalues are all non-real complex numbers.
However, by the above argument, $a$ must have a positive real eigenvalue, yielding a contradiction. Hence, $A$ must have at least one factor isomorphic to $\bb{R}$.
\end{proof}
\begin{corollary}
Let $A$ be a finite reduced $\bb{R}$-algebra. If $A$ has a positive model, i.e. there exists a finite, reduced and flat $\bb{R}_{+}$-algebra $A_{+}$, 
such that $A = A_{+} \otimes_{\bb{R}_{+}} \bb{R}$, then $A$ has at least one $\bb{R}$-point.
\end{corollary}
We therefore arrive at the following question 
that is closely related to the following version of Hilbert's Nullstellensatz in real algebraic geometry.
\begin{question}\label{mainquestion}
Let $A$ be a finitely presented flat $\bb{R}_{+}$-algebra. 
If $A$ is non-trivial, does it necessarily have an $\bb{R}$-point (or even better, an $\bb{R}_{+}$-point)?\\
\indent Equivalently, let 
\[\bb{R}^n_{+}:=\bb{R}_{+}[a_1,b_1,...,a_n,b_n]/f_n \sim g_n\]
where $f_n=1+\sum_{i=1}^n a_i^2 +b_i^2$ and $g_n=\sum_{i=1}^n 2a_i b_i$. By definition, $\bb{R}^n_{+}$ is the initial 
$\bb{R}_{+}[a_1,b_1,...,a_n,b_n]$-algebra $A$ such that the $f_n$ and $g_n$ are equal in $A$.\\
\indent Is it true that for every $n \ge 1$, there are no $\bb{R}_{+}$-algebra maps
\[\bb{R}^n_{+} \to A?\]
Let us explain the equivalence of the two questions. The existence of an $\bb{R}$-point of $A$ is equivalent to asking
for an $\bb{R}$-point of $A' = A \otimes_{\bb{R}_{+}} \bb{R}$. Picking a surjection $\bb{R}[x_1,...,x_n] \twoheadrightarrow A'$
with kernel $I$, $\spec{A'}$ does not admit $\bb{R}$-point if and only if $\text{rad}_{\bb{R}}(I) \neq \bb{R}[x_1,...,x_n]$, by the Real Nullstellensatz (see Theorem~\ref{realnull}). 
Therefore, $\spec{A'}$ admits $\bb{R}$-point if and only if $1 \notin \text{rad}_{\bb{R}}(I)$. Unpacking the definitions, 
and writing every element in $A'$ as the difference of two elements of $A$, we readily obtain the equivalence.
\end{question}
The equivalence mentioned above is due to the following version of Hilbert's Nullstellensatz in real algebraic geometry:
\begin{theorem}[Real Nullstellensatz]\label{realnull}
For an ideal $I \subseteq \bb{R}[x_1,...,x_n]$, let
\[V_{\bb{R}}(I) = \{x \in \bb{R}^n \mid f(x)=0 \text{ for all } f \in I\}\]
be its real variety, and 
\[\rad_{\bb{R}}(I)=\{f \in \bb{R}[x_1,...,x_n]:\; \exists m \in \bb{N}, f_1,...,f_d \in \bb{R}[x_1,...,x_n] :\;  -f^{2m}- (f_1^2 + ... + f_d^2) \in I\}.\]
Then $\rad_{\bb{R}}(I)$ is precisely the set of polynomials in $\bb{R}[x_1,...,x_n]$ that vanish on $V_{\bb{R}}(I)$. \cite[Chapter XI]{lang_algebra_2002}
\end{theorem}
Motivated by Question \ref{mainquestion}, we set-out the following goals:
\begin{enumerate}
\item[(a)] Develop a homotopical characterization of flatness for modules over semirings.
\item[(b)] Determine whether for a general flat finitely presented map of semirings $A \to B$, 
the map on spectra $\Spec{B} \to \Spec{A}$ is open.
\item[(c)] Determine when the equivalence relation generated by $f \sim g$ defines a flat map of semirings $\bb{R}_{+} \to \bb{R}_{+}[x_1,...,x_n]/(f \sim g)$.
\end{enumerate}
We will study these goals in the rest of the chapter.

\section{Homotopy theory of commutative monoids}
\indent Let $\Delta$ be the ordinary simplex category consisting of objects finite sets together with the usual set-maps.
For any category $\scr{C}$, we let $\scr{C}_{\Delta}$ denote the category consisting of objects being functors $\Delta^{\text{op}} \to \scr{C}$ and morphisms being 
natural transformations of functors; for any such functor $F: \Delta^{\text{op}} \to \scr{C}$, we call $F(n)$ the \textit{$n$-simplices} of $F$. For every object $C \in \scr{C}$, let $\underline{C} \in \scr{C}_{\Delta}$ be the constant pre-sheaf on $\Delta$ with value $C$; we will sometimes abuse notation and refer to such constant simplicial object $\underline{C}$ as simply $C$.\\
\indent We have a forgetful functor
\[U:\bb{N}\mathrm{-mod} \to \text{Set},\]
admitting a right adjoint
\[\bb{N}[-]:\mathrm{Set} \to \bb{N}\mathrm{-mod}\] 
assigning to every set $S \in \text{Set}$ the free commutative monoid $\bb{N}^{\oplus S}$. \\

\indent We also recall the Dold-Kan functor:
\[\DK:\text{Ch}(\mathrm{Ab}) \to \mathrm{Ab}_{\Delta}\]
As verified in \cite[III.2]{simpho}, the functor $\DK$ doesn't involve signs, so therefore descends to a well-defined functor
\[\DK:\text{Ch}(\bb{N}\mathrm{-mod}) \to \simpmod{\bb{N}}.\]
\begin{proposition}\label{prop:simpN}
Let $M,N \in \bb{N}\mathrm{-mod}_{\Delta}$ be simplicial commutative monoids, $S \in \mathrm{Set}_{\Delta}$ be a simplicial set, and $\Delta_{n}\in \mathrm{Set}_{\Delta}$ the simplicial set represented by $[n] \in \Delta$.
\begin{enumerate}
\item[(i)] The adjoint pair $(F \dashv U)$ extends to an adjoint pair:
\[\begin{tikzcd} \bb{N}\mathrm{-mod}_{\Delta} \arrow[r, bend right, "U"] & \text{Set}_{\Delta} \arrow[l, bend right, "F"]\end{tikzcd}\]
where $\bb{N}[S](n)=\bb{N}[S(n)]$.
\item[(ii)] The category $\bb{N}\mathrm{-mod}_{\Delta}$ may be endowed with a symmetric monoidal structure $\otimes_{\bb{N}}$, explicitly given by the following formula:
\[(M \otimes_{\bb{N}} N)(n) = M(n) \otimes_{\bb{N}} N(n).\]
There is an internal hom $\underline{\Hom}_{\bb{N}}$ given by the formula:
\[\underline{\Hom}_{\bb{N}}(M,N)(n) = \Hom_{\bb{N}\mathrm{-mod}_{\Delta}}(M \otimes_{\bb{N}}\bb{N}[\Delta_n],N).\]
 \end{enumerate}
 \end{proposition}
 \begin{proof}
 The proof of (i) is a standard fact for functor categories. For (ii), this follows from \cite[Tag017H]{sp}, after noting that $M \times \Delta_n$ is the same as $M \otimes_{\bb{N}}\bb{N}[\Delta_n]$ in our set-up.
 \end{proof}

We now wish to endow $\bb{N}\mathrm{-mod}_{\Delta}$ with a model structure in order to derive the tensor product $\otimes_{\bb{N}}$. 
Motived by the model structure on chain complexes that works so well, 
there is a natural such model structure, in hindsight perhaps suitably called the \textit{naive model structure}.
First we recall the \textit{homotopy-theoretic model structure} on $\text{Set}_{\Delta}$:
\begin{definition}[The homotopy-theoretic model structure on $\text{Set}_{\Delta}$]
A map $f: X \to Y$ in $\text{Set}_{\Delta}$ is a fibration if and only if it has the right lifting property with respect to all horn inclusions 
$\Lambda^n_k \to \Delta_n$ for all $n \ge 1$ and $0 \le k \le n$.\\
\indent The map $f: X \to Y$ in $\text{Set}_{\Delta}$ is a weak equivalence if and only if the induced map on geometric realizations
$|f|: |X| \to |Y|$ is a weak homotopy equivalence of topological spaces.\\
\indent The map $f:X \to Y$ is a cofibration if and only if it has the left lifting property 
with respect to all acyclic fibrations (i.e. fibrations that are weak equivalences).\\
\indent The map $f: X \to Y$ is an acyclic fibration if and only if it has the right lifting property with respect to all boundary inclusions
$\partial \Delta_n \to \Delta_n$ for all $n \ge 0$. \cite[I.11]{simpho}
\end{definition}
With this model structure on $\text{Set}_{\Delta}$, we may now define the \textit{naive model structure} on $\bb{N}\mathrm{-mod}_{\Delta}$.
\begin{definition}[The naive model structure]A map $f: X \to Y$ in $\bb{N}\mathrm{-mod}_{\Delta}$ is a fibration (resp. weak equivalence) 
if and only if $U(f)$ is a 
fibration (resp. weak equivalence) in $\text{Set}_{\Delta}$. 
Cofibrations are characterized to have the left lifting property with respect to all acyclic fibrations 
(i.e. fibrations that are weak equivalences).\\
\indent These three classes of 
morphisms define a model structure on $\bb{N}\mathrm{-mod}_{\Delta}$.
\end{definition}
This model structured has been considered many times in the literature, two notable references being \cite{kansemi} and \cite{rezksemi}.
Some immediate consequences of this definition are as follows.
\begin{proposition}\label{prop:hosimpN}
Let $X \in \simpmod{\bb{N}}$ be a simplicial commutative monoid.
\begin{enumerate}
\item[(i)] The adjoint functors $(F \dashv U)$ of Proposition~\ref{prop:simpN} form a Quillen adjunction for the naive model 
structure on $\simpmod{\bb{N}}$ and the homotopy-theoretic model structure on $\text{Set}_{\Delta}$. 
\item[(ii)] The adjoint functors $((X \otimes_{\bb{N}} -) \dashv \underline{\Hom}_{\bb{N}}(X,-)) : \simpmod{\bb{N}} \to \simpmod{\bb{N}}$ 
form a Quillen adjunction.
\item[(iii)] The category $\simpmod{\bb{N}}$ equipped with the naive model structure is a simplicial model category. 
In particular, if $Y \to Z$ is a cofibration in $\simpmod{\bb{N}}$, then so is $X \otimes_{\bb{N}} Y \to X \otimes_{\bb{N}} Z$.
\end{enumerate}
\end{proposition}
\begin{proof}
Both (i) and (ii) follow respectively from the fact that both $U$ and $\underline{\Hom}_{\bb{N}}(X,-)$ preserve both fibrations and acyclic fibrations. For $U$, 
this is immediate, and for $\underline{\Hom}_{\bb{N}}(X,-)$, this follows from the fact that it is a right adjoint, and the fact that fibrations and acyclic fibrations in $\simpmod{\bb{N}}$ are defined
via right lifting properties for the free $\bb{N}$-linear horn and boundary inclusions respectively (i.e. $\bb{N}[\Lambda^i_n] \hookrightarrow \bb{N}[\Delta_n]$ 
and $\bb{N}[\partial \Delta_n] \hookrightarrow \bb{N}[\Delta_n]$)\\
\indent For (iii), by \cite[Theorem II.5.1]{simpho} combined with \cite[Lemma II.6.1]{simpho}, it is enough to show that for any $X \in \simpmod{\bb{N}}$ 
we have a weak equivalence $X \to Q(X)$ with $Q(X) \in \simpmod{\bb{N}}$ and $U(Q(X))$ a Kan complex. This is provided to us by Kan's Extension Functor
that, to any simplicial set $S$, functorially associates a weak equivalence $S \mapsto \text{Ex}^{\infty}(S)$ where $\text{Ex}^{\infty}(S)$ a Kan complex.
\end{proof}
\begin{remark}
We may define the $\infty$-category $D_{\le 0}(\bb{N}\mathrm{-mod})$ to be the localization of simplicially-enriched category $\simpmod{\bb{N}}$ 
with respect to the weak equivalences defined by the naive model structure. Then $D_{\le 0}(\bb{N})$ is, up to 
equivalence of $\infty$-categories, the same as the \textit{animation} of $\bb{N}$-mod (see \cite[5.5.9.3]{htt}).
\end{remark}\label{rem:comptop}
With a model structure on $\simpmod{\bb{N}}$, we may proceed to calculate derived functors. In this paper, we will only be interested in doing so for $\otimes_{\bb{N}}$.
\begin{definition}
Let $N$ be an object of $\simpmod{\bb{N}}$. For any $M \in \simpmod{\bb{N}}$, we may take a cofibrant replacement $r: C(M) \to M$, and we define:
\[M \otimesl_{\bb{N}} N = C(M) \otimes_{\bb{N}} N,\]
as objects in $\Hot(\simpmod{\bb{N}})$.
\end{definition}
The naive model structure on $\simpmod{\bb{N}}$ is \textit{reasonable} in the following sense; by viewing $\bb{Z}\mathrm{-mod}_{\Delta} \hookrightarrow \bb{N}\mathrm{-mod}_{\Delta}$ 
as a full subcategory, then the induced model structure on $\simpmod{\bb{Z}}$ coincides with the usual model structure for chain complexes 
of abelian groups after applying Dold-Kan. However, there are quite a few conceptual differences between the two model-theoretic structures,
chief among them being the fact that objects in $\simpmod{\bb{N}}$ are rarely fibrant (i.e. Kan complexes) or even $\infty$-categories:
\begin{example}
\begin{enumerate}
\item[(a)] Objects of $\simpmod{\bb{N}}$ are rarely fibrant, or in other words, the underlying simplicial set is not a Kan complex.\\
\indent Indeed, while simple examples like $K(\bb{N},0) = \underline{\bb{N}}$, are Kan complexes, the next Eilenberg-Maclean space $K(\bb{N},1)$ is not \textemdash let us recall its definition. 
Consider the category $B\bb{N}$, which has one object and an endomorphism $1$ which acts freely; we define $K(\bb{N},1)=N(B\bb{N})$ where $N(B\bb{N})$ is its nerve. 
To show it is not a Kan complex, consider the map:
  \[\Lambda^2_{0} \to K(\bb{N},1)\]
  which takes the edge $[0,1] \to n$ and $[0,2] \to m$. If $K(\bb{N},1)$ were a Kan complex, then there would need to be a 
  number $h$ such that $n + h = m$, which is not solvable as soon as $n > m$.\\
  \indent Note that, while $K(\bb{N},1)$ is not a Kan complex, it is on the other hand an $\infty$-category, as is the nerve of any 
  ordinary category (see \cite[1.3]{keradon}).
\item[(b)] In fact, objects of $\simpmod{\bb{N}}$ are rarely $\infty$-categories. Indeed, we may construct a simplicial set 
modelling $K(\bb{N},2)$ which is not a quasi-category (see \cite{notquasi}). As $K(\bb{N},2)$ (which is the diagonal of the 
bisimplicial set $B^2 \bb{N}$) is modelled by an object in $\simpmod{\bb{N}}$ where each term is a finite free 
commutative monoid $\bb{N}^{\oplus r}$, 
from a homological standpoint, this example is essentially as nice as possible. 
We construct another example in Example~\ref{exmp:notquasi}.
\item[(c)] A surjection $X \twoheadrightarrow Y$ in $\simpmod{\bb{N}}$ need not be a fibration, even if it is termwise split on each degree.
\end{enumerate}
\end{example}
Motivated by trying to prove Question~\ref{mainquestion}, we wish to understand the structure of homotopy coequalisers in $\simpmod{\bb{N}}$. 
The structure of these equalizers is rather different from the case of abelian groups; in particular, we cannot show that they are always $1$-truncated. 
Calculating homotopy coequalizers amounts to understanding homotopy pushouts, which is greatly simplified in the case when the model category
is \textit{left proper} i.e. pushouts along cofibrations preserve weak equivalences. Fortunately, the naive model structure on $\simpmod{\bb{N}}$ is left proper. To show this, we could appeal to 
\cite[Theorem 8.1]{rezksemi}, which shows that it suffices to demonstrate that the functor 
\[ \simpmod{\bb{N}} \to \simpmod{\bb{N}}: \; \; M \mapsto M \oplus \underline{\bb{N}}\]
preserves weak equivalences, which can be done by hand. In particular, we note that for any $M \in \simpmod{\bb{N}}$ and $N \in \mod{\bb{N}}$,
\[M \oplus \underline{N} =\colim_{S' \subset U(N), |S'| < \infty} \prod_{S'} M = \hocolim_{S' \subset U(N), |S'| < \infty} \prod_{S'} M.\]
is a filtered colimit with cofibrant transition maps (i.e. monomorphisms in $\text{Set}_{\Delta}$) of products of copies of $M$, hence it preserves weak equivalences.
However, for completeness, we will give a more direct proof that follows the outline of Kan and Dwyer \cite{kansemi} 
(which unfortunately cannot be directly applied here as their categories do not result in commutative algebraic theories.)
\begin{proposition}
Let $M, M', N, N', K, K'$ be objects of $\simpmod{\bb{N}}$. 
\begin{enumerate}
\item[(i)]Let $f: M \to M'$ and $f': N \to N'$ be weak equivalences in $\simpmod{\bb{N}}$. 
Then the induced map
\[f \amalg f': M \oplus N \to M' \oplus N'\]
is also a weak equivalence.
\item[(ii)] Consider the pushout squares in $\simpmod{\bb{N}}$:
\[\begin{tikzcd}
M \arrow[r] \arrow[d] 
  & N \arrow[d] 
  && 
M' \arrow[r] \arrow[d] 
  & N' \arrow[d] \\
K \arrow[r] 
  & L
  \arrow[ul, phantom, "\ulcorner", very near start]
  &&
K' \arrow[r] 
  & L'
  \arrow[ul, phantom, "\ulcorner", very near start]
\end{tikzcd}\]
with weak equivalences $M \to M', N\to N', K\to K'$ making the diagrams commute, and moreover assume $M \to K$ and $M' \to K'$ are cofibrations.
Then the induced map $L \to L'$ is also a weak equivalence.
\item[(iii)]$\simpmod{\bb{N}}$ is left proper i.e. pushouts along cofibrations preserve weak equivalences.
\end{enumerate}
\end{proposition}
\begin{proof}
For (i), it suffices to prove the case that $M=M'$. We have already established the case when $M$ is discrete.
We may then proceed using the diagonal lemma for bisimplicial sets, as by definition, $M \oplus N = \mathrm{diag}(L_{\bullet, \bullet})$ where
\[L(i,j) = M(i) \oplus N(j).\]
The result then follows from the fact that for each fixed $i$, the map $M(i) \oplus N \to M(i) \oplus N'$ is a weak equivalence, so the diagonals
are weak equivalences.\\
\indent For (ii), we can follow \cite[Proposition 8.1]{kansemi} with the follow modifications. First, their free-product $*$ serves the same purpose as
our direct sum $\oplus$ in $\simpmod{\bb{N}}$. Their Lemma 2.7 is then replaced by our (i) and for their Proposition 8.2, we need
a characterization of cofibrations in $\simpmod{\bb{N}}$ similar to that of their 7.6. After noting that their free categories $F$ serve
the same functorial purpose of our $\bb{N}[-]$, then as $\bb{N}[\partial \Delta_n] \hookrightarrow \bb{N}[\Delta_n], \forall n \ge 0$ form a set of generating cofibrations,
their 7.6 also holds in our setting. The proof then goes through verbatim.\\
\indent Finally, (iii) follows from (ii) since any pushout diagram can be made into a cofibrant diagram which preserves weak equivalences.
\end{proof}
The following theorem describes our current understanding of the structure of homotopy coequalizers in $\simpmod{\bb{N}}$.
\begin{theorem}\label{thm:main0}
Let $f,g: M \to N$ be two maps in $\bb{N}\mathrm{-mod}$. There is a $Z(f,g) \in \simpmod{\bb{N}}$ such that $Z(f,g) \simeq \hocoeq(f,g)$ and:
\begin{enumerate}
\item[(i)] $Z(f,g)$ is a $(2,1)$-category that is weakly $1$-coskeletal and minimal in dimension $1$.
\footnote{In particular, $Z(f,g)$ as a simplicial set is isomorphic to the Duskin nerve of a $2$-category in which every 
$2$ morphism is an isomorphism, see \cite[4.8.1.5]{keradon}}
\item[(ii)] If $Z(f,g)$ is an $\infty$-category, then $Z(f,g)$ is further a $(1,1)$-category and $Z(f,g)=N(\Hot(Z(f,g)))$. For example, this is always true when $g=0$.
\item[(iii)] If $g=0$, $N$ is cancellative and $M$ is negative-free, then $Z(f,g)$ is discrete and $\hocoeq(f,g)=\coeq(f,g)$.
\end{enumerate}
\end{theorem}
\begin{remark}
It is worth comparing Theorem~\ref{thm:main0} with the usual structure in abelian groups. 
Indeed, the homotopy coequaliser of two maps $f,g: M \to N$ is quasi-isomorphic to the $\hocoker(f-g)$, 
which is concentrated in degrees $0$ and $1$. 
Our theorem recovers this fact. 
In particular, (i) says that $\hocoeq(f,g)=Z(f,g)$ is a $(2,1)$ category, and we know that $Z(f,g)$ is a Kan complex, so (ii) says that $Z(f,g)$ 
is in fact a $(1,1)$ category. Therefore, $Z(f,g)$ is a $1$-groupoid, and in particular, is $1$-truncated (see \cite[3.5]{keradon}).
\end{remark}
To prove Theorem~\ref{thm:main0}, we need to recall the following definition:
\begin{definition}\label{def:n1cat}
Let $n$ be a positive integer. A simplicial set $S$ is an $(n,1)$-category if it satisfies the following condition for every pair of integers $0 < i < m$: 
the restriction map 
\[\Hom_{\text{Set}_{\Delta}}(\Delta^m, S) \to \Hom_{\text{Set}_{\Delta}}(\Lambda_i^m, S)\]
is surjective, and if $m > n$, is also injective. A simplicial set $S$ is weakly $n$-coskeletal if the restriction map
\[\Hom_{\text{Set}_{\Delta}}(\Delta^m, S) \to \Hom_{\text{Set}_{\Delta}}(\partial \Delta^m, S)\]
is bijective for $m \ge n+2$ and an injection for $m=n+1$.
\end{definition}
\begin{proposition}\label{prop:11cat}
A simplicial set $S$ is isomorphic to the nerve of a category $\cal{C}_0$ if and only if it is a $(1,1)$-category. 
In particular, $S=N(\Hot(S))$.
\end{proposition}
\begin{proof}
See \cite[Proposition 1.3.4.1]{keradon}.
\end{proof}
\begin{proposition}\label{prop:coeqN}
Let $f,g: X \to Y$ be two morphisms in $\simpmod{\bb{N}}$, and let $Z(f,g)$ fit into a pushout diagram in $\simpmod{\bb{N}}$:
\[\begin{tikzcd}
X \amalg X \arrow[d, "\text{id}_X \otimes {\bb{N}[\iota]}"] \arrow[r, "(f\text{,}g)"] & Y \arrow[d]\\
X \otimes_{\bb{N}} {\bb{N}[\Delta^1]}\arrow[r] & Z(f,g)
\end{tikzcd}\]
where $\iota: \partial \Delta^0 \hookrightarrow \Delta^1$ is the canonical inclusion as a sub-simplicial set. Then $Z(f,g) \simeq \hocoeq(f,g:X \to Y)$.
\end{proposition}
\begin{proof}
Let $Z'=X \coprod^{\text{h}}_{\nabla,X \amalg X, (f,g)} Y$, note that $Z' \simeq \hocoeq(f,g:X \to Y)$. To show the claim, 
it suffices to show that $\text{id}_X \otimes \bb{N}[\iota]: X\amalg X \to X \otimes_{\bb{N}} \bb{N}[\Delta^1] $ is a cofibrant replacement for $\nabla: X \amalg X \to X$. This will follow from the following two claims:
\begin{enumerate}
\item[(i)] The canonical map $X \otimes_{\bb{N}} \bb{N}[\Delta^1] \to X$ induced by $\Delta^1 \to *$ is a homotopy equivalence.
\item[(ii)] The map $\text{id}_X \otimes \bb{N}[\iota]: X\amalg X \to X \otimes_{\bb{N}} \bb{N}[\Delta^1]$ is a cofibration.
\end{enumerate}
Note that (ii) follows from (i) and (iii) of Proposition~\ref{prop:hosimpN}, so it suffices to show (i). 
But this can be shown after applying $U$, so it suffices to prove a weak equivalence between the simplicial sets 
$\underline{\Hom}_{\text{Set}_{\Delta}}(\Delta^1, X)$ and $\underline{\Hom}_{\text{Set}_{\Delta}}(\Delta^0, X)$, which is clear as $\Delta^1 \to \Delta^0$ is a homotopy equivalence.
\end{proof}
\begin{example}\label{exmp:coeqN}
Consider the situation of Proposition~\ref{prop:coeqN} specialized to the case where $X=\underline{M}$ and $Y=\underline{N}$, where $M$ and $N$ are $\bb{N}$-modules. 
Then $Z(f,g)$, as a functor $\Delta^{op} \to \bb{N}$ (ignoring degeneracy maps), may be described as follows: $\forall (m_1,...,m_n) \in M^{\oplus n}, k \in N$
\begin{align*}
Z(f,g)([n]) &= M^{\oplus n} \oplus N, \; \; \forall [n] \in \ob(\Delta^{op}),\\
Z(f,g)(\partial_i)(m_1,...,m_n,k) &= (m_1,...,m_{i-1}, m_i + m_{i+1},...,m_n, k), \; \;  \forall 0 < i < n,\\
Z(f,g)(\partial_0) (m_1,...,m_n,k) &= (m_2,...,m_n, f(m_1)+k),\\
Z(f,g)(\partial_n) (m_1,...,m_n,k) &= (m_1,...,m_{n-1}, g(m_n)+k).
\end{align*}
$Z(f,g)$ has some nice features, as the next proposition shows.
\end{example}
\begin{proposition}\label{prop:weakskel}
Let $Z(f,g)$ be the simplicial set described in Example~\ref{exmp:coeqN}, i.e. $\forall (m_1,...,m_n) \in M^{\oplus n}, k \in N$
\begin{align*}
Z(f,g)([n]) &= M^{\oplus n} \oplus N, \; \; \forall [n] \in \ob(\Delta^{op}),\\
Z(f,g)(\partial_i)(m_1,...,m_n,k) &= (m_1,...,m_{i-1}, m_i + m_{i+1},...,m_n, k), \; \;  \forall 0 < i < n,\\
Z(f,g)(\partial_0) (m_1,...,m_n,k) &= (m_2,...,m_n, f(m_1)+k),\\
Z(f,g)(\partial_n) (m_1,...,m_n,k) &= (m_1,...,m_{n-1}, g(m_n)+k),
\end{align*}
where $M, N \in \bb{N}\mathrm{-mod}$ and $f,g: M \to N$ are morphisms in $\bb{N}\mathrm{-mod}$.\\
\indent Then for any $m \ge 3$ and $0 < i < m$, the restriction maps
\[a^m_i: \Hom_{\simpmod{\bb{N}}}(\Delta^m, Z(f,g)) \to \Hom_{\simpmod{\bb{N}}}(\Lambda_i^m, Z(f,g))\]
\[b^m: \Hom_{\simpmod{\bb{N}}}(\Delta^m, Z(f,g)) \to \Hom_{\simpmod{\bb{N}}}(\partial \Delta^m, Z(f,g))\]
are bijective, and injective for $m=2$.\\
\indent If $g=0$, then we also have bijectivity for $a^2_1$, so $Z(f,0) \simeq N(\Hot(Z(f,0)))$, 
and in particular, $Z(f,0)$ is an $\infty$-category. $\mathrm{Hot}(Z(f,0))$ can be described as follows:
\begin{itemize}
  \item $\ob(\Hot(Z(f,0))) = \{n| n'\in N\}$
  \item $\Mor(\Hot(Z(f,0)))(n,n') = \{m \in M | n'=n+f(m)\}$.
  \item The composition rule 
  \[\Mor(\Hot(Z(f,0)))(n',n'') \times \Mor(\Hot(Z(f,0)))(n,n') \to \Mor(\Hot(Z(f,0)))(n,n'')\]
  is described via addition on the commutative monoid $M$.
\end{itemize}
\end{proposition}
\begin{proof}
Let us first assume $m\ge 3$. We note that any $x=(m_1,...,m_m, n) \in Z(f,g)([m])$ is determined by 
\begin{align*}
\partial_0 x &= (m_2,...,m_m, f(m_1)+n)\\
\partial_n x &= (m_1,...,m_{m-1}, g(m_m)+n)\\
\partial_j x &= (m_1,...,m_{j-1}, m_j + m_{j+1},...,m_m, n)
\end{align*}
for any $0 < j < m$ and $j \neq i$. Thus, injectivity is clear, and for surjectivity, forming a candidate for $x$ given 
$\partial_j x$ for $j \neq i$ is straightforward by these equations, and the fact that it is indeed a lift to $Z(f,g)([m])$ of the boundaries
follows from the simplicial identities in both situations. \\
\indent Now, let us assume $m=2$. For injectivity, let $x=(m_1,m_2,n) \in M_{\bb{Z}}^{\oplus 2} \oplus N_{\bb{Z}}$, then:
\[\partial_0 x = (m_2,f(m_1)+n), \partial_2 x  = (m_1, n),\]
so indeed, $x$ is determined by $\partial_0 x, \partial_2 x$ and thus injectivity for $m=2$.\\
\indent If $g=0$, then the simplicial identities imply surjectivity for $a^2_1$. To be explicit, given $(m_1,n_1), (m_2,n_2)$ 
with $f(m_2)+n_2=n_1$ forming an element $\Lambda^2_1 \in Z(f,g)([1])$, then $x=(m_2,m_1,n_2)$ is an element in $Z(f,g)([2])$ with
$\delta_0x=(m_1,n_1)$ and $\delta_2x=(m_2,n_2)$.\\
\indent Therefore, $Z(f,0)$ is a $(1,1)$-category (see Definition~\ref{def:n1cat}), and thus $Z(f,0) \simeq N(\Hot(Z(f,0)))$ by Proposition~\ref{prop:11cat}. 
The description of $\Hot(Z(f,0))$ is then straightforward to verify.
\end{proof}
\begin{remark}\label{rem:mindim1}
We note that the simplicial set $Z(f,g)$ of Proposition~\ref{prop:weakskel} is also minimal in dimension $1$ (see \cite[Definition 4.7.6.4]{keradon}). 
This amounts to the following verification: given $x,x' \in Z(f,g)([1])$, if there exists simplices $y,y' \in Z(f,g)([2])$ such that
\[\partial_0 y =x, \partial_2 y = s(\partial_1 x), \; \partial_2 y' =x', \partial_0 y' = s(\partial_0 x'), \; \partial_1 y = \partial_1 y',\]
then $x=x'$. In our case, we can just directly verify:
\[x=\partial_1 y = \partial_1 y' = x'.\]
\end{remark}
\begin{example}\label{exmp:notquasi}
One may ask whether all simplicial sets $Z(f,g)$ constructed in Example~\ref{exmp:coeqN} are $\infty$-categories. 
We will show that this is not the case by consider $M=N=\bb{N}$ and $f,g: \bb{N} \to \bb{N}$ given by $f(n)=2n$ and $g(n)=n$ (note that $\bb{B}= \coeq(f,g: \bb{N}\to \bb{N})$).\\
\indent In particular, we remark that there exists a morphism $n \to n' \in Z(f,g)([0])=\bb{N}$ if and only if $n \le n' \le 2n$. 
So for example, there is an inner horn $\Lambda^1_2$ with edges $ 1 \to 2 \to 3$ in $Z(f,g)([1])$, but since there is no morphism from $1 \to 3$, 
there cannot be a $2$-simplex in $Z(f,g)([2])$ filling this horn.
\end{example}
\begin{proof}[Proof of Theorem~\ref{thm:main0}]
Using Proposition~\ref{prop:coeqN}, Example~\ref{exmp:coeqN} and Proposition~\ref{prop:weakskel}, 
we find a simplicial set $Z(f,g)$ which is a $(2,1)$-category that is weakly $1$-coskeletal 
such that $Z(f,g) \simeq \hocoeq(f,g:M \to N)$.
The second part of Proposition~\ref{prop:weakskel} implies that $Z(f,g)$ is an $\infty$-category when $g=0$, 
and indeed, whenever $Z(f,g)$ is an $\infty$-category, $Z(f,g)$ is a $(1,1)$-category 
by \cite[Proposition 4.8.1.7]{keradon} and Remark~\ref{rem:mindim1}, so by \cite[Example 4.8.1.3]{keradon}, $Z(f,g) \simeq N(\Hot(Z(f,g)))$.\\
\indent Finally, when $g=0$, $N$ is cancellative and $f$ is injective, then every two objects of $\Hot(Z(f,0))$ are 
connected by a unique morphism, and if $M$ is negative-free,
then there are no non-trivial automorphisms, so $\Hot(Z(f,0))$ is a poset. There is then a 
contracting homotopy taking $N(\Hot(Z(f,0)))$ to its set of connected components, and thus 
$\hocoeq(f,0)=\coeq(f,0)$.
\end{proof}
As an immediate application of Theorem~\ref{thm:main0}, we have the following:
\begin{theorem}\label{thm:ZN}
We have 
\[\bb{Z} \simeq \hocoeq(\mathrm{diag},0: \bb{N} \to \bb{N} \oplus \bb{N}).\]
Therefore:
\[\bb{Z} \otimesl_{\bb{N}} \bb{Z} \simeq \bb{Z}[0].\]
\end{theorem}
\begin{proof}
The conditions of Theorem~\ref{thm:main0} (iii) are satisfied, so we have:
\[\hocoeq(\mathrm{diag},0: \bb{N} \to \bb{N} \oplus \bb{N}) \simeq \coeq(\mathrm{diag},0: \bb{N} \to \bb{N} \oplus \bb{N}) \simeq \bb{Z},\]
Since the functor $-\otimesl_{\mathbb{N}} \mathbb{Z}$ commutes with homotopy coequalizers:
\[\bb{Z} \otimesl_{\bb{N}} \bb{Z} \simeq \hocoeq(\mathrm{diag},0: \bb{N} \to \bb{N} \oplus \bb{N}) \otimesl_{\bb{N}} \bb{Z}
\simeq \hocoeq(\mathrm{diag},0: \bb{Z} \to \bb{Z} \oplus \bb{Z}) \simeq \bb{Z}[0],\]
as desired.
\end{proof}
It is also interesting to relate the explicit construction of $Z(f,g)$ in Example~\ref{exmp:coeqN} to other potential constructions.
\begin{corollary}\label{cor:DKcoeq}
Let $f,g: X \to Y$ be two maps in $\simpmod{\bb{N}}$. Let $\Delta^{nd}_{\le 1}$ denote the subcategory of $\Delta_{\le 1}$ 
where we do not include the degeneracy map $[1] \to [0]$. 
Consider the canonical inclusion of categories $i:\Delta^{nd,op}_{\le 1} \to \Delta^{op}$, 
and let $F:\Delta^{nd,op}_{\le 1} \to \simpmod{\bb{N}}$ be the functor representing $f,g:X \to Y$.\\
\indent If $F'$ is the homotopy left Kan extension of $F$ along $i$, then $F' \simeq \hocoeq(f,g: X \to Y)$. 
Moreover, if $g=0$ and $X=\underline{X([0])}, Y=\underline{Y([0])}$ then $F'\simeq \DK(f: X[0] \to Y[0])$.
\end{corollary}
\begin{proof}
A defining property of homotopy left Kan extensions is that 
\[\hocolim_{\Delta^{nd}_{\le 1}} F = \hocolim_{\Delta^{op}} F',\]
so the first claim follows directly. For the second claim, by Proposition~\ref{prop:coeqN} and Example~\ref{exmp:coeqN}, 
we have an explicit description of $\hocoeq(f,g: X \to Y)$ using the simplicial set $Z(f,0)$ defined there, and in this case, 
one directly verifies that as simplicial sets, $Z(f,0)$ and $\DK(f: X[0]\to Y[0])$ are the same, resolving the second claim.
\end{proof}
The following theorem is rather surprising.
\begin{theorem}\label{thm:contractab}
Let $\iota_M: M \to M \otimes_{\bb{N}} \bb{Z}$ be the canonical map of a commutative monoid $M \in \mod{\bb{N}}$ 
to its groupification. Then $\hocoeq(\iota_M,0: M \to M \otimes_{\bb{N}} \bb{Z})$ is contractible.
\end{theorem}
\begin{proof}
By Theorem~\ref{thm:main0}, we know that $\hocoeq(\iota_M,0: M \to M \otimes_{\bb{N}} \bb{Z}) \simeq N(\scr{C})$ 
where $\scr{C}$ with objects $M \otimes_{\bb{N}} \bb{Z}$ and morphisms $n \mapsto n + i(m)$ for $m \in M$ (see Proposition~\ref{prop:weakskel}). 
We claim that $\scr{C}$ is a filtered category,
which will be enough to prove the desired contractibility of $N(\scr{C})$ (see \cite[Proposition 5.3.1.18]{htt}). We note the following trivial observation: 
for any $m \in M \otimes_{\bb{N}} \bb{Z}$, either $m$ or $-m \in \im(\iota_M)$.\\
\indent First, we show that for any two arrows $f,g: X \to Y \in \scr{C}$, there is a $Z \in \scr{C}$ and a map $h:Y \to Z$ such that $hf=hg$.\\
\indent The two maps $f,g$ may be represented as maps $n \to n + \iota_M(m)$ and $n \to n+\iota_M(m')$, 
where $n=X \in M \otimes_{\bb{N}} \bb{Z}$ and $m,m' \in M$, 
and we have the relation $n+\iota_M(m) = n + \iota_M(m')$ so $ \iota_M(m)  = \iota_M(m')$. Therefore, there exists a $k \in M$ 
such that $k + m' = k+m$. 
Take $Z = n + \iota_M(k + m) = n+\iota_M(k + m')$, and the map $Y \to Z$ given by $Y \to Y + \iota_M(k)$.\\
\indent Next we show that for any two elements $X, Y \in \scr{C}$, there is a third $Z \in \scr{C}$ and maps $X \to Z \leftarrow Y$.\\
\indent We represent $X, Y$ by two elements $n, n' \in M \otimes_{\bb{N}} \bb{Z}$. 
If $n$ and $n'$ are both in the image of $\iota_M$, so that $n = \iota_M(m)$ and $n' = \iota_M(m')$, 
then we take $Z = \iota_M(m + m')$. On the other hand, if $X$ is not in the image of $\iota_M$, then $-n=\iota_M(m)$, 
so there is an arrow $X \to 0$ given by $n \mapsto n + \iota_M(m)$. If $Y$ is in the image of $\iota_M$, then there is an arrow $X \to Y$ so we set $Z=Y$,
and otherwise, we set $Z=0$.\\
\indent These two facts combined are enough to show the claim.
\end{proof}
\begin{corollary}\label{cor:main}
For any $M \in \bb{N}\mathrm{-mod}$, we have $M \otimesl_{\bb{N}} \bb{Z} \simeq M \otimes_{\bb{N}} \bb{Z}[0]$.
\end{corollary}
\begin{proof}
By Theorem~\ref{thm:contractab} we have:
\[\hocoeq(\iota_M,0: M \to M \otimes_{\bb{N}} \bb{Z}) \simeq 0.\]
By Theorem~\ref{thm:ZN}, we have:
\begin{align*}
\hocoker(\iota_M: M \otimesl_{\bb{N}} \bb{Z} \to M \otimes_{\bb{N}} \bb{Z}) & \simeq \hocoeq(\iota_M,0: M \otimesl_{\bb{N}} \bb{Z} \to M \otimes_{\bb{N}} \bb{Z} \otimesl_{\bb{N}} \bb{Z})\\
&\simeq \hocoeq(\iota_M,0: M \to M \otimes_{\bb{N}} \bb{Z}) \otimesl_{\bb{N}} \bb{Z}\\
&\simeq 0.
\end{align*}
Thus, the map $\iota_M: M \otimesl_{\bb{N}} \bb{Z} \to M \otimes_{\bb{N}} \bb{Z}$ is a weak equivalence, as desired.
\end{proof}
\section{Flat topology for semiring schemes}
The previous section showed that homotopy theory is not enough to understand the algebraic structure of flatness. We thus now
turn to studying flatness in a more algebro-geometric manner. 
\begin{definition}[Zariski topology on semirings]
For a semiring $R$, an \textit{ideal} of $R$ will just be an $R$-submodule. We say an $R$-submodule $\ideal{p} \subset R$ is \textit{prime} if $R \setminus \ideal{p}$ is a multiplicatively closed subset. 
Furthermore, for any multiplicatively closed subset $S \subset R$, we may define the 
\textit{localization} semiring $S^{-1}R$ in the usual way: elements of $S^{-1}R$ are represented as `fractions' $r/a$ 
under the equivalence relation 
\[r/s \sim r'/t \Longleftrightarrow \exists q \in S \text{ such that } qtr = qsr'.\]
We may then define the semiring $R[1/a]$ for any $a \in R$ as the localization along the multiplicatively closed subset $\{1,a,a^2,...\} \subset R$.\\
\indent Let $\text{Spec}(R)$ denote the topological space whose points are prime ideals $\ideal{p} \subset \text{Spec}(R)$ equipped with the \textit{Zariski} topology, 
the topology generated by the subsets $D(a) = \{\ideal{p} \cap R \; | \; \ideal{p} \subset \text{Spec}(R[1/a])\}$.\\
\indent A semiring map $f: A \to B$ induces a map $\spec{f}: \spec{B} \to \spec{A}$ via the association $\ideal{q} \mapsto f^{-1}(\ideal{q})$.\\
\indent A semiring $A$ is \textit{local} if the following equivalent conditions hold:
\begin{enumerate}
\item[(a)] $\spec{A}$ has a unique closed point.
\item[(b)] $A$ has a unique maximal ideal.
\item[(c)] There is a proper ideal $\ideal{m} \subset A$ such that every element in $A \setminus \ideal{m}$ is invertible.
\end{enumerate}
One may check the equivalence of these conditions as in classical commutative algebra.
\end{definition}
\begin{proposition}\label{prop:spec}
Let $A\to B$ be a map of semirings. The following hold:
\begin{enumerate}
\item[(i)] Closed subsets of $\spec{A}$ correspond to ideals $I \subset A$ via the association $I \mapsto V(I) = \{\ideal{p} \; | \; I \subset \ideal{p}\}$.
\item[(ii)] $\spec{A}$ is a spectral space, and therefore has an associated constructible topology.
\item[(iii)] The map $\spec{B} \to \spec{A}$ induced by $A \to B$ is continuous with respect to the Zariski topology, and is spectral.
\end{enumerate}
\end{proposition}
\begin{proof}
The proofs of (i)-(iii) are similar to those in classical commutative algebra, see for example \cite[Tag 00DY]{sp}.
\end{proof}
On the other hand, we may define the \textit{fppf topology} on $\text{Spec}(R)$. 
\begin{definition}[The fppf topology on semirings]
Given an $R$-algebra $f:R \to S$, we say $S$ is a \textit{flat $R$-algebra} 
if it is flat as an $R$-module under the structure map $f$, and \textit{faithfully flat} if $S$ is additionally faithful i.e. 
for any map of $R$-modules $f: M \to N$, $f$ is an isomorphism if and only if the induced map 
$f \otimes_R S: M \otimes_R S \to N \otimes_R S$ is injective.\\
\indent An $R$-algebra $S$ is finitely presented if it can be written as a coequaliser of two $R$-algebra homomorphisms 
$f,g: R[y_1,...,y_r] \to R[x_1,...,x_d]$. Equivalently, the functor $\alg{R} \to \text{Set}$ taking an $R$-algebra $A$ to the set 
$\Hom_{\alg{R}}(S, A)$ commutes with filtered colimits.
\end{definition}
With some care, we can adapt arguments in classical commutative algebra. The main difference is that elements in a commutative monoid
no longer `communicate' with each other (via subtraction), so we must work with elements more directly. As a concrete example,
there is no single submodule that detects whether a map is injective $f: M \to N$, rather we must work with all the fibers $f^{-1}(n) \subset M$
for all $n \in N$.
\begin{proposition}\label{prop:flatsurjmax}
Let $f: A \to B$ be a flat map of semirings such that $\text{Spec}(f):\text{Spec}(B) \to \text{Spec}(A)$ is surjective on maximal ideals, then $f$ is faithfully flat.
\end{proposition}
\begin{proof}
We must show that if $\phi: M \to N$ is map of $A$-modules, then 
\[\phi \otimes_A B \text{ is an isomorphism} \implies \phi \text{ is an isomorphism}.\]
We shall first show that $\phi$ is injective.
\begin{claim}\label{claim:flatsurjmax1}
For any $A$-module $M$, the canonical map $M \to M \otimes_A B$ sending $m \mapsto m\otimes 1$ is injective.
\end{claim}
\begin{proof}[Proof of claim~\ref{claim:flatsurjmax1}]
First, we note that if $I \hookrightarrow A$ is an ideal, so a monomorphism of $A$-modules, 
then since $A \to B$ is flat, $I \otimes_A B = IB$.\\
\indent For any two elements $m,n \in M$, define the ideal $I_{m,n} \subset A$:
\[I_{m,n} = \eq(1 \mapsto m, 1 \mapsto n: A \to M),\]
and 
\[J_{m,n}=\eq(1 \mapsto m \otimes 1, 1 \mapsto n\otimes 1: B \to M \otimes_A B).\]
As $B$ is flat over $A$, we see that:
\[I_{m,n}B = I_{m,n} \otimes_A B = \eq(1 \mapsto m \otimes 1, 1 \mapsto n\otimes 1: B \to M\otimes_A B) = J_{m,n}.\]
We must show $J_{m,n} \neq B$ for $m \neq n$. If $m \neq n$, then $I_{m,n} \neq A$, so $I_{m,n}$ is contained in a maximal ideal $\ideal{m} \subset A$. Hence, 
$J_{m,n} = I_{m,n}B \subset \ideal{m}B$. By our assumption that $\text{Spec}(B) \to \text{Spec}(A)$, $\ideal{m}B \neq B$, and we therefore conclude
 that $J_{m,n} \neq B$.
\end{proof}
Hence, if $\phi\otimes_A B$ is injective, then by examining the diagram:
\[\begin{tikzcd} M \arrow[r, "\phi"] \arrow[hookrightarrow,d] & N\arrow[hookrightarrow,d]\\
    M \otimes_A B \arrow[r, "\phi \otimes_A B"] & N\otimes_A B 
\end{tikzcd}\]
we conclude that $\phi$ is injective.\\
\indent Let us now show that $\phi$ is surjective. For any $n \in N$, define the module $I_{n}$ as fitting into the pullback diagram:
\[
\begin{tikzcd}
I_n \arrow[r] \arrow[d]
  & M \arrow[d, "\phi"] \\
A \arrow[r, "1 \mapsto n"]
  & N
  \arrow[ul, phantom, "\lrcorner", very near start]
\end{tikzcd}
\]
and $J_{n\otimes 1}$ as fitting into the pullback diagram:
\[
\begin{tikzcd}
J_{n \otimes 1} \arrow[r] \arrow[d]
  & M\otimes_A B \arrow[d, "\phi \otimes_A B"] \\
B \arrow[r, "1 \mapsto n \otimes 1"]
  & N \otimes_A B
  \arrow[ul, phantom, "\lrcorner", very near start]
\end{tikzcd}
\]
By our assumption that $\phi \otimes_A B$ is an isomorphism, we see that $J_{n \otimes 1} = B$. Since $\phi$ is injective, 
$I_{n}$ is an ideal of $A$, and we must show $1 \in I_{n}$. As $A \to B$ is flat, 
\[I_nB = I_n \otimes_A B = J_{n\otimes1} = B.\]
By our assumption that $\text{Spec}(B) \to \text{Spec}(A)$ is surjective on maximal ideals, we see that $I_n$ is not contained in any maximal ideal of $A$, 
and therefore conclude $I_n = A$, so $1 \in I_n$ as desired.
\end{proof}
\begin{proposition}\label{prop:ff}
Let $f: A \to B$ is a faithfully flat map of semirings, then $\text{Spec}(B) \to \text{Spec}(A)$ is surjective.
\end{proposition}
\begin{proof}
We first show that $\text{Spec}(f)$ is surjective on maximal ideals. In particular, let $\ideal{m} \subset A$ be a maximal ideal and suppose $\ideal{m}B \subsetneq B$, 
then we may take a maximal ideal $\ideal{m}_B$ with $\ideal{m}B \subset \ideal{m}_B \subsetneq B$, and we see 
\[\text{Spec}(f)(\ideal{m}_B) = \ideal{m}\]
Hence, surjectivity on maximal ideals follows from the following claim:
\begin{claim}\label{claim:ff1}
Let $I \subsetneq A$ be a proper ideal. Then $IB=I\otimes_A B \neq B$.
\end{claim}
\begin{proof}[Proof of claim~\ref{claim:ff1}]
Suppose for the sake of contradiction that we may write $1 = \sum_{i=1}^{n} a_i \otimes b_i$ where $a_i \in I$ and $b_i \in B$, and let $\phi: A^{\oplus n} \to A$ denote the map sending $e_i \mapsto a_i$. Then we note that $\phi \otimes_A B$ is surjective. 
Because $A \to B$ is faithfully flat, this means $\phi$ is surjective, which is a contradiction.
\end{proof}
Let $\ideal{p} \subset A$ be any ideal of $A$. Consider the diagram:
\[\begin{tikzcd}
\text{Spec}(B) \arrow[r] & \text{Spec}(A) \\
\text{Spec}(B_{\ideal{p}}) \arrow[u] \arrow[r] & \text{Spec}(A_{\ideal{p}}) \arrow[u]
\end{tikzcd}\]
By what we just proved, the maximal ideal $\ideal{p}A_{\ideal{p}} \in \text{Spec}(A_{\ideal{p}})$ is in the image of the bottom-right arrow. Hence, $\ideal{p}$ is in the image of the top-right arrow, as desired.
\end{proof}
\begin{corollary}\label{cor:ff=fsurj}
A map $f: A \to B$ of semirings is faithfully flat if and only if it is flat and $\text{Spec}(f)$ is surjective.
\end{corollary}
Unsurprisingly now, we have going-down for flat ring maps.
\begin{theorem}[Going-down for semirings]\label{theorem:goingdown}
Let $f: A \to B$ be a flat map of semirings. Then the image of $\text{Spec}(f)$ is closed under generalizations.
\end{theorem}
\begin{proof}
Note that if $\varphi: R \to S$ is a map of \textit{local} semirings that is flat, it is automatically faithfully flat. Indeed, by supposition $\text{Spec}(\varphi)$ is surjective on maximal ideals, 
therefore by Proposition~\ref{prop:flatsurjmax}, is faithfully flat.\\
\indent Now let $\ideal{q} \subset B$ be any ideal mapping to an ideal $\ideal{p} \subset A$. Then the corresponding map $\text{Spec}(B_{\ideal{q}}) \to \text{Spec}(A_{\ideal{p}})$ is surjective, and therefore for any 
generalization $\ideal{p}' \subset \ideal{p} \subset A$ there is an ideal $\ideal{q}' \subset \ideal{q} \subset B$ such that $\text{Spec}(f)(\ideal{q}') = \ideal{p}'$
\end{proof}
We now have a proof of our main theorem.
\begin{theorem}\label{theorem:main}
A flat, finitely presented epimorphism of semirings $f: A \to B$ induces an isomorphism of $\text{Spec}(B)$ onto a quasi-compact open of $\text{Spec}(A)$ in the Zariski-topology.
\end{theorem}
\begin{proof}
Let $\ideal{p} \in \text{Spec}(A)$ be an prime ideal in the image of $\text{Spec}(f)$. Then $f \otimes_{A} A_{\ideal{p}}: A_{\ideal{p}} \to B \otimes_{A} A_{\ideal{p}}$ 
is a flat, finitely presented epimorphism such that $\text{Spec}(f \otimes_{A} A_{\ideal{p}})$ is surjective on maximal ideals. 
Hence, by Proposition~\ref{prop:flatsurjmax}, $f \otimes_{A} A_{\ideal{p}}$ is faithfully flat too, and is therefore an isomorphism.\\
\indent We next claim that there exists a $g \in A \setminus \ideal{p}$ such that $A_{g} \to B_g$ is an isomorphism, 
which is sufficient to conclude the proof of the theorem. Hence, it suffices to show:
\begin{claim}\label{claim:main1}
Let $I$ is a small filtered category, $F: I \to A\mathrm{-alg}_{\text{fp}}$ be a functor. 
Let $f: A \to B$ be a map in $A\mathrm{-alg}_{\text{fp}}$, and for each $i \in I$, let $f_i = f \otimes_A A_i$. 
If $\colim_{i \in I} f_i$ is an isomorphism, then there is an $i \in I$ such that $f_i$ is an isomorphism.
\end{claim}
\begin{proof} 
The proof is a standard result in category theory, but for the sake of convenience, we provide a proof.\\
\indent Let $B_i = B \otimes_A A_i$, $A_{\infty}=\colim_{i \in I} A_{\infty}, B_{\infty}=\colim_{i \in I}B_i$ and 
$f_{\infty}=\colim_{i \in I} f_i$, and let $g: B_{\infty} \to A_{\infty}$ be $A_{\infty}$-algebra inverse for $f$.\\
\indent By standard compactness arguments, there is an $i \in I$ and a map of $A_i$-algebras $g_i: B_i \to A_i$ such that
$g = g_i \otimes_{A_i} A_{\infty}$. Again by compactness, there is a $j \in I$ with $i \to j$ such that the composition
$g_i \otimes_{A_i} A_j \circ f_j: A_j \to A_j$ is the identity. By a symmetric argument, there is a $k \in I$ with $j \to k$ such that
$f_k \circ g_i \otimes_{A_i} A_k: B_k \to B_k$ is the identity, so $f_k$ is an isomorphism, as desired.
\end{proof}
\end{proof}
\begin{remark}[Images of finitely presented morphisms of semiring schemes are not constructible]\label{rem:badclosed}
One may ask if the epimorphism condition in Theorem~\ref{theorem:main} is necessary for the image of $\spec{f}$ to be open?
The usual method in classical commutative algebra to prove something like this rests on the fact that images of finitely presented
ring maps are constructible subsets (\cite[Tag 054k]{sp}), and then one uses going-down to conclude openness. However, consider the following map of semirings:
\[f: \bb{R}_{+}[x] \to \bb{R}_+, \; x \mapsto 1.\]
The map is finitely presented as it is cut-out by the equation $x=1$ in $\bb{R}_{+}[x]$. The spectrum of $\bb{R}_{+}$ has a unique prime ideal $\{0\} \subset \bb{R}_{+}$. 
The interesting fact about $f$ that happens to be true in the semiring setting is that $f^{-1}(\{0\}) = \{0\} \subset \bb{R}_{+}[x]$, so $\{0\}$ is the image of $\spec{f}$. 
But $\{0\}$ is not a constructible subset of $\spec{\bb{R}_{+}[x]}$. Indeed, if it were, then since $\{0\}$ is the generic point of $\spec{\bb{R}_{+}[x]}$, 
it would have to contain a non-empty open subset of $\spec{\bb{R}_{+}[x]}$\textemdash this is because $\spec{\bb{R}_{+}[x]}$ is a spectral space by Proposition~\ref{prop:spec} so 
we may appeal to \cite[Tag 0AAW]{sp}. 
Hence, $\{0\}$ would have to be an open subset of $\spec{\bb{R}_{+}[x]}$, which means that there is an $f \in \bb{R}_{+}[x]$ such that $D(f) = \{0\}$. 
But for every $a \in \bb{R}_{+}$, the ideal $(x+a) \subset \bb{R}_{+}[x]$ is prime, so since there is no polynomial $f$ that vanishes at $-a$ for all $a \in \bb{R}_{+}$, 
we conclude that $D(f) \neq \{0\}$, a contradiction.
\end{remark}
We will end this section with some further examples of how closed subschemes behave strangely in semiring algebraic geometry.
\begin{example}[Closed subschemes of semiring schemes behave unexpectedly]\label{exmp:badclosed}
Let $A=\bb{R}_{+}[x,y]/(x^2=x,y^2=y,1+xy=x+y)$. All elements of $A$ are of the form $a + bx + cy + dxy$ for $a,b,c,d \in \bb{R}_{+}$, 
and due to the relations, we may suppose at least one of the four coefficients is zero. The units are precisely the elements with $a \neq 0$ and $b=c=d=0$,
and there are no zero-divisors. We have the following prime ideals:
\begin{enumerate}
\item[(1)] The unique minimal prime $(0)$.
\item[(2)] The height one primes $(x)$, $(y)$.
\item[(3)] Height two primes of the form 
\[\ideal{p}=\{a + bx + cy + dxy| \; (a=0) \lor (d>0) \lor (a>0, b>0, c=d=0)\},\]
\[\ideal{q}=\{a + bx + cy + dxy| \; (a=0) \lor (d>0) \lor (a>0, c>0, b=d=0)\}.\]
\item[(4)] The unique maximal ideal $\ideal{m}$ consisting of the non-units.
\end{enumerate}
Therefore, $\spec{A}$ has Krull dimension $3$.\\
\indent A similar characterisation shows that $A[1/y]=\bb{R}_{+}[x]/(x=x^2)$ and $A[1/x]=\bb{R}_{+}[y]/(y=y^2)$ 
both have Krull dimension $2$. Note that $A \to A[1/x] \times A[1/y]$ is injective, and each $x$ and $y$ is integral over $A$, 
so $A \to A[1/x] \times A[1/y]$ is an integral extension of semirings. However, the map $\spec{A[1/x]} \coprod \spec{A[1/y]} \to \spec{A}$ 
is not surjective, in contrast to the usual going-up theorem in classical commutative algebra.\\
\indent In fact, for any $\ideal{p}$ of type $3$ above, there is no map of semirings $f:A \to B$ with $\ideal{p} = \eq(f,0:A \to B)$. 
In particular, both $xy$ and $1+xy$ are in $\ideal{p}$, so
\[1=1+f(xy)=f(1+xy)=0\]
in $B$, so $B$ is a trivial semiring, which would imply $\ideal{p}= A$, a contradiction.
\end{example}
\section{Flatness of quotient maps}\label{sect:flatquot}
This section is concerned with understanding when quotients $\bb{R}_{+}[x_1,...,x_n]/(f \sim g)$ are flat over $\bb{R}_{+}$. 
This is a surprisingly subtle question, even in the case of one variable polynomials.\\
\indent We begin with first an explicit construction of quotients of semirings.
\begin{definition}
Let $A$ be a semiring, $I$ a set, and for each $i \in I$, let $(f_i,g_i) \in A \times A$ be pairs of elements.
Let $\cal{R}(f_i\sim g_i)_{i\in I}$ be the smallest equivalence relation in $A\times A$ containing the elements $(f_i,g_i)$ that forms a subalgebra of $A \times A$.
Then the set $A/\cal{R}(f_i\sim g_i)_{i\in I}$ of equivalence classes has a natural semiring structure induced from $A$, 
and we call it the \textit{quotient} of $A$ by the relations $(f_i \sim g_i)_{i \in I}$.\\
\indent Let us show that $A/\cal{R}(f_i\sim g_i)_{i\in I}$ has the following universal property: for any semiring $B$ and any map $\phi: A \to B$ such 
that $\phi(f_i) = \phi(g_i)$ for all $i \in I$, there is a unique map $\overline{\phi}: A/\cal{R}(f_i\sim g_i)_{i\in I} \to B$ such that the diagram:
\[\begin{tikzcd} A \arrow[r, "\phi"] \arrow[d] & B \\
    A/\cal{R}(f_i\sim g_i)_{i\in I} \arrow[ur, "\overline{\phi}"]  & 
\end{tikzcd}\]
commutes. Let $\mathrm{diag}(B) \subset B \times B$ denote the trivial equivalence relation, clearly $\mathrm{diag}(B)$ is a sub-semiring. Consider the map 
$\phi \times \phi: A \times A \to B \times B$. Since $\phi(f_i) = \phi(g_i)$, we have $(f_i,g_i) \in (\phi \times \phi)^{-1}(\mathrm{diag}(B))$.
Moreover, $(\phi \times \phi)^{-1}(\mathrm{diag}(B))$ forms an equivalence relation, and is clearly a sub-semiring of $A \times A$,
so by minimality of $\cal{R}(f_i\sim g_i)_{i\in I}$, we have $\cal{R}(f_i\sim g_i)_{i\in I} \subset (\phi \times \phi)^{-1}(\mathrm{diag}(B))$. 
Hence, the map $\overline{\phi}$ taking $[a] \mapsto \phi(a)$ is well-defined, 
and is clearly unique.
\end{definition}
The set $\cal{R}(f_i \sim g_i)_{i \in I}$ is very hard to describe, and in my view,
is the primary reason for flatness and closed immersions in semiring algebraic geometry presenting many difficulties. In the case of one 
relation, we are able to be explicit.
\begin{proposition}\label{prop:quotientrels}
Let $f,g \in \bb{R}_{+}[x_1,...,x_n]$ be two elements and let $A=\bb{R}_{+}[x_1,...,x_n]/(f \sim g)$. The equivalence 
relation $\cal{R}(f \sim g) \subset \bb{R}_{+}[x_1,...,x_n] \times \bb{R}_{+}[x_1,...,x_n]$ 
generated by $(f \sim g)$ consists of pairs $(a_0 + b_0f + c_0g, a_t + b_tg + c_tf)$ for $a_0,b_0,c_0, a_t,b_t,c_t \in \bb{R}_{+}[x_1,...,x_n]$
such that there is a sequence of elements $(a_1,b_1,c_1),...,(a_{t-1},b_{t-1},c_{t-1})$ in $\bb{R}_{+}[x_1,...,x_n]$ with:
\[a_{i} + b_{i}g + c_{i}f=a_{i+1} + b_{i+1}f + c_{i+1}g \]
as elements in $\bb{R}_{+}[x_1,...,x_n]$ for all $0 \le i \le t-1$.
\end{proposition}
\begin{proof}
Let $\cal{R}'$ be the proposed relation. First, we note that $\cal{R}'$ is by design transitive, reflexive and symmetric, and hence an equivalence relation.
On the other hand, by reflexivity and symmetry, elements of the form $(a + bf + cg, a + bg + cf)$ are in $\cal{R}(f \sim g)$ for any $a,b,c \in A$.
Moreover, by transitivity, if there is a sequence of elements $(a_1,b_1,c_1),...,(a_{t-1},b_{t-1},c_{t-1})$ in $A$ with:
\[a_{i} + b_{i}g + c_{i}f= a_{i+1} + b_{i+1}f + c_{i+1}g\]
as elements in $\bb{R}_{+}[x_1,...,x_n]$ for all $0 \le i \le t-1$, then $(a_0 + b_0f + c_0g, a_t + b_tg + c_tf) \in \cal{R}(f\sim g)$. Therefore, $\cal{R}' \subset \cal{R}(f \sim g)$, and to finish,
 it suffices to show that $\cal{R}'$ is a sub-semiring of $A \times A$.\\
\indent For any $d \in \bb{R}_{+}[x_1,...,x_n]$ and $(h,h') \in \cal{R}'$, then $(dh,dh')\in \cal{R}'$, because the triples $\{(a_i,b_i,c_i)\}_{i=0}^t$
realizing the equality $(h,h')\in \cal{R}'$ may be multiplied by $d$ to realize $(dh,dh') \in \cal{R}'$. Therefore, for any two elements
$(h_1,h'_1), (h_2,h'_2) \in \cal{R}'$, we have $(h_1h_2,h_1h'_2) \in \cal{R}'$ and $(h'_1h_2,h'_1h'_2) \in \cal{R}'$, 
so $(h_1h_2,h'_1h'_2) \in \cal{R}'$
by transitivity. 
\end{proof}
\begin{corollary}\label{cor:univpointless}
The semiring $A:=\bb{R}_{+}[x_1,y_1]/(1+x_1^2+y_1^2 \sim 2x_1y_1)$ is not flat over $\bb{R}_{+}$.
\end{corollary}
\begin{proof}
Let $\cal{R}$ be the equivalence relation generated by $1+x_1^2 + y_1^2 \sim 2x_1y_1$. Using Proposition~\ref{prop:quotientrels}, we may observe that if $(a,b) \in \cal{R}$, then
$\mathrm{deg}(a)=\mathrm{deg}(b)$. Note that this would more generally hold for any relation of the form $p \sim q$ where $p$ and $q$ are polynomials of the same degree. 
Therefore, by matching top-degrees, we see that if $(a,b) \in \cal{R}$ and $a'$ and $b'$ are the top-degree terms of $a$ and $b$ respectively,
then $(a',b') \in \cal{R}'$, where $\cal{R}'$ is the equivalence relation generated by $x_1^2+y_1^2 \sim 2x_1y_1$.\\
\indent Since $\bb{R}_{+}$ is cancellative, we know by Proposition~\ref{prop:flatpositivity} that if $A$ is flat, then it is also cancellative.
We have two equations:
\[2x_1 + 2x_1^3 + 2x_1y_1^2 \sim 4x_1^2y_1\]
\[y_1 + x_1^2y_1 +  y_1^3 \sim 2x_1y_1^2,\]
so by substitution, cancellativity and matching top-degrees, we find that $(2x_1^3 + y_1^3,3x_1^2y_1)$ is in the degree $3$ part of $\cal{R}'$, which is generated by the relations:
\[x_1^3 + x_1y_1^2 \sim 2x_1^2y_1\]
\[x_1^2y_1 + y_1^3 \sim 2x_1y_1^2.\]
For arbitrary elements $h=a_1x_1^3 + a_2x_1^2y_1 + a_3x_1y_1^2 + a_4y_1^3$ and $k=b_1x_1^3 + b_2x_1^2y_1 + b_3x_1y_1^2 + b_4y_1^3$, if 
$(h,k)$ is in the degree $3$ part of $\cal{R}'$, then by Proposition~\ref{prop:quotientrels}, if $h$ and $k$ are not equal as polynomials, 
then because both degree $3$ generators for $\cal{R}'$ contain non-zero coefficients for $x_1y_1^2$ and $x_1^2y_1$ (on both sides), we see that
$a_2 + a_3$ and $b_2 + b_3$ are both strictly positive. Therefore, the pair $(2x_1^3 + y_1^3,3x_1^2y_1)$ cannot be in the degree $3$ part of $\cal{R}'$, so we have a contradiction, 
and $A$ is not flat over $\bb{R}_{+}$.
\end{proof}
\begin{remark}
Corollary~\ref{cor:univpointless} shows that the obvious choice for a counterexample to Question~\ref{mainquestion} does not work. 
A natural next step would be to prove that the higher dimensional analogue of the relation $1+x_1^2+y_1^2 \sim 2x_1y_1$, which is 
$1+f \sim g$ where $f=\sum_{i \in [n]} x_i^2 + y_i^2$ and $g=\sum_{i \in [n]} 2x_iy_i$, does not generate a $\bb{R}_{+}$-flat quotient
$A=\bb{R}_{+}[x_1,y_1,...,x_n,y_n]/(1+f \sim g)$ either. We did not pursue this question seriously, but we outline a possible approach mimicking the proof for $n=1$.\\
\indent Indeed, the key idea was to use cancellativity to find a relation that cannot be generated by $f \sim g$ alone. We could consider the following two equations:
\[g+(f-y_n^2)g + y_n^2g \sim g^2\]
\[y_n^2+ y_n^2f \sim y_n^2g,\]
so by substitution, cancellativity and matching top-degrees, we find that $(h,k)$ is in the degree $4$ part of $f \sim g$, where 
\[h=(f-y_n^2)(g+y_n^2)- x_n^2y_n^2 + y_n^4 \]
\[k=g^2 - x_n^2y_n^2.\]
Let $A_n \subset A$ denote the set of homogenous degree $n$ monomials. The degree $4$ part of $\cal{R}'$ is the transitive and symmetric closure of the relations
generated by the equations
\[af \sim ag, a \in A_2.\]
These equations however do not exhibit the same kind of invariant as the degree $3$ relations in the $n=1$ case, so it is not clear how to show that $(h,k)$ is not in the degree $4$ part of $\cal{R}'$.
\end{remark}
The fundamental tool to analyze flatness is the equational criterion for flatness, which holds in the semiring setting as well (see \cite[3.5]{borgersemi}).
\begin{theorem}[Equational criterion for flatness]\label{theorem:equationalflatness}
Let $A$ be a semiring, and let $M$ be an $A$-module. Then $M$ is flat if and only if for any finite collection of elements $m_1,...,m_k \in M$ satisfying a relation:
\[\sum_{i\in [k]} a_{i}m_i = \sum_{i\in [k]} b_{i}m_i,\]
with $a_{i}, b_{i} \in A$, there is a finite set of elements $\{w_l\}_{l\in [t]}, \subset M$, and for each $i\in[k], l \in [t]$,
there are elements $c_{il} \in A$ such that:
\[m_i = \sum_{l\in [t]} c_{il}w_l = \sum_{l\in [t]} d_{il}w_l\]
\[\sum_{i\in [k]} a_{i}c_{il} = \sum_{i\in [k]} b_{i}c_{il}, \; \forall l \in [t].\]
Consequently, $M$ is flat if and only if it can be written as a filtered colimit of free $A$-modules.
\end{theorem}
Here are a few consequences.
\begin{corollary}\label{cor:flatquotient}
Let $A$ be a semiring, $B=A[x_1,...,x_n]/(p_i \sim q_i)_{i \in [k]}$ be a quotient of a polynomial semiring over $A$. 
If $B$ is flat, then the following are true:
\begin{enumerate}
\item[(i)] Without changing the choice of surjection $A[x_1,...,x_n] \to B$, we may assume $p_i$ is a monomial for all $i \in [k]$.
\item[(ii)] Now let $A = \bb{N}$, and suppose that no monomial goes to $0$ in $B$. 
Using (i), for each $i$ write $p_i=x^{E_i}$ and $q_i = \sum_{j\in [m_i]} a_i x^{F_{ij}}$ for $E_i, F_{ij} \in \bb{N}^n$. 
Then for each $i \in [k]$, there do not exist $\{a_j\}_{j \in [m_i]} \subset \bb{N}$ such that
\[\sum_{j \in [m_i]} a_j (E_i-F_{ij}) = 0,\]
unless $a_j = 0$ for all $j \in [m_i]$. In particular, $E_i$ is not in the convex hull of $\{F_{ij}\}_{j \in [m_i]}$.
\end{enumerate}
\end{corollary}
We first prove a lemma.
\begin{lemma}[Boosting polynomial relations using flatness]\label{lemma:flatmonomial}
Let $A$ be a semiring, and let $\phi: A[x_1,...,x_n] \to B$ a surjection of semirings. Suppose $B$ is flat over $A$. 
Then for any two elements $f,g \in A[x_1,...,x_n]$ such that $\phi(f)=\phi(g)$, there is a finite set of elements $t_i,s_i \in A[x_1,...,x_n]$ with $\phi(t_i)=\phi(s_i)$ for all $i$,
such that each $t_i$ is a monomial, and the relation $f \sim g$ is implied by the relations $t_i \sim s_i$ for all $i$.
\end{lemma}
\begin{proof}
Write $f= \sum_{i\in [m]} a_i x^{E_i}$ and $g = \sum_{j\in [l]} b_j x^{F_j}$ for $a_i,b_j \in A$ and $E_i, F_j \in \bb{N}^n$. Since $f$ and $g$ have the same image in $B$, 
by Theorem~\ref{theorem:equationalflatness}, 
there are elements $w_1,...,w_t \in A[x_1,...,x_n]$, elements $\{c_{ik}\}_{i \in [m], k \in [t]}$ and $\{d_{jk}\}_{j \in [l], k \in [t]}$ in $A$ such that:
\[\phi(x^{E_i}) = \phi\left(\sum_{k \in [t]} c_{ik} w_k\right), \; \phi(x^{F_j}) = \phi\left(\sum_{k \in [t]} d_{jk} w_k\right),\]
and for each $k \in [t]$,
\[\sum_{i\in [m]} a_i c_{ik} = \sum_{j\in [l]} b_j d_{jk}.\]
By setting $t_i=x^{E_i}$ for $i \in [m]$ and $t_{m+j} = x^{F_j}$ for $j \in [l]$, and $s_i = \sum_{k \in [t]} c_{ik} w_k$ for $i \in [m]$ and
$s_{m+j} = \sum_{k \in [t]} d_{jk} w_k$ for $j \in [l]$, we conclude the proof.
\end{proof}
\begin{proof}[Proof of Corollary~\ref{cor:flatquotient}]
For (i), we may apply Lemma~\ref{lemma:flatmonomial} to each relation $p_i \sim q_i$ to obtain the desired result.\\
\indent For (ii), recall that $B$ is an integral commutative monoid (see e.g. Proposition~\ref{prop:flatpositivity}), and 
can therefore be endowed with a partial ordering by Proposition~\ref{prop:poset}. We will endow the space $\bb{N}^n$ with the following ordering. 
For two $E,F \in \bb{N}^n$, we say $E \le F$ if and only if $x^E \le x^F$ in $B$. It's clear that $(\bb{N}^n, \le)$ defines
a partial ordering, and since $B$ is integral as a module over itself, $\le$ is compatible with addition.\\
\indent Now suppose for the sake of contradiction that for some $i \in [k]$, there are numbers $\{a_j\}_{j \in [m_i]} \subset \bb{N}$ such that:
\[\sum_{j\in [m_i]} a_j (E_i-F_{ij}) = 0.\]
This implies that:
\[\left(\sum_{j\in [m_i]} a_j\right) E_i = \sum_{j\in [m_i]} a_j F_{ij}.\]
On the other hand, the equation $p_i=q_i$ implies that $E_i > F_{ij}$ for all $1 \le j \le m_i$. Therefore, we have:
\[\left(\sum_{j\in [m_i]} a_j\right) E_i > \sum_{j\in [m_i]} a_j F_{ij} = \left(\sum_{j\in [m_i]} a_j\right) E_i,\]
a contradiction unless $\sum_{j\in [m_i]} a_j=0 \implies a_j=0$ for all $j \in [m_i]$.
\end{proof}
\begin{corollary}
Let $A=\bb{N}[x_1,...,x_n]/(x^{E}\sim \sum_{j \in [m]} a_j x^{F_j})$ for $E,F_j \in \bb{N}^n$ and $a_j \in \bb{N}$. 
If $A$ is flat over $\bb{N}$, then
$A$ has an $\bb{R}_{+}$-point.
\end{corollary}
\begin{proof}
By Corollary~\ref{cor:flatquotient} (ii) and the Hahn-Banach theorem, there is a linear functional $\lambda: \bb{R}^n \to \bb{R}$ such that
\[\lambda(E) > \lambda(F_j), \; \forall j \in [m].\]
By perturbation we may assume $\lambda: \bb{Q}^n \to \bb{Q}$, and by scaling we may assume $\lambda: \bb{Z}^n \to \bb{Z}$. Define the semiring map
\[\phi: \bb{N}[x_1,...,x_n] \to \bb{N}[t^{\pm 1}]\]
induced by sending $x_i \mapsto t^{\lambda(e_i)}$ where $\{e_i\}_{i=1}^n$ is the standard basis of $\bb{Z}^n$. Let
\[A'=\bb{N}[t^{\pm1}]/(t^{\lambda(E)} \sim \sum_{j \in [m]} a_j t^{\lambda(F_j)}).\]
Then $\phi$ induces a map $A \to A'$, and since $\lambda(E) > \lambda(F_j)$ for all $j \in [m]$, the equation
\[t^{\lambda(E)} = \sum_{j \in [m]} a_j t^{\lambda(F_j)}\]
has at least one solution in $\bb{R}_{+}$. Therefore, $A'$ has an $\bb{R}_{+}$-point, and hence so does $A$.
\end{proof}

\chapter{Flatness and Descendability}

Let $R$ be a commutative ring and $M$ an $R$-module. Studying the extent to which flat modules are not projective has presented a fruitful 
line of inquiry in homological algebra. 
A truly remarkable characterization of flat modules was given by Raynaud and Gruson in their seminal 1971 paper \cite{Raynaud1971}.
\begin{theorem}\label{thm:rg1}
A module $M$ over a ring $R$ is projective if and only if it is flat, Mittag-Leffler and a direct sum of countably generated modules. 
\cite[Theorem 2.2.1]{Raynaud1971} (see also \cite[Tag 059Z]{sp})
\end{theorem}
The conditions of being Mittag-Leffler and countably generated are related to the well-known fact that the functor:
\[\lim_{I}: \fun(I, \mod{R}) \to \mod{R}\]
has derived functors $R^n\lim_{I}$ whose vanishing is related to the size of the indexing category $I$.
In particular, if $I$ is countable, then $R^k\lim_I = 0$ for all $k \geq 2$, and more generally if $|I|< \aleph_n$
then $R^k\lim_I = 0$ for all $k \geq n+1$. After writing $M$ as a filtered colimit of finite free $R$-modules,
the countably generated assumption allows us to regard $M$ as a filtered colimit indexed by a countable category, 
and the Mittag-Leffler condition ensures a vanishing of $R^1\lim_I$. Theorem~\ref{thm:rg1}
led to a proof that projective modules $M$ satisfy descent with respect to faithfully flat extensions $R \to S$, i.e. if
$M \otimes_R S$ is a projective $S$-module, then $M$ is a projective $R$-module (see \cite[Tag 05A4]{sp}).\\
\indent Theorem~\ref{thm:rg1} has been generalized in various directions, one of which is the relationship between the base ring $R$ and the projective
dimension of a flat module $M$. The following two results are quite remarkable:
\begin{theorem}\label{thm:gj0}
Let $R$ be a commutative ring and $M$ a flat $R$-module.
\begin{enumerate}
\item If $|R| \leq \aleph_n$, then $M$ has projective dimension at most $n+1$. \cite[Theorem 7.10]{gj}
\item If $R$ is Noetherian with $\text{dim}(R) = d$, then $M$ has projective dimension at most $d$. \cite[Corollary 7.2]{gj}
\end{enumerate}
\end{theorem}
These results suggest that there should exist flat modules of infinite projective dimension over rings of large cardinality and infinite Krull dimension.
This suggestion is especially convincing when one remarks that there exists rings $R$, called \textit{absolutely flat rings}, which are characterized by all their modules being flat.
Boolean rings $R$ for example, where for any $r\in R$, $r^2=r$, are absolutely flat, and Osofsky in \cite{osofsky} showed that there exist boolean rings of arbitrarily large cardinality
that admit modules of infinite projective dimension.\\
\indent The techniques used to prove Theorem \ref{thm:gj0} are quite unique, and we wish to briefly summarize them here (see also \cite{jiangdescendable} for an $\infty$-categorical perspective). First 
we recall that in a category $\scr{C}$, an object $X$ is said to be $\aleph_n$-Artinian if any filtered descending chain of subobjects of $X$ 
has a cofinal subchain of length at most $\aleph_n$. The key technical result used to prove Theorem \ref{thm:gj0} is the following:
\begin{theorem}\label{thm:gj1}
Let $\scr{C}$ be an abelian category, and $T:\scr{C} \to \mathrm{Ab}$ an exact functor to the category of abelian groups $\mathrm{Ab}$. If every object of $\scr{C}$ is $\aleph_n$-Artinian,
then the injective dimension of $T$ as an object of the functor category $\mathrm{Lex}(\scr{C}, \mathrm{Ab})$ of left exact functors is at most $n+1$. 
\cite[Theorem 6.5]{gj}
\end{theorem}
We also recall that an exact sequence 
\[\scr{E}: 0 \to A \to B \to C \to 0\]
of $R$-modules is said to be \textit{pure} if for any $R$-module $M$, the induced sequence
\[0 \to A \otimes_R N \to B \otimes_R N \to C \otimes_R N \to 0\]
is exact for any $R$-module $N$. One can observe that $\scr{E}$ is pure if and only if $C$ is a flat $R$-module, if and only if $A \to B$ is \textit{universally injective}.
 A module $K$ is said to be \textit{pure injective} if for any pure exact sequence $\scr{E}$, the induced sequence
\[0 \to \Hom_R(C, K) \to \Hom_R(B, K) \to \Hom_R(A, K) \to 0\]
is exact. Let $\mod{R^{\mathrm{fp}}}$ denote the category of finitely presented $R$-modules. Let $\overline{\overline{M}}$ denote the functor:
\[\overline{\overline{M}}: \mod{R^{\mathrm{fp}}} \to \mathrm{Ab}, \quad  \overline{\overline{M}}(F) = F \otimes_R M\]
for an $R$-module $M$.
\begin{theorem}\label{thm:gj2}
Let $R$ be a commutative ring.
\begin{enumerate}
\item [(i)] A functor $F \in \mathrm{Lex}(\mod{R^{\mathrm{fp}}}, \mathrm{Ab})$ is injective if and only if $F \cong \Hom_R(-, K)$ for some pure injective $R$-module $K$. 
\cite[Proposition 1.2]{gj}
\item [(ii)] If the injective dimension of $\overline{\overline{M}}$ in $\fun(\mod{R^{\mathrm{fp}}}, \mathrm{Ab})$ is at most $n$ for any flat module $M$,
then the projective dimension of any flat $R$-module is at most $n$. \cite[Proposition 3.8]{gj}
\end{enumerate}
\end{theorem}
Finally, to prove Theorem \ref{thm:gj0} (i), one shows that any finitely presented module over a ring $R$ of cardinality at most $\aleph_n$
is $\aleph_n$-Noetherian. We have an equivalence:
\[\Phi: \mathrm{Lex}(\mod{R^{\mathrm{fp}}}, \mathrm{Ab}) \cong \mathrm{Lex}(C(R)^{\text{op}}, \mathrm{Ab}), \; \; F \mapsto \Hom_{\mathrm{Lex}(\mod{R^{\mathrm{fp}}}, \mathrm{Ab})}(-,F)\]
where $C(R)^{\text{op}}$ is the full subcategory of coherent functors in $\mathrm{Lex}(\mod{R^{\mathrm{fp}}}, \mathrm{Ab})$,
and a functor $F$ is coherent if it fits into an exact sequence
\[\Hom_R(-, M) \to \Hom_R(-, N) \to F \to 0\]
for some finitely presented $R$-modules $M$ and $N$. Moreover, $\Phi(\overline{\overline{M}})$ becomes an exact functor (see \cite[Lemma 5.3]{gj}), and the injective dimension
of $\Phi(\overline{\overline{M}})$ and $\overline{\overline{M}}$ in their respective abelian categories is the same. 
Any object in $C(R)$ is $\aleph_n$-Noetherian (see \cite[Theorem 7.9]{gj}), so therefore
every object of $C(R)^{\text{op}}$ is $\aleph_n$-Artinian (see \cite[Theorem 5.6]{gj}), so we can conclude by Theorem~\ref{thm:gj1} and Theorem~\ref{thm:gj2}.\\
\indent In this chapter, we study the relationship between flatness and projective dimension from a different perspective.
In particular, we are interested in the notion of \textit{descendable} map of rings $f: A \to B$, which can be thought of as a strong version of $f$ satisfying descent. 
To be more precise, $f$ is descendable if the map $\{A\}_n \to \{ \text{Tot}_n B^{\bullet}\}_n$ is a pro-isomorphism of $A$-modules where $B^{\bullet}$ 
is the derived Cech nerve and on the other hand, for descent, one usually requires the weaker condition that 
$A \to R\lim B^{\bullet}$ is an isomorphism. 
For a more detailed study of descendability we refer the reader to \cite{DescendMathew}. We briefly mention two results that illustrate the good nature of descendability.
\begin{lemma}
Let $A \to B$ be a descendable map of rings.
\begin{enumerate}
\item[(i)] Let $B \to C$ be a map of rings. If $A \to C$ is descendable, then so is $A \to B$.
\item[(ii)] Let $A \to D$ be a map of rings. Then $D \to B \otimes_A D$ is descendable.
\end{enumerate}
\end{lemma}
\begin{proof}
$A \to B$ is descendable when there exists an $n \in \bb{N}$ such that the map $A \to \lim_{\Delta_{\le n}} B^{\bullet}$ admits a section in the derived category of $A$-modules.
For (i), since $A \to C$ is descendable, there exists an $n \in \bb{N}$ such that $A \to \lim_{\Delta_{\le n}} C^{\bullet}$ admits a section, and we need only pre-compose this section with
the map $\lim_{\Delta_{\le n}} B^{\bullet} \to \lim_{\Delta_{\le n}} C^{\bullet}$ to obtain a section for $A \to \lim_{\Delta_{\le n}} B^{\bullet}$. 
(ii) follows a similar argument by base-changing the section of $A \to \lim_{\Delta_{\le n}} B^{\bullet}$ along $A \to D$.
\end{proof}
It is a celebrated theorem that the $2$-category of quasi-coherent modules have descent with respect to faithfully flat ring maps (see \cite[Tag0306]{sp}).
For descendable maps, we have a much more general result:
\begin{theorem}\label{thm:luriedescend}
Let $A \to B$ be a descendable map of ($\bb{E}_{\infty}$)-rings. Let $\scr{C}$ be a stable $A$-linear $\infty$-category. Then the functor:
\[\scr{C} \to \mathrm{LMod}_{B}(\scr{C}), \quad X \mapsto B \otimes_A X\]
is comonadic, i.e. we have an equivalence of $A$-linear $\infty$-categories:
\[\scr{C} \cong \lim_{B^{\bullet}}(\scr{C})\]
where $B^{\bullet}$ is the derived Cech nerve of $A \to B$. \cite[D.3.4.1]{LurieSAG}
\end{theorem}
In this chapter, we are interested in the question of whether faithfully flat rings maps are \text{descendable}. 
In concrete terms, if $f: A \to B$ is a faithfully flat map of (discrete, commutative) rings, we have an exact sequence $A \to B \to \coker(f)$ 
corresponding to an element $\eta \in \Ext_A^{1}(\coker(f), A)$. The map $f$ will be descendable precisely when there exists an 
integer $n$ such that $\eta^{\otimes_A n}=0$ as an element of $\Ext^n_A(\coker(f)^{\otimes_A n}, A)$.\\
\indent By Theorem~\ref{thm:gj0}, we know that if $A$ is a ring of cardinality at most $\aleph_n$, then any flat $A$-module has projective dimension at most $n+1$,
so therefore $f$ is descendable, of exponent at most $n+1$, since $\coker(f)^{\otimes_A k}$ is a flat $A$-module for any $k > 0$. 
Similarly, if $A$ is a Noetherian ring of Krull dimension $d$, then any flat $A$-module has projective dimension at most $d$,
so $f$ is descendable of exponent at most $d$.\\
\indent Therefore, if faithfully flat rings maps were to fail to be descendable, then the base ring $A$ would have to satisfy $|A| \ge \aleph_{\omega}$ and, 
if Noetherian, $\dim(A) =\infty$. We will construct examples of faithfully flat ring maps $A \to A'$ that is not descendable for a large
class of base rings $A$ that contain an \textit{indivisible sequence}, such as various boolean rings or polynomial rings in infinitely many variables over a `large' field. 
Our strategy can be summarized as follows:
\begin{enumerate}
\item[(i)] Show that there exists rings $A$ with arbitrarily large non-vanishing cup-products, 
i.e. $A$-linear module maps $f_i: M_i \to N_i$ such that $\otimes_{i\in [n],A} f_i \neq 0$, 
using \textit{indivisible sequences} (Definition~\ref{definition:indivisible}).
\item[(ii)] Devise a general procedure to construct non-descendable faithfully flat ring maps from non-vanishing cup-products of universally
injective module maps.
\end{enumerate}
One may wonder whether, granted (ii), one can construct non-descendable modules in a more straightforward manner given any flat module
$M$ of infinite projective dimension. In particular, the author initially speculated that it would be possible to take a flat module $M$ of infinite projective dimension,
 consider a projective resultion $P_{\bullet} \to M$ with boundaries $M_i = \ker(P_i \to P_{i-1})$ (and define $M_0=M$) forming extension classes $\eta_i \in \Ext^1_A(M_i, M_{i+1})$
whose composition $\eta_0 \circ \eta_1 \circ \eta_2 \circ... \circ \eta_n$ would be non-zero for all $n\ge 1$, and then try to construct non-vanishing cup-products from these extension classes.
We believe the following result illustrates why this shouldn't be possible,
since it could be the case that $M_i \otimes_A M_{i+1}=0$ for at least some $i\ge 1$, and therefore 
that descendability is more subtle than just the projective dimension of flat modules.
\begin{lemma}
There exists a ring $A$ and a flat $A$-module $M$ such that $M \otimes_A M = 0$ but $M \neq 0$.
\end{lemma}
\begin{proof}
Let $A=k[x_1, x_2, \ldots]/(x_1^2, x_2^2, \ldots)$ where $k$ is a field. Let $M=\colim_{i \in I} A$ where $I=\bb{N}$ and the 
maps in the colimit are given by multiplication by $x_i$.
Then $M$ is flat since it is a filtered colimit of free modules. However, $M \otimes_A M = 0$ since $M \otimes_A M = \colim_{(i,j) \in I \times I} A$ 
where the maps $(i,j) \to (i+1,j)$ is multiplication by $x_{i+1}$ and the map $(i,j) \to (i,j+1)$ is multiplication by 
$x_{j+1}$, so any element in $A$ in position $(i,j)$ is sent to $0$ in the factor of $A$ in position $(\max(i,j)+1, \max(i,j)+1)$.
\end{proof}

\subsubsection*{Notation and conventions} Let $A$ be a commutative ring. The derived category $D(A)$ of $A$-modules, regarded as a stable $\infty$-category, admits a symmetric monoidal structure given by $\otimesl_A$. Therefore, given $A$-modules $M,N,M',N'$, there is an $A$-bilinear pairing:
\[\text{RHom}_A(M,N) \times \text{RHom}_A(M',N') \to \text{RHom}_A(M \otimesl_A M', N \otimesl_A N').\]
Passing to the homotopy category, we obtain an $A$-bilinear pairing of $A$-modules:
\[\Ext^0_A(M,N) \times \Ext^0_A(M',N') \to \Ext^0_A(M \otimesl_A M', N \otimesl_A N'),\]
and for a pair of $f \in \Ext^0_A(M,N), g \in \Ext^0_A(M',N')$, we will set 
$f \otimesl_A g \in \Ext^0_A(M \otimesl_A M', N \otimesl_A N')$ to be the image of $(f,g)$ under the aforementioned pairing. 
Moreover, if $N=N'=A$, we will compose with the multiplication isomorphism $A \otimesl_A A \to A$ so that the target is 
$\Ext^0_A(M \otimesl_A M', A)$, and note that flatness conditions on our modules remove the necessity for derived tensor products.
\begin{enumerate}
\item[(i)] If $S$ is a set and $R$ a ring, define $\Hom^{\mathrm{fin}}_{\mathrm{Set}}(S, R) \subset \Hom_{\mathrm{Set}}(S,R)$ to be the set-theoretic maps with finite support i.e. $f \in \Hom^{\mathrm{fin}}_{\mathrm{Set}}(S, R)$ if there exists a subset $K \subset S$ with $|K| < \infty$ such that $f(S\setminus K) = 0$. For each $t \in S$, set $\delta_{s=t} \in \Hom^{\mathrm{fin}}_{\mathrm{Set}}(S, R)$ to be the function that is $1$ if $s=t$ and $0$ else.
\item[(ii)]Further, if $S_i$ are sets indexed by the set $\{1,2,...,n\}$, then we will identify the vector $1_{s_1,s_2,...,s_n}$ in the nested direct sum $\oplus_{S_1} \oplus_{S_2} \oplus ... \oplus_{S_n} R$ as follows: Every vector in $v \in \oplus_{S_1} \oplus_{S_2} \oplus ... \oplus_{S_n} R$ can be identified as a function $f_{v} \in \Hom^{\mathrm{fin}}_{\mathrm{Set}}(\prod_i S_i, R)$, and under this identification, $1_{s_1,...,s_n}$ corresponds to $\delta_{t=s}$ where $s=(s_1,s_2,...,s_n) \in \prod_i S_i$.
\item[(iii)]For any set $S$, we let $p^n_i: S^{\times n} \to S^{\times n-1}$ be defined as follows: For any $s = (s_1,s_2,..,s_n) \in S^{\times n}$,
\[p^n_i(s) = (s_1,s_2,...,s_{i-1},s_{i+1},...,s_n).\]
For a point $s = (s_1,...,s_{n-1}) \in S^{\times n-1}$, we consider the degeneracy map $t_i^{s}: S \to S^{\times n}$ which takes a point $s' \in S$ to the point:
\[(s_1, ..., s_{i-1}, s', s_{i},...,s_{n-1}) \in S^{\times n}.\]
\item[(iv)] For a ring $R$ and a set $S$, let $R[S]=\Sym_R(\oplus_{s \in S} R) = R[X_s]_{s \in S}$ and $R\{S\} = R[X_s]_{s \in S}/I$ where $I$ is the ideal generated by $X_s^2 - X_s \in R[S]$. \\
\indent Let $B(S) = \{K \subset S| \; |K| < \infty\}$, which can be endowed with the structure of a group under $\cup$, with unit $\emptyset$, which moreover satisfies $x \cup x = x$ for any $x \in B(S)$. Then $R\{S\} \simeq R[B(S)]$, the group ring on $B(S)$, and so we have an $R$-linear decomposition $R\{S\} = \oplus_n R\{S\}_{=n}$, where $R\{S\}_{=n} = \oplus_{\underset{|K|=n}{K\subset S}} R$. Let $\iota_S: S\to B(S)$ be the set-theoretic inclusion viewing elements of $S$ as one-element subsets of $S$, and let $R(\iota_S): \oplus_S R \to R\{S\}$ be the corresponding identification of $\oplus_S R$ as the degree $1$ elements of $R\{S\}$. \\
\indent We remark that $B(S \amalg T) = B(S) \times B(T)$ and therefore $R\{S \amalg T\} \simeq R\{S\} \otimes_R R\{T\}$.
\item[(v)] For an (injective) $A$-linear map $f: M \to N$ of $A$-modules, we may form an exact sequence 
\[0 \to M \to N \to \coker{f} \to 0.\]
This corresponds to a class in $\Ext^1_A(\coker{f}, M)$, which we denote by $\cl{f}$.\\
\indent For a ring map $f: A \to B$ that is faithfully flat, we say that $f$ has \textit{descendability exponent $n \in \bb{N}$} 
when $n$ is the minimal number such that $\text{cl}(f)^{\otimes_A n} \neq 0$ but $\text{cl}(f)^{\otimes_A n+1} = 0$. 
When $n$ doesn't exist, we say that $f$ is \textit{not descendable}, and has descendability exponent $\infty$.\footnote{In the literature one will typically define the exponent of descendability to be the minimal $n$ such that $\text{cl}(f)^{\otimes_A n} = 0$, 
however, due to the relationship between descendability and projective dimension of flat modules, we find this definition more convenient for our purposes. 
So our definition of exponent of descendability is always one less than the usual definition.}
\end{enumerate}
For any natural number $n \in \bb{N}$ greater than $1$, let $[n] = \{1,2,...,n\}$ and set $[\infty] = \bb{N}_{> 0}$.

\section{Indivisible Sequences and Descendability}

\begin{definition}\label{definition:indivisible}
Let $R$ be a commutative ring. For a set $S$, and a set-theoretic map
\[\Psi: S \to R\]
we associate the $R$-linear map
\[\Psi_{R,S}: \oplus_{S} R \to R \oplus \bigoplus_{S} R\]
defined by $\Psi_{R,S}(1_s)_{0}= 1$ and $\Psi_{R,S}(1_{s})_{s'} = \delta_{s'=s}$, where we regard the factor not indexed by $S$ on the right-hand side as being in degree $0$.\\
\indent A \textit{$1$-indivisible sequence} is the data $\{S, \Psi: S \to R\}$ where $S$ is a set and $\Psi:S \to R$ a set-theoretic map such that the associated map $\Psi_{R,S}$ is an 
\textit{$R$-universally injective map}, in the sense that $\Psi_{R,S} \otimes_{R,f}R'$ is injective for any ring map $R \over{f}{\to} R'$. Equivalently, $\Psi_{R,S}$ is injective and 
$\coker{\Psi_{R,S}}$ is $R$-flat (as $\text{Tor}^1_R(\coker{\Psi_{R,S}}, R')=0$ for any ring map $R \to R'$).\\
\indent For any $n \in \bb{N} \cup \{\infty\}$, an \textit{$n$-indivisible sequence} is the data $\{\{S_i, \Psi_i:S_i \to R\}\}_{i \in [n]}$ where each $\{S_i, \Psi_i: S_i\to R\}$ 
is a $1$-indivisible sequence, and the elements $\{\Psi_i(s_i)\}_{i \in [n]}$ of $R$ do not generate the unit ideal for any choice of 
elements $s_i \in S_i$. An $\infty$-indivisible sequence of $R$ will simply be called an \textit{indivisible sequence} of $R$.
\end{definition}
While it may not be immediately obvious, there are many examples of rings $R$ with $n$-indivisible sequences. First, there is a general way to produce $n$-indivisible sequences from $1$-indivisible sequences.
\begin{proposition}\label{proposition:indivequiv}
Let $k$ be a commutative ring and $f: R \to R'$ a map of $k$-algebras, and for any $i \in [n]$, let $t_i: R \to R^{\otimes_k n}$ be the inclusion by $1$ into the $i^{\text{th}}$ factor.
\begin{enumerate}
\item[(i)] $\Psi_{R,S} \otimes_{R,f} R' = (f \circ \Psi)_{R',S}$. In particular, if $\{S, \Psi: S \to R\}$ is $1$-indivisible, then so is $\{S, f \circ \Psi: S \to R'\}$.
If $f$ is faithfully flat, then $\{\{S_i, \Psi: S \to R\}\}$ is $n$-indivisible if and only if $\{\{S_i, f \circ \Psi: S \to R'\}\}$ is $n$-indivisible.
\item[(ii)] For each $n \in \bb{N} \cup \{\infty\}$ and $i \in [n]$, consider $1$-indivisible sequences $\{S_i, \Psi_i:S_i \to R\}$ such that for every $s_i \in S_i$, the composed map 
\[k \to R \to R/\Psi_i(s_i)\]
is faithfully flat. Then $\{\{S_i, t_i \circ \Psi_i: S_i \to R^{\otimes_k n}\}\}_{i \in [n]}$ is an $n$-indivisible sequence.
\end{enumerate}
\end{proposition}
\begin{proof}
(i) is standard. For (ii), by (i) the sequences $\{S_i, t_i \circ \Psi_i:S_i \to R^{\otimes_k n}\}$ are $1$-indivisible for each $i \in [n]$. 
Using the K\"unneth spectral sequence applied to the complexes $R \over{\Psi(s_i)}{\to} R$, we can show that:
\[R^{\otimes_{k} n}/(t_1 \circ \Psi_1({s_1}),...,t_n \circ \Psi_n(s_n)) = \bigotimes_{i\in [n], k} R/\Psi_i(s_i).\]
for arbitrary elements $s_i \in S_i$. Under the faithfully flat assumption of the rings on the right, the latter ring is non-zero, hence (ii).
\end{proof}
\begin{example}\label{example:indiv}
Here we will list a few examples of $n$-indivisible sequences.
\begin{enumerate}
\item[(i)] Let $k$ be a commutative ring. Then $\{k, \Psi:k \to k[x]\}$, where $\Psi(a) = x-a$ for each $a \in k$, 
is a $1$-indivisible sequence. Using Proposition~\ref{proposition:indivequiv} we find that 
$\{\{k, \Psi_i: k \to k[x_1,...,x_n]\}\}_{i\in [n]}$, where $\Psi_i(a) = x_i - a$ for every $a \in k$, forms an $n$-indivisible sequence, 
for any $n \in \bb{N} \cup \{\infty\}$. 
\item[(ii)] Let $S$ be a set and $R$ be a ring with a subset $\{e_s\}_{s \in S}$ of orthogonal idempotents. Then the sequence 
$\{S, \Psi:S \to R\}$ defined by $\Psi(s) = 1-e_s$ is $1$-indivisible. Using Proposition~\ref{proposition:indivequiv}, we find that 
$\{\{S, t_i \circ \Psi:S \to R^{\otimes_{\bb{F}_p} n}\}\}_{i \in [n]}$, forms an $n$-indivisible sequence for any 
$n \in \bb{N}\cup \{\infty\}$. We remark that $R^{\otimes_{\bb{F}_p} n}$ is a $p$-boolean ring (see \ref{definition:pboolean}).
\end{enumerate}
Let us now briefly discuss why each of the examples are $1$-indivisible. For (i), we note that
\[\coker{\Psi_{k[x],k}} = \sum_{a \in k} \frac{1}{x-a}k[x] \subset \text{Frac}(k[x]).\]
It's standard to show that $\coker{\Psi_{k[x],k}}$ is therefore flat. Since $k[x]$ is a domain, $\Psi_{k[x],k}$ is clearly injective, 
so $\Psi_{k[x],k}$ is universally injective.\\
\indent For (ii), by (i) Proposition~\ref{proposition:indivequiv}, it suffices to show that $\Psi_{R,S}$ is injective. 
In particular, suppose we have an equation:
\[\Psi_{R,S}(\sum_{i \in [n]} r_i 1_{s_i}) = (\sum_{i \in [n]} r_i) \oplus \bigoplus_{i \in [n]} \Psi(s_i)r_i = 0\]
for a subset of distinct elements $\{s_1,...,s_n\} \subset S$ and $r_i \in R$. Therefore, $r_i \in e_{s_i}R, \; \; \forall i \in [n]$, 
and 
\[\sum_{i \in [n]} r_i = 0,\]
which, after multiplying by $e_{s_i}$, gives that $r_ie_{s_i}=0,$ and therefore $r_i \in (1-e_{s_i})R \Rightarrow r_i=0 \;\; \forall i \in [n]$, as desired. 
\end{example}
The key property of $n$-indivisible sequences is that they give rise to non-vanishing cup-products in $\Ext^n$ groups.
\begin{theorem}\label{theorem:indivisiblecup}
Let $n \in \bb{N}$, $R$ be a ring and $\{\{S_i, \Psi_i: S_i \to R\}\}_{i \in [n]}$ an $n$-indivisible sequence in $R$. Then the class:
\[\bigotimes_{i \in [n],R} (R(\iota_{S_i}) \circ \cl{\Psi_{R,S_i}}) \in \Ext^{n}_{R}(\otimes_{i \in [n],R} \coker{\Psi_{R,S_i}}, R\{\underset{i \in [n]}{\amalg} S_i\})\]
is non-zero as soon as $|S_i| \ge \aleph_{i-1}, \; \forall i \in [n]$.\footnote{Note that:
\[\bigotimes_{i\in [n],R}^n (R(\iota_{S_i}) \circ \cl{\Psi_{R,S_i}}) \neq 0 \Longleftrightarrow \bigotimes_{i\in [n],R} (\cl{\Psi_{R,S_i}}) \neq 0\]
since for each $i \in [n]$, $R({\iota}_{S_i}): \oplus_{S_i} R \to R\{S_i\}$ admits an $R$-linear left-inverse.}
\end{theorem}
The following lemma is helpful for showing that certain $\Ext^n$ classes are non-zero.
\begin{lemma}\label{lemma:extcrit}
Let $A$ be a ring, $P_1, P_2, P_1', P_2', ..., P_r'$ be projective $A$-modules, together with two flat $A$-modules $M$ and $M'$ that form two exact sequences:
\[0 \to P_1 \stackrel{d_1}{\to} P_2 \to M \to 0,\]
\[0 \to P_1' \stackrel{d_2}{\to} ... \to P_r' \to M' \to 0.\]
We may view these as extensions $\eta_1 \in \Ext^1_A(M, P_1)$ and $\eta_2 \in \Ext^{r-1}_A(M',P_1')$, 
and as projective resolutions $P^{\bullet}_1, P^{\bullet}_2$. Then the projective resolution:
\[P^{\bullet}_1 \otimes_A P^{\bullet}_2 \to M \otimes_A M',\]
gives rise to an extension $\Ext^{r}_A(M \otimes_A M', P_1 \otimes_A P_1')$ which is equal to $\eta_1 \otimes_A \eta_2$. 
Moreover, showing this extension is non-zero is equivalent to verifying that 
\[P_1 \otimes P_1' \stackrel{d_1 \otimes 1_{P'_1} \oplus 1_{P_1} \otimes d_2}{\to} P_2 \otimes_A P_1' \oplus P_1 \otimes_A P_2\]
doesn't admit an $A$-linear left inverse.
\end{lemma}
\begin{proof}
For the first claim, the resolutions $P_1^{\bullet}, P_2^{\bullet}$ are cofibrant replacements for $M$ and $M'$, and the 
maps $P_1^{\bullet}[-(r-1)] \to P_1$ and $P_2^{\bullet}[-1] \to P_1'$ correspond to chain maps that are identities on degree 0. 
The tensor product of this map is then just the identity on $P_1 \otimes_A P_1'$ in $P_1^{\bullet} \otimes_A P_2^{\bullet}[-r]$, 
which corresponds to the extension class of $M\otimes_A M'$ by $P_1 \otimes_A P_1'$ given by $P_1^{\bullet} \otimes_A P_2^{\bullet}$. 
The final claim then follows readily, as $\eta_1 \otimes_A \eta_2 = 0$ if and only $\mathrm{id}_{P_1 \otimes_A P_1'}$ is in the image 
of the map 
\[\Hom_A(P_2 \otimes_A P_1' \oplus P_1 \otimes_A P_2, P_1 \otimes_A P_1') \to \Hom_A(P_1 \otimes_A P_1', P_1 \otimes_A P_1')\]
induced by $d_1 \otimes 1_{P'_1} \oplus 1_{P_1} \otimes d_2$.
\end{proof}
\begin{proof}[Proof of Theorem~\ref{theorem:indivisiblecup}]
Using Lemma~\ref{lemma:extcrit}, we may reduce to showing there does not exist an $R$-linear extension $r$ fitting into the 
commutative diagram:
\[\begin{tikzcd}R\{\underset{i \in [n]}{\amalg} S_i\} &[5em]\\[5ex] \underset{{\underset{i \in [n]}{\prod} S_i}}{\bigoplus} R \arrow[r, "\bigoplus_i \Psi_{R,S_i}^n"] \arrow[u,"\underset{i \in [n],R}{\bigotimes} R(\iota_{S_i})"]& \bigoplus_{i} \underset{\underset{j \in [n] \setminus \{i\}}{\prod} S_j}{\bigoplus} R \oplus \underset{\underset{j \in [n]}{\prod}S_j}{\bigoplus}  R \arrow[ul, dashed, swap, "r"]\end{tikzcd}\]
Here, $\Psi^n_{R, S_i}$ is defined as the unique $R$-linear map satisfying:
\[\Psi^n_{R, S_i}(1_s) = 1_{p^n_i(s)} \oplus \Psi_i(s_i) 1_{s}, \; \; \forall s \in \prod_{j \in [n]} S_j.\]
Suppose such an extension $r$ existed. For every $ i \in [n]$, $s' \in \prod_{j \in [n]\setminus \{i\}} S_j$ and 
$s \in \prod_{j \in [n]}S_j$, let
\[r(1_{i,s'}) = f_i(s'), r(1_{i,s}) = g_i(s).\]
Here, $f_i(s'), g_i(s)$ are elements $R\{\amalg_{i \in [n]} S_i\}$, which satisfy the following equation in 
$R\{\amalg_{i \in [n]} S_i\}$: 
\[\sum_i f_i(p^n_i(t)) + \sum_i \Psi_{i}(\text{pr}_i(t)) g_i(t) = \prod_{i \in [n]} R(\iota_{S_i})(1_{\text{pr}_i(t)}), \; \; (*)\]
for every $t \in \prod_{j \in [n]} S_j$.\\
\indent Let 
\[d_n: R\{\amalg_{i \in [n]} S_i\} \to \underset{\underset{i \in [n]}{\prod}S_i}{\bigoplus}  R\]
be the $R$-linear left-inverse to ${\otimes}_{i \in [n],R} R(\iota_{S_i})$. Then for any $t' \in  \prod_{j \in [n]\setminus \{i\}} S_j$, $d_n(f_i(t'))$ 
is a vector in $\bigoplus_{\underset{i \in [n]}{\prod}S_i}  R$, which we view as an element of $\Hom^{\mathrm{fin}}_{\mathrm{Set}}({\prod}_{i \in [n]}S_i , R)$. 
By Proposition~\ref{proposition:setbounds}, there exists an element $s \in \prod_{i \in [n]} S_i$ such that $d_n(f_i(p^n_i(s)))(s)=0, \; \forall i \in [n]$. 
Therefore, applying $d_n$ to $(*)$, we see that:
\[\sum_i \Psi_i(\text{pr}_i(s)) d_n(g_i(s))(s) = d_n(\prod_{i \in [n]} R(\iota_{S_i})(1_{\text{pr}_i(s)}))(s) = 1.\]
But this implies that $\{\Psi_i(\text{pr}_i(s))\}_{i \in [n]}$ generate the unit ideal in $R$, which is a contradiction.
\end{proof}
The following set-theoretic result was used above.
\begin{proposition}\label{proposition:setbounds}
Let $R$ be a ring and $S$ a set. For each $ i\in \{1,...,n\}$, consider functions:
\[v_i: \prod_{j \in [n] \setminus \{i\}} S_j \to \Hom^{\mathrm{fin}}_{\mathrm{Set}}(\prod_{i \in [n]} S_i,R).\]
If $|S_i| \ge \aleph_{i-1}$, then there exists an element $s \in \prod_{i \in [n]} S_i$ such that $v_i(p^n_i(s))(s) = 0$ 
for each $i \in \{1,2,...,n\}$.
\end{proposition}
\begin{proof}
Define $v'_i = v_i(p^n_i(s))(s): \prod_{i \in [n]} S_i \to R$. The functions $v'_i$ have the property that for any 
$s' \in \prod_{j \in  [n] \setminus \{i\}} S_i$, 
\[v'_i \circ t_i^{s'}=v_i(s') \circ t_i^{s'} \in \Hom^{\mathrm{fin}}_{\mathrm{Set}}(S,R).\]
For each $s' \in \prod_{i \in [n-1]} S_i$, define a subset $V_{s'} \subset S_n$ as follows:
\[V_{s'} = \{s \in S_n | \; v'_n(t_n^{s'}(s)) \neq 0 \}.\]
By supposition, $V_{s'}$ is a finite subset of $S_n$. Notice that:
\[|\cup_{s' \in \prod_{i \in [n-1]} S_i} V_{s'}| \le \aleph_0 \aleph_{n-2} < \aleph_{n-1}.\]
Therefore, there is an $s_{n} \in S_n$ such that:
\[v'_n(t_n^{s'}(s_n)) = 0, \; \; \forall s' \in \prod_{i \in [n-1]} S_i.\]
We may repeat this argument for the restricted functions $v'_i \restriction \prod_{j \in [n] \setminus \{i\}} S_j \times s_n$, 
and we eventually end up with elements $s_i \in S_i$ defining an element $s \in \prod_{i \in [n]}S_i$ such that 
$v'_i(s) = v_i(p^n_i(s))(s)=0, \; \; \forall i \in [n]$, as desired.
\end{proof}
\section{Constructing non-descendable ring maps}
Theorem~\ref{theorem:indivisiblecup} allows us to construct non-vanishing cup-products of maps between modules from indivisible sequences.
Given a ring $R$ and a module $M$, there are many possible functorial ways to construct a corresponding algebra, such as $\Sym_R(M)$. 
For any such construction, one would need to be able to carry over the non-vanishing results that we've already proved. The next sequence of results
give a general conditions for which this is possible.
\begin{proposition}\label{proposition:modalg}
Let $n \in \bb{N} \cup \{\infty\}$ and $R$ a commutative ring. Suppose for each $i \in [n]$, we have universally injective $R$-linear maps
 $\Psi_i: M_i \to N_i$, $R$-linear maps $d_i: M_i \to T_i$ with $T_i$ having the structure of a flat $R$-algebra with unit $R \to T_i$ 
 admitting an $R$-linear left-inverse, 
 and $R$-algebra maps 
 $T_i \over{g_i}{\leftarrow} A_i \over{f_i}{\to} B_i$ fitting into a commutative diagram:
\[\begin{tikzcd}
T_i&\\
A_i \arrow[u, "g_i", swap] \arrow[r, "f_i"] & B_i \\
M_i \arrow[u] \arrow[uu, bend left, "d_i"] \arrow[r, "\Psi_{i}"] & N_i \arrow[u]
\end{tikzcd}\]
with each $f_i$ faithfully flat. Suppose that for every finite subset $S' \subset [n]$,
\[\bigotimes_{s'\in S'} (d_{s'} \circ \cl{\Psi_{s'}}) \neq 0.\]
Then the faithfully flat ring map:
\[\begin{tikzcd}\bigotimes_{i \in [n], R} A_i \arrow[r, "\otimes_{i \in [n],R} f_i"]& \bigotimes_{i \in [n],R} B_i\end{tikzcd}\]
has descendability exponent $n$.
\end{proposition}
\begin{proof}
Let $h_i: T_i \to U_i$ be the base-change of $f_i$ along $g_i$. By base-change along 
$\otimes_{i \in [n], R} g_i : \otimes_{i \in [n], R} A_i \to \otimes_{i \in [n], R} T_i$, 
it suffices to show that $\otimes_{i \in [n], R} h_i$ has descendability exponent $n$. But this follows from
Proposition~\ref{proposition:cup} (ii).
\end{proof}
The following proposition was used above.
\begin{proposition}\label{proposition:cup}
Let $A \to B$ be a flat map of rings.
\begin{enumerate}
\item[(i)] For any natural number $n$, let $M_i$ be $B$-modules in $D(B)$ for $i \in [n]$, $\mu_{B/A}^{k}: B^{\otimes_A k+1} \to B$ 
be the multiplication map realizing $B$ as an $A$-algebra, and $\R_{B|A}:D(B) \to D(A)$ be the restriction of scalars functor. 
Consider elements $\eta_i \in \Ext^1_B(M_i, B)$, then:
\[\mu^{n-1}_{B/A} \circ \otimesl_{i\in [n],A} \R_{B|A}(\eta_i) \neq 0 \implies \otimesl_{i\in [n], B} \eta_i \neq 0.\]
\item[(ii)] For every $s \in \bb{N}$, suppose we have flat $A$-algebras $B_s$, faithfully flat maps $f_s: B_s \to C_s$ of $A$-algebras 
and universally injective maps $\Psi_s: M_s \to N_s$ of flat $A$-modules fitting into a commutative diagram:
\[\begin{tikzcd}
B_s \arrow[r, "f_s"] & C_s \\
M_s \arrow[u, "h_s"] \arrow[r, "\Psi_s"] & N_s \arrow[u]
\end{tikzcd}\]
Assume the unit $A \to B_s$ admits an $A$-linear left-inverse $k_s: B_s \to A$. Then for any subset $S \subset \bb{N}$, and any finite subset $S' \subset S$:
\[\bigotimes_{s' \in S',A} (h_{s'} \circ \cl{\Psi_{s'} }) \neq 0 \implies \underset{s' \in S', \otimes_{s \in S,A}B_s}{\bigotimes} \cl{ \otimes_{s \in S,A} f_s} \neq 0.\]
\end{enumerate}
\end{proposition}
\begin{proof}
For (i), we remark more generally that for any pair of morphisms $f: N_1 \to K_1, g: N_2 \to K_2$ in $D(B)$, 
the bar construction of \cite[Chapter 4.4.2]{LurieHA} gives a homotopy coherent diagram:
\[\begin{tikzcd}
N_1 \otimesl_B K_1 \arrow[r, "f \otimesl_B g"] & N_2 \otimesl_B K_2 \\
N_1 \otimesl_A K_1 \arrow[u] \arrow[r, "f \otimesl_A g"] & N_2 \otimesl_A K_2 \arrow[u]
\end{tikzcd}\]
Applying this to the $\eta_i$ repeatedly, we end up with a homotopy commutative diagram:
\[\begin{tikzcd}
\otimesl_{i\in [n], B} M_i[-1] \arrow[r, "\otimesl_{i\in [n],B} \eta_i"] & B^{\otimesl_B n} \arrow[r, "\simeq"] & B \\
\otimesl_{i\in [n], A} M_i[-1] \arrow[u, "\varphi"] \arrow[r,swap, "\otimesl_{i\in [n],B} \R_{B|A}(\eta_i)", {yshift=-3pt}] & B^{\otimesl_A n} \arrow[u] \arrow[ru, swap,"\mu^{n-1}_{B/A}"] &
\end{tikzcd}\]
We therefore see that:
\[\otimesl_{i\in [n],B} \eta_i =0 \implies (\otimesl_{i\in [n],B} \eta_i) \circ \varphi =0 \implies \mu^{n-1}_{B/A} \circ \otimesl_{i\in [n],A} \R_{B|A}(\eta_i) = 0.\]
For (ii), let $B_{S} = \otimes_{s \in S} B_{s}$,  $i_s: B_s \to B_S$ be the inclusion in the $s^{\text{th}}$-factor, and $g_s: \coker{f_s} \to \coker{\otimes_{s \in S} f_s}$ the induced map. Note that:
\[\bigotimes_{s' \in S',A} (h_{s'} \circ \cl{\Psi_{s'} }) \neq 0 \implies \bigotimes_{s' \in S',A} \cl{f_{s'}} \neq 0.\]
By (i) it suffices to show that:
\[\mu^{|S'|-1}_{B_{S}/A} \circ \bigotimes_{s' \in S', A} \R_{B_{S}|A}(\cl{\otimes_{s \in S} f_s}) \neq 0.\]
This follows from the following calculation:
\begin{align*}k_{S'}\! \circ \!\mu^{|S'|-1}_{B_{S}/A} \!\circ \!\bigotimes_{s' \in S', A} \R_{B_{S}|A}(\cl{\otimes_{s \in S} f_s}) \! \circ \!\bigotimes_{s' \in S', A} g_{s'}\! &= \! k_{S'} \!\circ \! \mu^{|S'|-1}_{B_{S}/A} \circ \!\bigotimes_{s' \in S', A} (\R_{B_{S}|A}(\cl{\otimes_{s \in S} f_s}) \circ g_{s'})\!\\
& = k_{S'} \circ \mu^{|S'|-1}_{B_{S}/A} \circ \bigotimes_{s' \in S'} (i_{s'} \circ \cl{f_{s'}})\\
&= \left (k_{S'} \circ \mu^{|S'|-1}_{B_{S}/A} \circ \bigotimes_{s' \in S'} i_{s'}\right) \circ \bigotimes_{s' \in S'} \cl{f_{s'}}\\
&= \bigotimes_{s' \in S'} \cl{f_{s'}} \\
&\neq 0,
\end{align*}
where $k_{S'}: \otimes_{s \in S} R \to \otimes_{s \in S \setminus S'}R$ is defined as the tensor product of 
the $A$-linear maps $\text{id}_{B_{s''}}: B_{s''} \to B_{s''}$ and $k_{s'}$ for each $s'' \in S \setminus S'$ and $s' \in S'$. \\
\end{proof}
\begin{theorem}\label{theorem:indivalg}
Let $n \in \bb{N} \cup \{\infty\}$, and $\{\{S_i, \Psi_i:S_i \to R\}\}$ be a $n$-indivisible sequence in $R$. For all $i \in [n]$, suppose we have $R$-algebra maps $R\{S_i\} \over{g_i}{\leftarrow} A_i \over{f_i}{\to} B_i$ fitting into a commutative diagram:
\[\begin{tikzcd}
R\{S_i\}&\\
A_i \arrow[u, "g_i", swap] \arrow[r, "f_i"] & B_i \\
\oplus_{S_i} R \arrow[u] \arrow[uu, bend left, "R(\iota_{S_i})"] \arrow[r, "\Psi_{R,S_i}"] & R \oplus \bigoplus_{S_i} R \arrow[u]
\end{tikzcd}\]
and each $f_i$ are faithfully flat. If for all $i \in [n]$, $|S_i| \ge \aleph_{i-1}$, then the faithfully flat ring map:
\[\begin{tikzcd}\bigotimes_{i \in [n], R} A_i \arrow[r, "\otimes_{i \in [n],R} f_i"]& \bigotimes_{i \in [n],R} B_i\end{tikzcd}\]
has descendability exponent $n$.
\end{theorem}
\begin{proof}
This follows directly from Theorem~\ref{theorem:indivisiblecup} and Proposition~\ref{proposition:modalg}.
\end{proof}
Before stating our main result, we need to recall a few definitions.
\begin{definition}\label{definition:pboolean}
A ring $R$ is said to be \textit{p-boolean} if $R$ is in characteristic $p$ and for every $r \in R, r^p=r$.
\end{definition}
The following is a discrete version of a derived result appearing in \cite[Theorem 4.6]{pbool}.
\begin{theorem}\label{theorem:p-booleanmonadic}
Let $R$ be a $p$-boolean ring. Let $\mathrm{CAlg}_{R}$ (resp. $\mathrm{CAlg}^{\mathrm{perf}}_{R}$, $\mathrm{CAlg}^{\phi=1}_{R}$) denote the category of (commutative) $R$-algebras (resp. perfect $R$-algebras, $p$-boolean $R$-algebras). Each of the forgetful functors:
\[\mathrm{CAlg}^{\phi=1}_{R} \to \mathrm{CAlg}^{\mathrm{perf}}_{R} \to \mathrm{CAlg}_{R} \to \text{Mod}_R\]
are monadic. In particular, the left adjoint of this forgetful functor is the functor $F_R^p: \text{Mod}_R \to \mathrm{CAlg}^{\phi=1}_{R}$ defined on objects by the rule:
\[M \mapsto \Sym_R(M) \mapsto \Sym_R(M)_{\mathrm{perf}} \mapsto \coeq(\phi, \mathrm{id}: \Sym_R(M)_{\mathrm{perf}} \to \Sym_R(M)_{\mathrm{perf}}),\]
where $(-)_{\mathrm{perf}}$ is the colimit perfection functor of commutative algebras that on objects is defined as:
\[A \mapsto A_{\mathrm{perf}} = \colim(A \stackrel{\phi}{\to} A \stackrel{\phi}{\to} A \stackrel{\phi}{\to} ...),\]
and $\phi$ is the Frobenius endomorphism of $A$.
\end{theorem}
\begin{proposition}\label{proposition:symformula}
Let $R$ be a ring and $f: M \to N$ an injective map of flat $R$-modules such that $K=\coker{f}$ is $R$-flat. Then 
\[\Sym_R(f): \Sym_R(M) \to \Sym_R(N)\]
is a faithfully flat map of $R$-algebras. Furthermore, if $R$ is $p$-boolean, then $F_R^p(f)$ is a faithfully flat map of $p$-boolean rings.
\end{proposition}
\begin{proof}
By Lazard's theorem, $K=\colim_{i\in I} K_i$ where $I$ is a filtered indexing category and $K_i$ is a finite free $R$-module. If we define $N_i = N \times_{K} K_i$, then
\[\colim_{i \in I} N_i = \colim_{i \in I} N \times_{K} K_i = N \times_{K} \colim_{i \in I} K_i = N,\]
and we have an exact sequence
\[0 \to M \over{f_i}{\to} N_i \to K_i \to 0,\]
which is split. Therefore:
\[\Sym_R(N_i) = \Sym_R(M) \otimes_R \Sym_R(K_i)\]
and if $R$ is $p$-boolean:
\[F^p_R(N_i) = F_R^p(M) \otimes_R F_R^p(K_i).\]
Note that as $K_i$ are finite free, the unit $R \to \Sym_R(K_i)$ is faithfully flat, and if $R$ is futher assumed to be $p$-boolean, then likewise for $R \to F^p_R(K_i)$.\footnote{If $K_i =R^{\oplus n}$, then $F^p_R(K_i) = R[X_1,...,X_n]/(X_1^p-X_1, ..., X_n^p-X_n)$.} Hence, the map:
\[\Sym_R(M) \over{\Sym_R(f_i)}{\to} \Sym_R(N_i)\]
is faithfully flat, and if $R$ is $p$-boolean, then:
\[F^p_R(M) \over{F^p_R(f_i)}{\to} F^p_R(N_i)\]
is also faithfully flat. The result follows by noting $\Sym_R(f) = \colim_{i\in I}\Sym_R(f_i)$ and if $R$ is $p$-boolean, $F^p_R(f) = \colim_{i \in I} F^p_R(f_i)$.
\end{proof}
Our main result on descendability of faithfully flat ring maps can then be stated as follows.
\begin{theorem}\label{theorem:mainindiv}
Let $n \in \bb{N} \cup \{\infty\}$, $R$ be a commutative ring and $\{\{S_i, \Psi_i: S_i \to R\}\}_{i \in [n]}$ an $n$-indivisible sequence, and $|S_i| \ge \aleph_{i-1}$. Then the ring map 
\[\bigotimes_{i=1, R}^n \Sym_R(\Psi_{R,S_i}): \bigotimes_{i=1, R}^n \Sym_R(\oplus_{S_i} R) \to \bigotimes_{i=1, R}^n \Sym_R(R \oplus \bigoplus_{S_i} R),\]
is faithfully flat with exponent of descendability $n$. In fact, the ring map:
\[\begin{tikzcd}R\{\underset{i=1}{\amalg^n}S_i\}\arrow[r, "\underset{i=1,R}{\bigotimes^n} \Sym_R(\Psi_{R,S_i}) \underset{\underset{i=1,R}{\bigotimes^n} \Sym_R(\oplus_{S_i} R)}{\bigotimes}  R\{\underset{i=1}{\amalg^n}S_i\}"] &[8em] \underset{i=1,R}{\bigotimes^n} \Sym_R(R \oplus \bigoplus_{S_i} R)  \underset{\underset{i=1,R}{\bigotimes^n} \Sym_R(\oplus_{S_i} R)}{\bigotimes}R\{\underset{i=1}{\amalg^n}S_i\}\end{tikzcd},\]
also has descendability exponent $n$ too.\\
\indent If $R$ is $p$-boolean, then the ring map:
\[\bigotimes_{i=1, R}^n F^p_R(\Psi_{R,S_i}): \bigotimes_{i=1, R}^n F^p_R(\oplus_{S_i} R) \to \bigotimes_{i=1, R}^n F^p_R(R \oplus \bigoplus_{S_i} R),\]
is a faithfully flat of $p$-boolean rings of descendablity exponent $n$.
\end{theorem}
\begin{proof}
After noting commutative diagrams:
\[\begin{tikzcd} \oplus_{S_i} R \arrow[r] \arrow[rd]\arrow[rr, bend left, "R(\iota_{S_i})"]&  \Sym_R(\oplus_{S_i} R) \arrow[r]& R\{S_i\}\\
& F^p_R(\oplus_{S_i} R) \arrow[ru]&
\end{tikzcd}\]
the result follows from combining Proposition~\ref{proposition:symformula} and Theorem~\ref{theorem:indivalg}.
\end{proof}
\begin{corollary}\label{corollary:mainindiv}
For any $n \in \bb{N} \cup \{\infty\}$:
\begin{enumerate}
\item[(i)] There exists a faithfully flat ring map $A \to A'$ between $\aleph_{n-1}$-Noetherian rings, with $A$ of Krull-dimension $n$, which has descendability exponent $n$. 
\item[(ii)] There exists a faithfully flat ring map $A \to A'$ between $\text{min}(\beth_{n-1}, 2^{\aleph_{n-1}})$-countable $p$-boolean rings which has descendability exponent $n$.
\end{enumerate}
\end{corollary}
\begin{proof}
For (i), we may use the first example of Example~\ref{example:indiv}. In particular, if we set $R = k[x_1,x_2,...,x_n]$, 
then from Theorem~\ref{theorem:mainindiv} we may conclude that there is a faithfully flat map:
\[k\{k^{\amalg n}\}[x_1,...,x_n] = R\{k^{\amalg n}\} \to R\{k_{+}^{\amalg n}\}= A'\]
which has descendability exponent $n$ when $|k| \ge \aleph_{n-1}$, where the set $k_+ = \{*\}\amalg k$. 
Furthermore, it's clear that $R\{k^{\amalg n}\}$ has Krull-dimension $n$ (its irreducible components are isomorphic to R), 
and that the rings are $\aleph_{n-1}$-Noetherian (since they're $\aleph_n$-countable).\\
\indent For (ii), by using Theorem~\ref{theorem:mainindiv} and Proposition~\ref{proposition:indivequiv}, 
it suffices to find a $1$-indivisible sequence $\Psi: S \to R$ where the image of $S$ is a set of orthogonal idempotents in $R$ 
with $|S|=\aleph_{n-1}$ and $R$ is a $p$-boolean ring with $|R|=2^{\aleph_{n-1}}$ or $\beth_{n-1}$.\\
\indent If $S$ is an arbitrary set with $|S|=\aleph_{n-1}$, then $R= \text{Hom}_{\mathrm{Set}}(S, \bb{F}_p)$ is a $p$-boolean ring with $|R|=2^{\aleph_{n-1}}$ containing the set $\{\delta_{t=s}\}_{s \in S}\subset R$ of orthogonal idempotents.
\begin{claim}
There exists a $p$-boolean ring $R$ with $|R| = \beth_n$ and a subset $S \subset R$ consisting of orthogonal idempotents with $|S|=\beth_n$.
\end{claim}
\begin{proof}
We will prove the claim by induction on $n$. In the case $n=0$, we remark that the $p$-boolean ring $R_0=F_{\bb{F}_p}^p(R_0^{\oplus \bb{N}})=\bb{F}_p[X_1,X_2,...]/(X_1^p-X_1,X_2^p-X_2,...)$ admits a countably infinite subset of orthogonal idempotents given by a sequence $S_0=\{a_n\}_{n \in \bb{N}}$ of elements of $R$ defined by $a_n = (1-X^{p-1}_n) \prod_{0 \le i \le n-1} X^{p-1}_{i}$. The sequence $\{S_0, a: S_0 \to R_0\}$ solves the $n=0$ base-case.\\
\indent Suppose by induction we have a ring $R_{n-1}$ such that $|R_{n-1}| = \beth_{n-1}$ and that there is a subset $S_{n-1} \subset R_{n-1}$ consisting of orthogonal idempotents with $|S_{n-1}| = \beth_{n-1}$. Consider the ring:
\[R_n=\prod_{S_{n-1}}R_{n-1}.\]
We note $|R_n| = \beth_{n-1}^{\beth_{n-1}} = \beth_n$, and the set
\[S_n = \prod_{S_{n-1}} S_{n-1} \subset R_n\]
consists of orthogonal idempotents and $|S_n|=\beth_{n-1}^{\beth_{n-1}}=\beth_n$, as desired.
\end{proof}
\end{proof}
In Corollary~\ref{corollary:mainindiv} (ii), our cardinality bounds are optimal under the assumption of the Generalized Continuum Hypothesis (GCH). 
Indeed, if $f: A \to A'$ is a faithfully flat map of $p$-boolean, then we know that if $A$ is $\aleph_{n-1}$-counterable, 
then $f$ has descendability exponent at most $n$ by Theorem~\ref{thm:gj0}, so we would prefer to have upgraded our cardinality bounds from
$\text{min}(\beth_{n-1}, 2^{\aleph_{n-1}})$-countable to $\aleph_{n-1}$-countable.\\
\indent On the other hand, (i) is optimal; our ring $A$ is $\aleph_{n-1}$-Noetherian of Krull-dimension $n$, so again by Theorem~\ref{thm:gj0} 
any faithfully flat ring map out of it has descendability exponent at most $n$. However, the ring is quite pathological, since $A$ contains
a large collection of idempotents. In the next section, we will see how to construct more `geometric' examples of rings with faithfully flat maps of high descendability exponent.
\section{The case of an infinite polynomial ring}
The goal of this section is to construct a faithfully flat 
algebra over the infinite polynomial ring on an algebraically closed field $k[x_1,x_2,...]$ that is not descendable. The argument
will essentially boil down to a refinement of Theorem~\ref{theorem:indivisiblecup}, 
where we will need to construct a function $f:k \to k$ that is not polynomial on any infinite subset of $k$.\\
\indent We begin by recalling the following form of Lagrange's interpolation theorem.
\begin{theorem}[Lagrange interpolation]
Let $k$ be a field and $f \in k[x]$ a polynomial of degree $\le n$. Then there exists a polynomial $P_n \in \bb{Z}[x_1,...,x_n]$ such that for $z \in k^{\times n+1}$ we have:
\[\sum_{i=1}^{n+1} (-1)^iP_n(p^{n+1}_i(z))f(\text{pr}_i(z)) = 0.\]
Explicitly, $P = \prod_{ 1 \le i < j \le n} (x_j - x_i)$.
\end{theorem}
\begin{definition}
Let $k$ be a field. For $z \in k^{\times n+1}$, we may define a $k$-linear operator 
\[\nabla^z_i: \Hom(k^{\times m}, k) \to \Hom(k^{\times m-1},k)\]
as follows: For any function $f \in  \Hom(k^{\times m}, k)$, set:
\[\forall s \in k^{\times m-1}: \; \;  \nabla^{z}_i(f)(s) = \sum_{j=1}^{n+1} (-1)^j P_n(p^{n+1}_j(z)) f(t^{s}_i(\text{pr}_j(z))).\]
For any two $z \in k^{\times n+1}, z' \in k^{\times n'+1}$, one can verify that:
\[\nabla^z_i \nabla^{z'}_j = \nabla^{z'}_j \nabla^{z}_i\]
for $i \neq j$. Furthermore, if $f(x_1,...,x_m) = \prod_i f_i(x_i)$, and $z^i \in \Hom(k^{\times n_i}, k)$ for $i\in \{1,2,...,m\}$, 
\[(\nabla^{z^m}_m \circ \nabla^{z^{m-1}}_{m-1} \circ ... \circ \nabla^{z^1}_{1})(f) = \prod_i \nabla^{z^i} f_i.\]
\end{definition}
\begin{corollary}\label{corollary:interpol}
Let $f: \Omega \to k$ be function defined on a countably infinite subset $\Omega \subset k$. 
Then: $\nabla^{z} f = 0,\;  \forall z \in \Omega^{\times n+1}$ if and only if $f$ is a polynomial of degree $\le n$.
\end{corollary}
Corollary~\ref{corollary:interpol} provides a convenient criterion for checking if a function is polynomial on a given subset.
Using this, we can now construct functions on a field that are not polynomial on any infinite subset.
\begin{proposition}\label{proposition:notpoly}
For any field $k$, there exists an algebraically closed field extension $k \subset E(k)$, and a function $f: E(k) \to E(k)$ 
such that there is no countably infinite subset $S \subset E(k)$ such that $f$, when restricted to $S$, is a polynomial. 
If $|k| = \infty,$ then $|E(k)| = |k|$.
\end{proposition}
\begin{proof}
We define a sequence of field extensions $0=k_0\subset k=k_1 \subset k_2 \subset ...$ and functions $f_{n}:k_n \to k_{n+1}$ inductively. For the fields, set 
\[k_{n+1} = \overline{k_{n}(k_{n} \setminus k_{n-1})}\]
That is, $k_{n+1}$ is the algebraic closure of the field of fractions of the free polynomial ring $k_{n}[X_s]$ indexed by elements $s \in k_{n} \setminus k_{n-1}$. For the functions, we take:
\[\forall s \in k_{n} \setminus k_{n-1}: \; \; f_n(s) = X_s \in k_{n+1}\]
\[\forall s \in k_{n-1}: \; \; f_n(s) = f_{n-1}(s) \in k_{n} \subset k_{n+1}.\]
Set $E(k) = \underset{n}{\cup} k_n$, and $f=\colim_{n} f_n$, which is a map $E(k) \to E(k)$. \\
\indent We claim that $f$ is not polynomial on any infinite subset $S \subset E(k)$. Suppose that there exists a polynomial $p$ of degree $d$ such that $f|_{S} = p|_{S}$. There are two cases; there exists a $t \in \bb{N}$ such that $|S \cap k_t\setminus k_{t-1}| \ge d+1$, or for any $n \in \bb{N}$, $S \cap (k \setminus k_n) \neq 0$.\\
\indent In the first case, there exists an element $z \in (k_t \setminus k_{t-1})^{d+1}$ such that:
\[\nabla^z f = \sum_i (-1)^iP_d(p^{d+1}_i(z))f(\text{pr}_i(z))= 0.\]
As $P_d(p^{d+1}_i(z)) \in k(z_1,...,z_{n+1})$ for every $i$, this is a contradiction by construction.\\
\indent In the second case, take an $N \in \bb{N}$ such that the coefficients of $p$ are contained in $k_N$, and an $s \in S \cap (k \setminus k_N)$. Suppose that $s \in k_n\setminus k_{n-1}$. By supposition, $f(s) = p(s) \in k_n$, which is again a contradiction by construction.
\end{proof}
Later, we will need some results of infinite Ramsey theory, which originated from the paper \cite{ErdosRado}. Before stating the main theorem in a more convenient language, let us introduce some notation. 
\begin{construction}
Define the \textit{beth numbers} $\beth_n$ by the following formula:
\[\beth_0 = \aleph_0, \; \beth_n = 2^{\beth_{n-1}}.\]
\end{construction}
Note that $\beth_{n} \ge \aleph_{n}$, and equality holds if the generalized continuum hypothesis is assumed.
 For any set $S$, define $[S]^r=\{K \subset S: \; |K|=r\}$, for any cardinal $\alpha$, define $\alpha_{+}$ to be its 
 successor, and define $\beth_{\omega}=\bigcup_{\beta < \omega} \beth_{\beta}$ where $\omega$ is the first limit ordinal.
\begin{theorem}[Theorem 7.3 \cite{Higherinfinite}]\label{theorem:erdosrado}
For any set $S$ of size $\beth_{r+}$, and any set map:
\[f: [S]^{r+1} \to \bb{N}\]
There exists a set $S' \subset S$ with $|S'|=\aleph_{1}$ and an integer $n \in \bb{N}$ such that $[S']^{r+1} \subset f^{-1}(n)$.
\end{theorem}
\begin{remark}\label{remark:boxuseful}
Theorem~\ref{theorem:erdosrado} also implies the following weaker statement: For a set $S$ of cardinality $\beth_{r+}$, and a set theoretic map
\[f: S^{\times r+1} \to \bb{N}\]
such that $f(s) = f(\sigma(s))$ for every $s \in S^{\times r+1}$ and $\sigma \in \mathrm{S}_{r+1}$, the symmetric group on $r+1$ elements, there exists an $n \in \bb{N}$ and countably infinite subsets $S_1, ..., S_{r+1}$ of $S$ such that $ \prod_{i=1}^{r+1} S_i \subset f^{-1}(n)$.
\end{remark}
Armed with these tools, we may now state and prove our refinement of Theorem~\ref{theorem:indivisiblecup}.
\begin{theorem}\label{theorem:countermod}
For $m \in \bb{N} \cup \{\infty\}$, let $F$ be a field of cardinality $\beth_{m-1+}$, and $k=E(F)$ (see Proposition~\ref{proposition:notpoly}). 
Let $R_m=k[x_1,...,x_m]$, and for any $n < m+1$, let $R_n=k[x_1,...,x_n]$ and $\{\{S_i,\Psi_{i}: k \to R_n\}\}$ be the $n$-indivisible sequence 
of Example~\ref{example:indiv} (i) (where $S_i=k$). Let $i:R_n \to R_m$ be the canonical inclusion, and $\{\{S_i,i \circ \Psi_{i}: k \to R_m\}\}$ be 
the induced $n$-indivisible sequence in $R_m$ (see Proposition~\ref{proposition:indivequiv}).
Then there exists an $k[x]$-linear map $h: \oplus_{k} k[x] \to k[x]$ such that:
\[\bigotimes_{i\in [n], R_m} (h \otimes_{k[x], t_i} R_m \circ \cl{\Psi_{R_m,S_i}}) \neq 0 \in \Ext^n_{R_m}(\otimes_{i\in [n], R_m} \coker(\Psi_{R_m,S_i}), R_m),\]
where $t_i: k[x] \to R_m$ is the $k$-algebra map sending $x$ to $x_i$.
\end{theorem}
\begin{proof}
The map $R_n \to R_m$ admits an $R_n$-linear left-inverse, and since 
\[\cl{\Psi_{R_m,S_i}} = \cl{\Psi_{R_n,S_i}} \otimes_{R_n} R_n, \]
it suffices to show the claim for $m=n$. By Lemma~\ref{lemma:extcrit}, we need to show that there is no $R_n$-linear extension $r$ fitting into the diagram:
\[\begin{tikzcd} R_n & \\
\bigoplus_{k^{\times n}} R_n \arrow[r, "\oplus_i \Psi_{R_n,S_i}"] \arrow[u, "h^{\otimes_k n}"] & \bigoplus_{i} \bigoplus_{k^{\times n-1}} R_n \oplus \bigoplus_{k^{\times n}} R_n \arrow[ul, swap, "r", dashed]\end{tikzcd}\]
We shall assume such an $r$ exists and derive a contradiction.\\
\indent Let $H: k \to k$ be a function that is not polynomial on any countably infinite subset of $k$ (see Proposition~\ref{proposition:notpoly}). Set:
\[h(1_{s}) = H(s) \in k[x],\]
for every $s \in k$. Further, for every $s' \in k^{\times n-1}$ and $s \in k^{\times n}$, let
\[r(1_{i,s'}) = f_i(s'), r(1_{i,s}) = g_i(s).\]
Under the identification $R_n = k[x_1,...,x_n]$, $f_i(s'), g_i(s)$ are polynomials subject to the equation:
\[\sum^n_{i=1} f_i(p^n_i(t))(x_1,...,x_n) + \sum^n_{i=1} (x_i - \text{pr}_i(t)) g_i(t)(x_1,...,x_n) = \prod^n_{i=1} h(1_{\text{pr}_i(t)})(x_i).\]
for every $t \in k^{\times n}$. In particular, we have:
\[\sum^n_{i=1} f_i(p^n_i(t))(\text{pr}_1(t),...,\text{pr}_n(t)) = \prod^n_{i=1} h(1_{\text{pr}_i(t)})(\text{pr}_i(t)). \; \; (*)\]
\begin{claim}
There exists a countably infinite subsets $S_1, ..., S_n$ of $k$, such that for every $\omega \in \Omega =\prod^n_i S_i$, the polynomials $f_i(p^n_i(\omega))$ have (total) degree $\le m$ for some $m \in \bb{N}$.
\end{claim}
\begin{proof}
Consider the function:
\[c: k^{\times n} \to \bb{N}\]
which takes a point $t \in k^{\times n}$ to $\underset{i\in\{1,2,...,n\}}{\text{max}}(\text{deg}(f_i(p^n_i(t)))$. By Remark~\ref{remark:boxuseful}, there exists a $m \in \bb{N}$ and countably infinite subsets $S_1, ..., S_n$ such that $ \prod_{i=1}^{n} S_i \subset f^{-1}(m)$. Set $\Omega = \prod_{i=1}^n S_i$ to get the claim.
\end{proof}
For each $i \in \{1,2,...,n\}$, let us take points $z^i \in S_i^{\times m+1}$. Note that for all $i \in \{1,2,...,n\}$, we have functions $\tilde{f}_i, \tilde{h} \in \Hom_{\mathrm{Set}}(k^{\times n},k)$ defined as follows:
\[\forall t \in k^{\times n}: \; \; \tilde{f}_i(t) = f_i(p^n_i(t))(\text{pr}_1(t),...,\text{pr}_n(t))\]
\[\forall t \in k^{\times n}: \; \; \tilde{h}(t) = \prod_{i=1}^n H(\text{pr}_i(t)).\]
We see that $(*)$ may be re-written as an equality
\[\sum^n_{i=1} \tilde{f}_i = \tilde{h},\]
of functions in $\Hom_{\mathrm{Set}}(k^{\times n}, k)$. For any $s \in \prod_{j \in \{1,2,...,n\} \setminus \{i\}} S_j$, we see
\[\tilde{f}_i \circ t^{s}_i: k \to k\]
is a polynomial of degree $\le m$. Therefore:
\[\nabla^{z^i}_i \tilde{f}_i = 0,\]
for all $i \in \{1,2,...,n\}$. Therefore:
\[ \prod_i^n \nabla^{z^i} H = (\nabla^{z^n}_n \circ \nabla^{z^{n-1}}_{n-1} \circ ... \circ \nabla^{z^1}_1)( \tilde{h}) = (\nabla^{z^n}_n \circ \nabla^{z^{n-1}}_{n-1} \circ ... \circ \nabla^{z^1}_1)(\sum_{i=1}^n \tilde{f}_i)=0.\]
This implies that there exists an $i \in \{1,2,...,n\}$ such that $\nabla^{z^i} H = 0$ for any such $z^i \in S_i^{\times m+1} \subset k^{m+1}$. But this would imply that $H$ is a polynomial of degree $m$ on the infinite subset $S_i \subset k$ (due to Corollary~\ref{corollary:interpol}), which is a contradiction.
\end{proof}
\begin{theorem}\label{theorem:maincounter}
$F$ a field with $|F| = \beth_{n-1+}$, and $k=E(F)$. 
For any $n \in \bb{N} \cup \{\infty\}$, there exists a faithfully flat ring map $k[x_1,...,x_n] \to A_{n}$ 
whose descendability exponent is $n$. In particular, setting $n=\infty$, we obtain a faithfully flat ring map $k[x_1,x_2,...] \to A_{\infty}$ 
that is not descendable.
\end{theorem}
\begin{proof}
For $n \in \bb{N} \cup \{\infty\}$, let $R=k[x_1,...,x_n]$, let $\{\{S_i, \Psi_{i}: k \to A\}\}$ be the $n$-indivisible sequence of Example~\ref{example:indiv} (i) (where $S_i=k$). 
Let $h: \oplus_{k} k[x] \to k[x]$ be the map constructed in Theorem~\ref{theorem:countermod}, and consider the commutative diagram:
\[
\begin{tikzcd}[column sep=6em, row sep=4em]
R \arrow[r] & A_n \\
\mathrm{Sym}_{R}(\oplus_{S_i} R)
  \arrow[u, "g_i"', swap]
  \arrow[r, "\mathrm{Sym}_{R}(\Psi_{R,S_i})"]
& \mathrm{Sym}_{R}(R \oplus \bigoplus_{S_i} R)
  \arrow[u] \\
\oplus_{S_i} R
  \arrow[u]
  \arrow[uu, bend left=50, looseness=1.4, "h^{\otimes_{k} n}"', swap]
  \arrow[r, "\Psi_{R,S_i}"]
& R \oplus \bigoplus_{S_i} R
  \arrow[u]
\end{tikzcd}
\]
where $g_i$ is the $R$-algebra map induced by $h^{\otimes_{k} n}$, and 
\[A_n = R \otimes_{g_i,\mathrm{Sym}_{R}(\oplus_{S_i} R),\mathrm{Sym}_{R}(\Psi_{R,S_i})} \mathrm{Sym}_{R}(R \oplus \bigoplus_{S_i} R).\]
By Proposition~\ref{proposition:symformula}, the map $\mathrm{Sym}_{R}(\Psi_{R,S_i})$ is a faithfully flat map of commutative rings, 
so $R \to A_n$ is also faithfully flat. 
By Theorem~\ref{theorem:countermod}, the conditions of Proposition~\ref{proposition:modalg} are satisfied, so the 
descendability exponent of $R \to A_n$ is at least $n$. Setting $n=\infty$, we obtain a faithfully flat ring map $k[x_1,x_2,...] \to A_\infty$ that is not descendable.
\end{proof}

\chapter{Affineness of the maximal \'{e}tale locus}\label{chapter:affine}

Purity of the ramification locus is an important result in algebraic geometry. Let us recall its statement:
\begin{theorem}[Purity of the ramification locus]\label{thm:nagata-purity}
Let $X \to Y$ be a morphism of locally of finite-type between locally Noetherian schemes. Let $x \in X$ be a point and let $y = f(x)$. 
Assume that $\cal{O}_{X,x}$ is normal, $\cal{O}_{Y,y}$ is regular, and $\dim \cal{O}_{X,x} = \dim \cal{O}_{Y,y}\ge 1$. 
If for every point specialization $x' \leadsto x$ such that $\dim \cal{O}_{X,x'} = 1$, the morphism $f$ is unramified at $x'$, 
then $f$ is \'{e}tale at $x$. \cite[0EA4]{sp}
\end{theorem}
A stronger form of this purity statement has been proved in the equicharacteristic case. The statement is as follows:
\begin{theorem}[Affineness of the complement of the ramification locus, equicharacteristic case]\label{thm:affine-equi}
Let $X \to Y$ be a morphism of finite-type between locally Noetherian schemes, where $Y$ is excellent and regular containing a field and 
$X$ is normal. Let $V \subset X$ be the maximal open subset such that $f|_V: V \to Y$ is \'{e}tale. 
Then the inclusion $V \to X$ is an affine morphism. \cite[0ECD]{sp}.
\end{theorem}
It is an open question whether the same statement holds in the mixed-characteristic case, and in this chapter, we will prove this is indeed the case.
Let us state the result explicitly.
\begin{theorem}[Affineness of the complement of the ramification locus]\label{thm:affine-mixed}
Let $X \to Y$ be a morphism of finite-type between locally Noetherian schemes, where $Y$ is excellent and regular and 
$X$ is normal. Let $V \subset X$ be the maximal open subset such that $f|_V: V \to Y$ is \'{e}tale. 
Then the inclusion $V \to X$ is an affine morphism.
\end{theorem}
Note that for a Noetherian local ring $(A,\ideal{m}_A, k)$, the punctured spectrum $\spec{A} - \{\ideal{m}_A\}$ is affine if and only if $\dim{A}\le 1$ \cite[0BRC]{sp}.
Hence, in the setting of Theorem~\ref{thm:nagata-purity}, if $V \subset \spec{\cal{O}_{X,x}}$ 
is the maximal open subset such that $f_{x}|_V: V \to \spec{\cal{O}_{Y,y}}$ is \'{e}tale, and we know that $V$ is affine by Theorem~\ref{thm:affine-mixed},
then we can conclude that $\dim \cal{O}_{X,x'} \le 1$ for any generic point $x' \in X_{x} \setminus V$. If $f_{x'}$ is unramified for $x' \in X_{x}$ with $\dim{\cal{O}_{X,x'}}=1$,
which in particular implies that $f_{x'}$ is \'{e}tale, then $X_{x} \setminus V = \emptyset$ and we conclude that $f$ is \'{e}tale at $x$.\\
\indent Perfectoid spaces were introduced by Peter Scholze in \cite{scholze2012perfectoid} and have been used to solve various problems in arithmetic geometry
by reducing problems to geometry over fields of characteristic $p$. A more general theory of perfectoid rings was first developed in
the work of Kedlaya and Liu \cite{kedlaya2015relative}, and later in the work of Bhatt, Morrow and Scholze \cite{bms} and Gabber and Ramero \cite{gabber2018almost}. 
Let us recall the definition of perfectoid rings, following \cite[Section 2.1]{purflat}. 
For a ring $R$, let 
\[R^{\flat} = \varprojlim_{x \mapsto x^p} R/pR\]
be its tilt.  Recall the Witt functor $W(-)$, which 
is a functor from the category of perfect $\bb{F}_p$-algebras to the category of $p$-adically complete $\bb{Z}_p$-algebras, right
adjoint to reduction modulo $p$. The ring $W(R^{\flat})$ is called the ring of Witt vectors of $R^{\flat}$, 
and there is a natural morphism $\theta_{R}: W(R^{\flat}) \to R$ of $p$-adically complete $\bb{Z}_p$-algebras.
\begin{definition}\label{def:perfectoid}
A ring $R$ is called \textit{perfectoid} if it is $\varpi$-adically complete for some $\varpi \in R$ such that $\varpi^p$ divides $p$,
and the induced map $\theta_{R}: W(R^{\flat}) \twoheadrightarrow R$ is surjective with kernel generated by a non-zero-divisor.\\
\indent If $R$ is $p$-torsion free, then the condition on $\theta_{R}$ is equivalent to the condition that the Frobenius map $R/\varpi R \to R/\varpi^p R$ is an isomorphism.\\
\indent A pair $(R,\varpi)$, with $R$ a ring and $\varpi \in R$, is called \textit{formally perfectoid} if $\varpi^p$ divides $p$, 
and the $\varpi$-adic completion $R^{\wedge \varpi}$ of $R$ is perfectoid.
\end{definition}
We can form formally perfectoid schemes and perfectoid formal schemes by gluing locally on the affine level due to the following lemma:
\begin{lemma}\label{lem:locperfectoid}
Let $(R,\varpi)$ be a formally perfectoid (resp. perfectoid) ring, and $f: R \to R'$ a map of rings 
such that $R/\varpi^n \to R'/\varpi^n$ is a weakly \'{e}tale morphism for all $n \ge 1$.
 Then $(R',\varpi)$ (resp. $R'^{\wedge \varpi}$) is also formally perfectoid (resp. perfectoid). If moreover
 $f$ is $\varpi$-adically faithfully flat, then the converse also holds.
\end{lemma}
\begin{proof}
This is settled in Gabber and Ramero's work \cite[16.7.21]{gabber2018almost}. 
In the case when $R$ is $p$-torsion free then we can work with the equivalent condition on the Frobenius map as follows.
Consider the following diagram:

\[\begin{tikzcd}
R'/\varpi \arrow[r, "\phi"] & R'/\varpi^p \arrow[r] & R'/\varpi \\
R/\varpi \arrow[u, "\psi_1"] \arrow[r, "\phi"] & R/\varpi^p \arrow[u, "\psi_p"] \arrow[r] & R/\varpi \arrow[u," \psi_1"]
\end{tikzcd}\]
We claim the left square is cocartesian, which will prove the proposition. In other words, we want to show the canonical map:
\[\Phi: R'/\varpi \otimes_{R/\varpi, \phi} R/\varpi^p \to R'/\varpi^p\]
is an isomorphism of $R/\varpi^p$-algebras. 
The module on the right is flat, so to show this isomorphism, it suffices to reduce modulo $\varpi$, where the result follows since the outer square is cocartesian by \cite[Theorem 3.5.13]{gabber2003almost}.\\
\indent For the converse one argues by faithfully flat descent (see \cite[16.7.21]{gabber2018almost} for details).
\end{proof}
\begin{definition} A \textit{formally perfectoid scheme} $X$ is a scheme which admits an open cover by affine schemes of the form $\spec{R}$, 
where $R$ is a formally perfectoid ring. A \textit{perfectoid formal scheme} is a formal scheme $X$ which admits an open cover by affine 
formal schemes of the form $\spf{R}$, where $R$ is a perfectoid ring.
\end{definition}
A key feature of perfectoid rings is that they admit a tilting equivalence, which often allows us to reduce problems about perfectoid rings in mixed 
characteristic to problems about perfectoid rings in characteristic $p$.
\begin{theorem}[Tilting equivalence]\label{thm:tilting}
Fix a perfectoid ring $R$, and an element $\xi = (\xi_0, \xi_1, \ldots) \in W(R^{\flat})$ which generates the kernel of $\theta_{R}$. 
Then there is an equivalence of categories between the category of perfectoid $R$-algebras and the category of perfectoid $R^{\flat}$-algebras, given by the functor:
\[S \mapsto S^{\flat} = \varprojlim_{x \mapsto x^p} S/pS\]
with inverse given by the functor:
\[S' \mapsto W(S') \otimes_{W(R^{\flat})} R.\]
The category of perfectoid formal schemes $X \to \spec{R}$ is equivalent to the category of perfectoid formal schemes $X^{\flat} \to \spec{R^{\flat}}$.
\end{theorem}
\begin{proof}
See \cite[Proposition 2.1.9]{purflat} for the proof of the tilting equivalence. The non-affine case follows formally
by gluing data via Lemma~\ref{lem:locperfectoid}.
\end{proof}
Let us now outline a slightly informal potential strategy to answer Theorem~\ref{thm:affine-mixed} using perfectoid formal schemes,
mimicking the approach in the equicharacteristic case.
Via standard reduction arguments using excellence and normality, we can reduce to the case that $Y=\spec{A}$ is a regular complete local ring
and $X=\spec{B}$ is a normal complete local ring, and $f:Y \to X$ is a local morphism. Let $V$ be the locus where $f$ is \'{e}tale. 
By induction on the dimension of $A$, we may assume $H^i(\cal{O}_V)$ is an $A$-module supported on the maximal ideal $\{\ideal{m}_A\}$ for $i \ge 1$, 
and by purity of the ramification locus and Hartshorne-Lichtenbaum vanishing (\cite[0EB7]{sp}), $H^{d-1}(\cal{O}_V) = H^{d}(\cal{O}_V) = 0$, where $d = \dim{A}$. 
This cohomological property of $V$ is rather curious, so we will give it a name:
\begin{definition}\label{def:cohpure}
Let $(A,\ideal{m}_A,k)$ be a Noetherian local ring of dimension $d$. We say a quasi-compact and separated $A$-scheme $V \to \spec{A}$ is \textit{cohomologically pure in codimension $1$} 
if $H^i(\cal{O}_V)$ is an $A$-module supported on $\{\ideal{m}_A\}$ for $i \ge 1$, 
$H^{d-1}(\cal{O}_V) = H^{d}(\cal{O}_V) = 0$ where $d = \dim{A}$, and $V \to \spec{A}$ is dominant.
\end{definition}
If $A$ was of equicharacteristic, then it is a rather surprising result that under these assumptions, $V$ is indeed an affine scheme. This
was shown in \cite[0ECC]{sp} by showing that any Frobenius unit module supported on $\{\ideal{m}_A\}$ is isomorphic to a direct
sum of copies of $E_{A}(k_A)$, the injective hull of the residue field of $A$. Over arbitrary $\bb{F}_p$-algebras,
this characterization of Frobenius unit modules is related to the \textit{Riemann-Hilbert correspondence} in characteristic $p$ \cite{rhbl}.
\begin{theorem}\label{thm:rhbl}
Let $A$ be an $\bb{F}_p$-algebra. There is an equivalence of categories:
\[D_{\mathrm{cons}}(\spec{A}, \bb{F}_p) \stackrel{\mathrm{RH}}{\to} D_{\mathrm{fgu}}(A)\]
between the derived category of Frobenius unit $A$-modules and the derived category of constructible \'{e}tale $\bb{F}_p$-sheaves on $\spec{A}$.\\
\indent Under this correspondence, $i_{x*} \bb{F}_p$ corresponds to $R\Gamma_{x}(A)$ for any $x \in \spec{A}$, where $i_x: \spec{k(x)} \to \spec{A}$ 
is the inclusion of the point $x$. \cite[Theorem 12.6.1]{rhbl}
\end{theorem}
In particular, in the setting of Theorem~\ref{thm:rhbl}, if $A$ is local and Gorenstein, then a Frobenius unit module $M$ supported on 
$\{\ideal{m}_A\}$ lies in the essential image of 
\[D_{\mathrm{cons}}(\spec{k}, \bb{F}_p) \stackrel{i_*}{\to} D_{\mathrm{cons}}(\spec{A}, \bb{F}_p) \stackrel{\mathrm{RH}}{\to} D_{\mathrm{fgu}}(A),\]
and is therefore isomorphic to a direct sum of shifts of the complex $R\Gamma_{\ideal{m}_A}(A) \simeq E_A(k_A)$, recovering the earlier result.\\
\indent Unfortunately, a Riemann-Hilbert correspondence of this nature is not yet available in mixed characteristic,\footnote{Bhatt and Lurie are currently working on a manuscript that
describes the correspondence.} so one would like to use the tilting equivalence to reduce to the equicharacteristic case.\\
\indent We can show that $A$ admits a faithfully flat perfectoid cover $A_{\infty}$, (see Corollary~\ref{cor:perfectoid-cover}) and so the pullback 
$V_{\infty} = V \times_{\spec{A}} \spec{A_{\infty}}$ is cohomologically pure in codimension $1$ over $A_{\infty}$.
Moreover, the associated $p$-adic formal scheme $\mathfrak{V}_{\infty} \to \spf{A_{\infty}}$ is a perfectoid formal scheme by \ref{lem:locperfectoid}
which is $p$-completely \'{e}tale over $\spf{A_{\infty}}$.
Under the tilting equivalence, $\mathfrak{V}_{\infty}$ corresponds to a perfectoid formal scheme $\mathfrak{V}_{\infty}^{\flat} \to \spf{A_{\infty}^{\flat}}$ which is 
$p^{\flat}$-completely \'{e}tale over $\spf{A_{\infty}^{\flat}}$ for some non-zero-divisor $p^{\flat} \in A_{\infty}^{\flat}$. 
The essential feature of the tilting equivalence is that it preserves the underlying topological space. In particular, we have the following commutative diagram:
\[\begin{tikzcd}\mathfrak{V}_{\infty} \arrow[d] \arrow[rd] & & \mathfrak{V}_{\infty}^{\flat} \arrow[d] \arrow[ld]\\
    \spf{A_{\infty}} \arrow[rd] & V_{\infty,p=0}\simeq V^{\flat}_{\infty, p^{\flat}=0} \arrow[d] & \spf{A_{\infty}^{\flat}} \arrow[ld]\\
    & \spec{A_{\infty}/p} \simeq \spec{A_{\infty}^{\flat}/p^{\flat}}
\end{tikzcd}\]
To control the cohomological properties of $\mathfrak{V}_{\infty}^{\flat}$, it is therefore necessary to understand
the cohomological properties of $V_{\infty,p=0}$, which already presents a challenge in the situation when $A$ is a ramified complete local ring.
Moreover, the ring $A_{\infty}^{\flat}$ is not Noetherian, and in general not coherent.\footnote{The notion of when the perfection of 
a Noetherian $\bb{F}_p$-algebra $A$ is coherent, so called the \textit{F-coherence} of $A$, is very subtle, related to the underlying singularities
of $A$. See \cite{fcoherence} for a detailed discussion.} If we algebraize $\mathfrak{V}_{\infty}^{\flat}$
to a scheme $U \to \spec{A_{\infty}^{\flat}}$ which is genuinely \'{e}tale such that $\mathfrak{V}_{\infty}^{\flat} = U^{\wedge p^{\flat}}$,
then unfortunately it is not clear that the complex $R\Gamma(\cal{O}_U)$ is a Frobenius-unit complex over $\spec{A_{\infty}^{\flat}}$, 
which we need to apply the Riemann-Hilbert correspondence in characteristic $p$. 
Finally, even if we could show that $U$ is affine, it is not clear how to descend this affineness property to $V_{\infty}$, as while the tilting equivalence would indeed
give us that $R\Gamma(\cal{O}_{\mathfrak{V}_{\infty}})$ is discrete, we would only be able to obtain that:
\[R\Gamma(\cal{O}_{V_{\infty}}) \simeq \hofib(R\Gamma(\cal{O}_{\mathfrak{V}_{\infty}}) \oplus R\Gamma(\cal{O}_{V_{\infty}}[1/p]) \to R\Gamma(\cal{O}_{\mathfrak{V}_{\infty}}[1/p])).\] 
And as while both $R\Gamma(\cal{O}_{\mathfrak{V}_{\infty}})$ and $R\Gamma(\cal{O}_{V_{\infty}}[1/p])$ are discrete, the homotopy fiber is not necessarily discrete 
unless we knew that $H^0(\cal{O}_{\mathfrak{V}_{\infty}}) = H^0(\cal{O}_{V_{\infty}})^{\wedge p}$, which is not clear. 
\begin{warning}
Let $A$ be a commutative ring and $f\in A$ an element. Let $U \to \spec{A}$ be an $A$-scheme, and consider the pushout square in algebraic spaces (see \cite[Theorem 1.4]{bhattalg}):
\[\begin{tikzcd}U^{\wedge f}[1/f] \arrow[r] \arrow[d] & U^{\wedge f} \arrow[d]\\
U[1/f] \arrow[r] & U
\end{tikzcd}\]
Then even if $U^{\wedge f}$ is an affine formal scheme and $U[1/f]$ is an affine scheme, $U$ does not necessarily need to be affine.
\end{warning}
\indent Said more informally, working in mixed characteristic introduces an extra `arithmetic' parameter leading to `off-by-one' bounds in cohomology groups, 
a feature which has been observed in various contexts such as \cite{bhattcohen} and \cite{purflat}.\\

In this chapter, we will first show that Theorem~\ref{thm:affine-mixed} is true in the case that $X$ is unramified i.e. 
for every closed point $y \in Y$, the $p \in \ideal{m}_{Y,y} \setminus \ideal{m}_{Y,y}^2$. There are two key inputs, the first being a topological
result to control the cohomological properties of $V_{p=0}$ (see Theorem~\ref{thm:topinput}), and the second being an arithmetic input of classifying $D$-modules over unramified complete local rings (see Theorem~\ref{thm:dmodunram}).
This proof mimics the approach in equicharacteristic case. A similar method was also used by O. Gabber and W. Zhang in \cite{gabber}.\\
\indent We will then prove Theorem~\ref{thm:affine-mixed} in general. The above outlined strategy with perfectoid formal schemes, the tilting equivalence and the Riemann-Hilbert correspondence in characteristic $p$
will not play a role, and indeed, it seems that informally, our argument is the dual of the above outlined strategy. In particular,
instead of controlling $H^i(\cal{O}_V)$ for $i\ge 1$, we will focus on $H^0(\cal{O}_V)$, and show that $H^0(\cal{O}_V)\otimesl_A k$ is a discrete ring (see Theorem~\ref{Theorem:torindepmixed}).
The argument rests on the work of Bhatt in \cite{bhattcohen} where we have control over the local cohomology of the absolute integral closure
$A^{+}$ of $A$. The discreteness of $H^0(\cal{O}_V)\otimesl_A k$ is a variant of the Tor-independence for perfectoid rings, which we have formulated for formally perfectoid rings 
to overcome the `off-by-one' bounds in cohomology groups arising
by $p$-completion (see Proposition~\ref{prop:torindep}).

\subsubsection*{Notation and conventions}
For a ring $A$ and an ideal $I \subset A$, we denote by $A^{\wedge I}$ the $I$-adic completion of $A$. We will
use cohomological indexing when working with objects in the derived category. 
For a domain $R$, we let $K(R)$ denote its field of fractions, and for a natural number $n \in \bb{N}$ greater than $1$, we denote $[n]=\{1,2,\ldots,n\}$.

\section{Artin-Rees for perfections}
A well-known result in commutative algebra is that for a Noetherian ring $A$ and an ideal $I \subset A$, the 
$I$-adic completion $A^{\wedge I}$ is a flat $A$-algebra via the natural map $A \to A^{\wedge I}$. 
When working in $p$-adic geometry one often needs to work with completions of non-Noetherian rings for which this result may not be true.
The main failure is that finitely generated $A$-modules may not have the Artin-Rees property.
\begin{definition}[Artin-Rees property]\label{def:artin-rees}
Let $A$ be a coherent ring. We say $A$ has the Artin-Rees property if the following holds: for every finitely presented ideal $I \subset A$ and 
every finitely presented module $M$ over $A$ with a finitely presented $A$-submodule $N \subset M$, there exists an integer $c \ge 0$ such that for all $n \ge c$,
\[
I^n M \cap N = I^{n-c}(I^c M \cap N).
\]
\end{definition}
\begin{theorem}\label{thm:artin-reescoh}
Let $J$ be a filtered category and $\{A_j\}_{j \in J}$ be a diagram of coherent commutative rings such that for each $j \in J$, $A_j$ has the Artin-Rees property, 
and for each morphism $i \to j$ in $J$, the induced map $A_i \to A_j$ is flat. 
Then the colimit $A = \colim_{j \in J} A_j$ also is a coherent ring that has the Artin-Rees property.
\end{theorem}
Since any Noetherian ring has the Artin-Rees property (see \cite[00IN]{sp}), we have the following corollary:
\begin{corollary}\label{cor:artin-rees}
Let $J$ be a filtered category and $\{A_j\}_{j \in J}$ be a diagram of Noetherian rings such that for each morphism $i \to j$ in $J$, the induced map $A_i \to A_j$ is flat.
Then the colimit $A = \colim_{j \in J} A_j$ has the Artin-Rees property.
\end{corollary}
\begin{proof}[Proof of Theorem~\ref{thm:artin-reescoh}]
Let $M$ be a finitely presented $A$-module and $N \subset M$ a finitely presented $A$-submodule and $I$ a finitely generated ideal of $A$. 
Then there exists some $j \in J$ such that $M$ is the base change of a finitely presented $A_j$-module $M_j$, 
and $N$ is the base change of a finitely presented $A_j$-submodule $N_j$ together with a map $N_j \to M_j$, and $I_{j} \subset A_j$ is a finitely generated ideal such that $I = I_j A$.
Since $A_j$ is coherent, $\ker(N_j \to M_j)$ is finitely presented $A_j$-module, and since $N \to M$ is injective, at the expense of increasing $j$, we may assume $\ker(N_j \to M_j) = 0$, so that $N_j$
is a finitely presented submodule of $M_j$. Since $A_j$ has the Artin-Rees property, there exists some integer $c$ such that for all $n \ge c$,
\[I_j^n M_j \cap N_j = I_j^{n-c}(I_j^c M_j \cap N_j).\]
Since $A_j$ is flat over $A_i$ for all $j \ge i$, we have know that $A_j \to A$ is flat as well. Therefore, for all $n \ge c$,
\[I^n M \cap N = I_j^n M_j \cap N_j \otimes_{A_j} A = I_j^{n-c}(I_j^c M_j \cap N_j) \otimes_{A_j} A = I^{n-c}(I^c M \cap N),\]
as desired.
\end{proof}
\begin{theorem}\label{thm:artin-rees-completions}
Let $A$ be a coherent ring with the Artin-Rees property and $I \subset A$ a finitely generated ideal. Then the $I$-adic completion $A^{\wedge I}$ is a flat $A$-algebra via the natural map $A \to A^{\wedge I}$.
\end{theorem}
\begin{proof}
The proof will primarily be a matter of checking that the results of \cite[0BNH]{sp} will still hold in our setting.
\begin{claim}\label{claim:artin-rees}
Let $K \to N$ be an injective map of finitely presented $A$-modules.
\begin{enumerate}
\item[(i)] If $0 \to K \to N \to M \to 0$ is a short exact sequence of finitely presented $A$-modules, then the sequence of completions $0 \to K^{\wedge I} \to N^{\wedge I} \to M^{\wedge I} \to 0$ is also exact.
\item[(ii)] If $M$ is a finitely presented $A$-module, then the natural map $M \otimes_A A^{\wedge I} \to M^{\wedge I}$ is an isomorphism.
\end{enumerate}
\end{claim}
\begin{proof}
Since $A$ is coherent, $K$ is a finitely presented submodule of $N$ and $I$ is a finitely presented $A$-module. For each $n\ge 1$ we get a short exact sequence:
\[0 \to K/(I^nN \cap K) \to N/(I^n N) \to M/(I^n M) \to 0.\]
Each map $K/(I^nN \cap K) \to K/(I^{n+1}N \cap K)$ is surjective, so the inverse system $\{K/(I^nN \cap K)\}_{n \ge 1}$ satisfies the Mittag-Leffler condition, and we have an exact sequence:
\[0 \to \lim_{n} K/(I^nN \cap K) \to \lim_{n} N/(I^n N) \to \lim_{n} M/(I^n M) \to 0.\]
Since $A$ has the Artin-Rees property, the pro-system $\{K/(I^nN \cap K)\}_{n \ge 1}$ is pro-isomorphic to the pro-systems $\{K/I^n K\}_{n \ge 1}$, so we have an isomorphism:
\[\lim_{n} K/(I^nN \cap K) \simeq \lim_{n} K/I^n K = K^{\wedge I}.\]
This proves (i). For (ii), let $0 \to K \to A^{\oplus t} \to M \to 0$ be a presentation of $M$, and note that $K$ is a finitely presented $A$-module. Consider the following commutative diagram:
\[\begin{tikzcd} & K \otimes A^{\oplus t} \arrow[r] \arrow[d] & A^{\oplus t} \otimes A^{\oplus t} \arrow[r] \arrow[d] & M \otimes A^{\oplus t} \arrow[r] \arrow[d] & 0 \\
0\arrow[r]& K^{\wedge I} \arrow[r] & A^{\oplus t} \arrow[r] & M^{\wedge I} \arrow[r] & 0
\end{tikzcd}\]
By (i), the bottom row is exact, and the top row is exact by right exactness of the tensor product. 
The middle arrow is an isomorphism, so we obtain that by the snake lemma that $M \otimes_A A^{\wedge I} \to M^{\wedge I}$ is always a surjection for any finitely presented $A$-module (alternatively,
we may use \cite[0315]{sp}), and in particular for $K$, and therefore again by the snake lemma, all maps in the above diagram are isomorphisms, proving (ii).
\end{proof}
By \cite[00HD]{sp}, it suffices to show for any finitely generated ideal $J \subset A$, $J \otimes_A A^{\wedge I} \to A^{\wedge I}$ is injective. 
Since $A$ is coherent, $J$ is finitely presented, so by Claim~\ref{claim:artin-rees} (ii), we have an injection 
\[J \otimes_A A^{\wedge I} \simeq J^{\wedge I} \to A^{\wedge I}\]
as desired.
\end{proof}
As a concrete application of Theorem~\ref{thm:artin-rees-completions}, we have the following result:
\begin{corollary}\label{cor:perfectoid-cover}
Let $A$ be a complete local regular ring. Then $A$ has a perfectoid cover $A_{\infty}$ such that $A_{\infty}$ is flat over $A$.
\end{corollary}
\begin{proof}
In \cite[Lemma 3.1.1. (b)]{purflat}, it is shown that $A$ admits a system of tower $\{A_i\}_{n\ge 0}$ of finite, 
free $A$-algebras of $p$-power order whose colimit $A_{\infty}$ is formally perfectoid. Explicitly,
if $A = W(k)[[x_1, \ldots, x_n]]/(p-f)$ for some perfect field $k$ of characteristic $p$ and $f \in (p,x_1,...,x_n)^2$, then we can take 
\[A_i = W(k)[[x_1^{1/p^i}, \ldots, x_n^{1/p^i}]]/(p-f), \; \; \forall i \ge 1.\]
By Theorem~\ref{thm:artin-reescoh}, $A_{\infty}$ has the Artin-Rees property, so by Theorem~\ref{thm:artin-rees-completions}  the $p$-adic completion $A_{\infty}^{\wedge p}$ is a perfectoid ring that is flat over $A_{\infty}$. 
Since $A \to A_{\infty}$ is flat, we conclude that $A \to A_{\infty}^{\wedge p}$ is also flat as desired.
\end{proof}
\begin{remark}
The converse of Corollary~\ref{cor:perfectoid-cover} also holds, as has been shown in \cite{padickunz}.
\end{remark}
\section{The unramified case}
In this section, we prove Theorem~\ref{thm:affine-mixed} in the case that $X$ is unramified over $Y$ by similar methods as in the equicharacteristic case.
Our key result is the following.
\begin{theorem}\label{thm:cohpureunram}
Let $(A, \ideal{m}_A,k)$ be an unramified regular local ring and $U \to \spec{A}$ an \'{e}tale morphism that is cohomologically pure in codimension $1$.
Then $U$ is an affine scheme.
\end{theorem}
The proof will consist of two key inputs, the first being a topological result to control the cohomological properties of $U_{p=0}$, 
and the second being an arithmetic input of classifying $D$-modules over unramified complete local rings. We begin with the topological input. 
Recall the following connectedness lemma: 
\begin{theorem}\label{thm:connectedness}
Let $A$ be a normal complete Noetherian local ring with maximal ideal $\ideal{m}_A$ of dimension at least $3$. For any $f \in \ideal{m}_A$,
the punctured spectrum $\spec{A/fA} - \{\ideal{m}_A\}$ is connected.
\end{theorem}
\begin{proof}
This is a consequence of \cite[0EG3]{sp} (after noting that $A$ is $f$-adically complete, by say \cite[090T]{sp})
\end{proof}
We also have the following result that explains our terminology of cohomological purity in codimension $1$:
\begin{lemma}\label{lem:cohpuretop}
Let $(A,\ideal{m}_A,k)$ be a Noetherian local ring of dimension $d$, $V \to \spec{A}$ a morphism cohomologically pure in codimension $1$.
\begin{enumerate}
\item[(i)] If there is an open immersion $V \to \spec{A'}$ of $A$-schemes, where $A'$ is a Noetherian regular domain that is integral over $A$, 
then $\spec{A'} - V$ is of pure codimension $1$ in $\spec{A'}$. In particular, $V$ is affine. 
\item[(ii)] Suppose $V \to \spec{A}$ is \'{e}tale, and let 
$g: A \to B$ be the normalization of $A$ in $V$ giving a compactification 
$V \hookrightarrow \spec{B}$ of $V$ by Zariski's main theorem (\cite[02LR]{sp}). If $A$ is regular with $d \ge 3$, 
and the pair $(A, \ideal{m}_A)$ is henselian,
then $\spec{B} - V$ has dimension at least $2$. In particular, $V \to \spec{A}$ misses the maximal ideal.
\end{enumerate}
\end{lemma}
\begin{proof}
(i): Let $I \subset A'$ define the complement $\spec{A'} - V$. Let $\ideal{p} \subset A'$ be a component of $V(I)$ of height $c$. Then 
\[H_{V(I)}^c(A') \otimes_{A'} A'_{\ideal{p}} \simeq H_{\ideal{p}A'_{\ideal{p}}}^c(A'_{\ideal{p}}) \neq 0.\]
On the other hand, for $2 \le c \le d-1$, we have $H_{V(I)}^c(A') \simeq H^{c-1}(\cal{O}_V)$ is supported on $\{\ideal{m}_{A}\}$, so $H_{V(I)}^c(A') \otimes_{A'} A'_{\ideal{p}} = 0$,
and for $c=d$, we have $H_{V(I)}^d(A') \simeq H^{d-1}(\cal{O}_V) = 0$, so $H_{V(I)}^c(A') \otimes_{A'} A'_{\ideal{p}} = 0$ as well.
Therefore, we conclude $c=1$, as desired.\\
(ii): Since $V \to \spec{A}$ is \'{e}tale and $A$ is regular, $V \simeq \prod_{i \in [n]} V_i, n \in \bb{N}$ where each $V_i$ is an integral scheme \'{e}tale over $A$. 
Hence, we may assume $V$ is integral.
Then $B$ is a normal domain integral over $A$. Since
$(A,\ideal{m}_A)$ is henselian, $B$ is a local ring.\\
\indent There are two cases, either $B$ is \'{e}tale over $A$ or it has ramification. In the first case, 
we may conclude by (i),
and in the second case, $V$ is contained in the maximal open subset of $\spec{B}$ where $\spec{B} \to \spec{A}$ is \'{e}tale, 
so the complement $\spec{B} - V$ contains the ramification locus of $g$, which is of pure codimension $1$ in $\spec{B}$ by purity of the ramification locus, 
so $\spec{B} - V$ has dimension at least $2$.
\end{proof}
\begin{theorem}[Topological input]\label{thm:topinput}
Let $A$ be an regular local ring of dimension at least $3$, and $U \to \spec{A}$ an \'{e}tale morphism that is cohomologically pure in codimension $1$.
Let $f \in A$ be an element, and write $f = \prod_{i=1}^n f_i^{n_i}$ uniquely up to units of $A$ as a product of irreducible elements $f_i \in \ideal{m}_A \setminus \ideal{m}_A^2$.
Then $U_{f=0} \to \spec{A/f}$ is cohomological pure in codimension $1$ as well.
\end{theorem}
\begin{proof}
We may assume $A$ is a complete regular local ring since the map $A \to A^{\wedge \ideal{m}_A}$ is faithfully flat.\\
\indent Since $U \to \spec{A}$ is flat, it suffices to show, by the Bockstein sequence, that $H^{d-2}(\cal{O}_{U_{f=0}}) = 0$.
Let $B$ be the normalization of $A$ in $U$ so that we have
a compactification $U \to \spec{B} \stackrel{g}{\to} \spec{A}$ by Zariski's main theorem. Since $A$ is excellent, $B$ is module-finite over $A$ 
and thus complete and local (since $A$ is henselian). 
Let $I \subset B$ be the ideal defining the complement $\spec{B} - U$.\\
\indent By Lemma~\ref{lem:cohpuretop}, $V(I)$ has dimension at least $2$ in $\spec{B}$. Therefore, $V(I/f)$ has dimension at least $1$ in $\spec{B/fB}$, so $U_{f=0}$ is strictly contained in $\spec{B/fB} - \{\ideal{m}_{B/fB}\}$,
 which is connected by Theorem~\ref{thm:connectedness}.\\
 \indent Now assume that $H^{d-2}(\cal{O}_{U_{f=0}}) \neq 0$. 
 Then by Hartshorne-Lichtenbaum vanishing (\cite[0EB7]{sp}), there exists a minimal prime $\ideal{p} \subset B/fB$ such that $V(\ideal{p}) \cap V(I/f) = \{\ideal{m}_{B/f}\}$.
 Let $h\in A$ be the element defining the closed irreducible subscheme $g(V(\ideal{p})) \subset \spec{A}$, which is a component of $V(f)$ and is therefore normal.
 We have a diagram of schemes:
\[\begin{tikzcd}
V(\ideal{p}) - \{\ideal{m}_{B/hB}\} \arrow[rd] \arrow[r] & U_{h=0} \arrow[d] \\
& \spec{A/h} - \{\ideal{m}_{A/h}\} \end{tikzcd}\]
Since $V(\ideal{p}) - \{\ideal{m}_{B/hB}\}$ is closed in $U_{h=0}$, and the down arrow is \'{e}tale, we conclude that
the right-down map is a dominant, unramified morphism of schemes. Since $A/h$ is normal, the scheme
$\spec{A/h} - \{\ideal{m}_{A/h}\}$ is geometrically unibranch, so the right-down map is in fact \'{e}tale. Hence, the horizontal map is also \'{e}tale, and in particular
an open immersion. Therefore, $V(\ideal{p}) - \{\ideal{m}_{B/hB}\}$ is a proper, closed subset of $\spec{B/hB} - \{\ideal{m}_{B/hB}\}$ that is also open, 
which is a contradiction since the latter is connected.
\end{proof}
The arithmetic input relies on a characterization of $D$-modules over unramified complete local rings. Let us first recall
the Grothendieck's construction for the ring of differential operators over a commutative ring $A$.
\begin{definition}
Let $k$ be a commutative ring, $A$ a $k$-algebra and $M,N$ be $A$-modules.\\
\indent A differential operator of order $0$ from $M$ to $N$ is an $k$-linear map. 
For $n \ge 1$, a differential operator of order at most $n$ from $M$ to $N$ is a $k$-linear map 
$\delta: M \to N$ such that for all $a \in A$, $m \mapsto \delta(am)-a\delta(m)$ is a 
differential operator of order at most $n-1$.\\
\indent The set of all differential operators from $M$ to $N$ of order $k$ is denoted by 
$\mathrm{Diff}_{A/k}^k(M,N)$, and let $\mathrm{Diff}_{A/k}(M,N) = \bigcup_{k \ge 0} \mathrm{Diff}_{A/k}^k(M,N)$. When $M=N=A$, 
we write $\mathrm{Diff}_{A/k}(A) = \mathrm{Diff}_{A/k}(A,A)$, 
which is a (non-commutative) ring under composition and is called the ring of $k$-linear differential operators on $A$. An $A$-module 
$M$ is called a $\mathrm{Diff}_{A/k}$-module ($D$-module for short) if it is equipped with an action of $\mathrm{Diff}_{A/k}(A)$ that extends the $A$-module structure on $M$.
\end{definition}
We now have the following characterization of $D$-modules over the closed point of an unramified complete local ring.
\begin{theorem} \label{thm:dmodunram}
Let $V$ be a DVR with uniformizer $\pi$ and residue field $k$. Let $A=V[[x_1, \ldots, x_n]]$ be a complete local ring that is unramified over $V$. 
Let $M$ be an $A$-module that is supported on the closed point $\{\ideal{m}_A\}$ and is equipped with a structure of a $\mathrm{Diff}_{A/V}(A)$-module. 
If $\pi: M \to M$ is surjective, then $M$ is isomorphic to a direct sum of copies of $E_A(k)$, the injective hull of the residue field of $A$.
\end{theorem}
\begin{proof}
For any multi-index $E=k_1e_1 + \cdots + k_ne_n \in \bb{N}^{n}$, we have the $\bb{Z}_p$-linear differential operators on $A$ defined by:
\[D_E=\frac{1}{k_1! \cdots k_n!} \partial_{x_1}^{k_1} \cdots \partial_{x_n}^{k_n}.\]
Now, for any $m \in M$ and $i \in [n]$, there exists a $k \in \bb{N}$ such that $x_i^km=0$. Therefore:
\[D_{k e_i}(x^km) = \frac{1}{k!} \partial_{x_i}^k(x_i^km) = m + \sum_{j \in [k]} \frac{j!}{k!} x_i^{k-j} \partial_{x_i}^{k-j}(m) = 0.\]
In particular, any $m \in M$ is in $x_iM$ for all $i \in [n]$, so $x_i: M \to M$ is surjective. Also, if $\pi \cdot m = 0$, then $\pi \cdot \partial^{j}_{x_i}(m) = 0$ for any $j \in \bb{N}$ and $i \in [n]$, 
so $x_i: M[\pi] \to M[\pi]$ is also surjective for all $i \in [n]$. By the snake lemma applied to the commutative diagram:
\[\begin{tikzcd}
0 \arrow[r] & M[x_i] \arrow[r] \arrow[d, "\pi"] & M \arrow[r, "x_i"] \arrow[d, "\pi"] & M \arrow[r] \arrow[d, "\pi"] & 0\\
0 \arrow[r] & M[x_i] \arrow[r] & M \arrow[r, "x_i"] & M \arrow[r] & 0\end{tikzcd}\]
we conclude that $\pi: M[x_i]\to M[x_i]$ is also surjective.\\
\indent We will now proceed by induction on $n$. For $n=0$, we observe that $M[\pi]$ is a $k$-vector space, so we obtain a map $h$:
\[\begin{tikzcd}
M \arrow[r, dashed, "h"] & \oplus_{i\in I} E_A(k) \\
M[\pi] \arrow[r, "\simeq"] \arrow[u, hook] & \oplus_{i \in I} k \arrow[u, hook]
\end{tikzcd}\]
for some indexing set $I$, by the injectivity of $\oplus_{i \in I} E_A(k)$,\footnote{The injectivity of $\oplus_{i \in I} E_A(k)$ follows from the fact that $E_A(k)$ is an injective $A$-module, and the direct sum of injective modules is also injective over Noetherian rings. 
In fact, the converse is also true, see \cite[Theorem 3.46]{Lam}.} 
and we will show that $h$ is an isomorphism. Since $\pi: M \to M$ is surjective, and $M$ is supported on the closed point, the desired isomorphism follows
by observing that $h \otimesl_A k: M[\pi] \to \oplus_{i \in I} k$ is an isomorphism by construction, and $h[1/\pi] : M[1/\pi] \to \oplus_{i \in I} E_A(k)[1/\pi]$ is an isomorphism as well since both sides are zero.\\
\indent Now assume that $n \ge 1$. The $M$ submodule $M[x_n]$ is a module over $A/x_n=V[[x_1, \ldots, x_{n-1}]]$ that is supported on the closed point 
and is equipped with a structure of a $\mathrm{Diff}_{V[[x_1, \ldots, x_{n-1}]]/V}$-module such that $\pi: M[x_n] \to M[x_n]$ is surjective. Consider the following commutative diagram:
\[\begin{tikzcd}M \arrow[r, dashed, "h"] & \oplus_{i\in I} E_A(k) \\
M[x_n] \arrow[r, dashed, "h'"] \arrow[u, hook] & \oplus_{i \in I} E_{A/x_n}(k)=E_A(k)[x_n] \arrow[u, hook]\\
M[\ideal{m}_A] \arrow[r, "\simeq"] \arrow[u, hook] & \oplus_{i \in I} k \arrow[u, hook]
\end{tikzcd}\]
By the inductive hypothesis, we have an isomorphism $h': M[x_n] \to \oplus_{i \in I} E_{A/x_n}(k)$. Since $x_n: M \to M$ is surjective, we therefore
obtain that $h \otimesl_A A/x_n: M[x_n] \to \oplus_{i \in I} E_A(k)[x_n]$ is an isomorphism. On the other hand, $h[1/x_n]: M[1/x_n] \to \oplus_{i \in I} E_A(k)[1/x_n]$ is an isomorphism as well since both sides are zero. 
The result follows.
\end{proof}
Given a scheme $X \to \spec{A}$, if every affine open subset of $X$ admits an $\mathrm{Diff}_{A/V}$-module structure on its global sections,
then we would like that the coherent cohomology groups $H^i(\cal{O}_X)$ also admits a $\mathrm{Diff}_{A/V}$-module structure. The following
results allow us to glue $D$-module structures.
\begin{theorem}\label{thm:Dmodglue}
Let $A \to B$ be a map of commutative rings and $M,N$ be $B$-modules
\begin{enumerate}
\item[(i)] For any multiplicatively closed subset $S \subset A$, every $A$-linear differential operator $\delta: M \to N$ of order $k$ 
induces a unique $A$-linear differential operator $S^{-1}\delta: S^{-1}M \to S^{-1}N$ of order $k$, compatible with addition and composition.
\item[(ii)] If $B \to C$ is formally \'{e}tale, then any $A$-linear differential operator $M \to N$ of order $k$ induces a unique 
$A$-linear differential operator $M \otimes_B C \to N \otimes_B C$ (of the same or lesser order), compatible with addition and composition.
\end{enumerate}
\end{theorem}
\begin{proof}
For (i), see \cite[0G36]{sp}, and for (ii), see \cite[0H94]{sp}.
\end{proof}
\begin{corollary}\label{cor:dmodcoh}
Let $V$ be a discrete valuation ring, $A$ be a $V$-algebra. Let $f: X \to \spec{A}$ be a quasi-separated, quasi-compact, formally \'{e}tale  morphism of schemes
and $M$ an $A$-module.  Then each cohomology group $H^i(X, f^{-1}(\tilde{M}))$ admits a natural structure of a $\mathrm{Diff}_{A/V}$-module.
\end{corollary}
\begin{proof}
First assume that $f$ is separated. Let $\{U_i=\spec{B_i}\}_{i \in I}$ be a finite affine open cover of $X$. 
Since $f$ is formally \'{e}tale, each derivation of finite order in $\mathrm{Diff}_{A/V}$ extends uniquely to a derivation on $B_i$ of finite order, 
so $B_i$ admits a natural structure of a $\mathrm{Diff}_{A/V}$-module. The cohomology groups $H^i(X, f^{-1}(\widetilde{M}))$ can be computed by the \v{C}ech complex associated to the cover $\{U_i\}_{i \in I}$,
explicitly given by:
\[ \prod_{i_0} B_{i_0} \otimes_A M \to \prod_{i_0 < i_1} B_{i_0 i_1} \otimes_A M \to \cdots \to \prod_{i_0 < \cdots < i_n} B_{i_0 \cdots i_n} \otimes_A M\]
where $B_{i_0 \cdots i_n}$ is the coordinate ring of the intersection $U_{i_0} \cap \cdots \cap U_{i_n}$. 
By Theorem~\ref{thm:Dmodglue}, each term in the above complex admits a unique structure of a $\mathrm{Diff}_{A/V}$-module, and the differentials are compatible with the $\mathrm{Diff}_{A/V}$-module structures due to this uniqueness.
Hence, every cohomology group $H^i(X, f^{-1}(\widetilde{M}))$ admits a natural structure of a $\mathrm{Diff}_{A/V}$-module.\\
\indent If $f$ is not separated, we instead work with hypercovers, and the above proof goes through verbatim.
\end{proof}
We are now ready to prove Theorem~\ref{thm:cohpureunram}.
\begin{proof}[Proof of Theorem~\ref{thm:cohpureunram}]
We may assume $A$ is a complete regular local ring since the map $A \to A^{\wedge \ideal{m}_A}$ is faithfully flat. We will proceed by induction on $\dim(A)$; 
the case of $d=2$ is trivial so we may assume $d \ge 3$.\\
\indent By the Cohen-Structure theorem, as $A$ is unramified, we have $A \simeq V[[x_1, \ldots, x_n]]$ for some discrete valuation ring $V$ with uniformizer $\pi$. 
By Theorem~\ref{thm:topinput}, $U_{\pi=0} \to \spec{A/\pi}$ is cohomologically pure in codimension $1$ as well, so by induction, we may assume that $U_{\pi=0}$ is an affine scheme. 
Since $U \to \spec{A}$ is flat, the Bockstein sequence associated to the short exact sequence $0 \to \cal{O}_U \stackrel{\pi}{\to} \cal{O}_U \to \cal{O}_{U_{\pi=0}} \to 0$ 
implies that $H^i(\cal{O}_U)=0$ for all $i\ge 2$, and $\pi: H^1(\cal{O}_U)\to H^1(\cal{O}_U)$ is surjective. By Corollary~\ref{cor:dmodcoh}, 
$H^1(\cal{O}_U)$ admits a natural structure of a $\mathrm{Diff}_{A/V}$-module supported on $\{\ideal{m}_A\}$ such that $\pi: H^1(\cal{O}_U) \to H^1(\cal{O}_U)$ is surjective.
By Theorem~\ref{thm:dmodunram}, we conclude that $H^1(\cal{O}_U)$ is isomorphic to a direct sum of copies of $E_A(k)$. In particular,
$H^1(\cal{O}_U) \otimesl_A k$ is concentrated in cohomological degree $-d$.\\
\indent By Lemma~\ref{lem:cohpuretop}, $U \to \spec{A}$ misses the maximal ideal $\ideal{m}_A$, so $R\Gamma(U, \cal{O}_U) \otimesl_A k \simeq 0$. 
But by investigating the spectral sequence 
\[E_2^{i,j} = \mathrm{Tor}_{-i}^A(H^jR\Gamma(\cal{O}_U), k) \Rightarrow H^{i+j}(R\Gamma(\cal{O}_U) \otimesl_A k),\]
we conclude that no differentials in the later pages of the spectral sequence hit the $H^0(\cal{O}_U) \otimes_A k=E_2^{0,0}$-term. 
Hence, $H^0(\cal{O}_U) \otimes_A k=H^0(R\Gamma(\cal{O}_U) \otimesl_A k) =0$, so the map $U \to \spec{H^0(\cal{O}_U)}$ is a surjective quasi-compact open-immersion. Thus, $U$ is affine as desired.
\end{proof}
\section{The general case}
The main result of this section is the following:
\begin{theorem}\label{thm:cohpuregen}
Let $(A,\ideal{m}_A,k)$ be a regular local ring and $V \to \spec{A}$ an \'{e}tale morphism that is cohomologically pure in codimension $1$. Then $V$ is an affine scheme.
\end{theorem}
Let us first discuss a different way to prove Theorem~\ref{thm:cohpuregen} in equicharacteristic than \cite[0ECD]{sp}. First we recall 
the following Tor-independence result for perfect rings.
\begin{lemma}\label{lem:perftorind}Let $B \leftarrow A \to C$ be a diagram of perfect rings. Then the natural map $B \otimesl_A C \to B \otimes_A C$ is an isomorphism.
\end{lemma}
\begin{proof}
This has been observed many times. The earliest published reference we are aware of is \cite[Lemma 3.16 or Proposition 11.6]{BhattScholzeProj}.
\end{proof}
\begin{theorem}\label{thm:cohpuregenp}
Let $(A,\ideal{m}_A,k)$ be a regular local ring containing a field of characteristic $p>0$ and $V \to \spec{A}$ a quasi-compact, \'{e}tale morphism. 
Then $H^0(\cal{O}_V) \otimesl_A k$ is a discrete ring.
\end{theorem}
\begin{proof}
Let $A_{\text{perf}}$ be the perfection of $A$, and $A'=H^0(\cal{O}_V)$. Since $A$ is regular, $A \to A_{\text{perf}}$ is a faithfully flat map of local rings.
Since the map $k \to k_{\text{perf}}$ is faithfully flat, it suffices to show that $A' \otimesl_A k_{\text{perf}}$ is a discrete ring.
The map $A \to A_{\text{perf}}$ is faithfully flat, so 
\[A' \otimesl_A k_{\text{perf}} \simeq A' \otimesl_A A_{\text{perf}} \otimesl_{A_{\text{perf}}} k_{\text{perf}} \simeq A' \otimes_A A_{\text{perf}} \otimesl_{A_{\text{perf}}} k_{\text{perf}}.\]
Now by flat base-change, $A' \otimes_A A_{\text{perf}} \simeq H^0(\cal{O}_{V'})$ where $V' = V \times_{\spec{A}} \spec{A_{\text{perf}}}$.
Since $V' \to \spec{A_{\text{perf}}}$ is still \'{e}tale, the ring $H^0(\cal{O}_{V'}) = A' \otimes_A A_{\text{perf}}$ is a perfect ring, and hence by Lemma~\ref{lem:perftorind}, we have
\[A' \otimes_A A_{\text{perf}} \otimesl_{A_{\text{perf}}} k_{\text{perf}} \simeq A' \otimes_A A_{\text{perf}} \otimes_{A_{\text{perf}}} k_{\text{perf}} .\]
Hence, $A' \otimesl_A k = H^0(\cal{O}_V) \otimesl_A k$ is a discrete ring as claimed.\\
\end{proof}
We then have:
\begin{theorem}\label{thm:cohpuregenequi}
Let $(A,\ideal{m}_A,k)$ be a regular local ring containing a field and $V \to \spec{A}$ an \'{e}tale morphism that is cohomologically pure in codimension $1$. 
Then $V$ is an affine scheme.
\end{theorem}
\begin{proof}
The case $\mathrm{dim}(A) = 2$ is clear and we may therefore assume $d=\mathrm{dim}(A) \ge 3$. 
We may assume the pair $(A, \ideal{m}_A)$ is henselian since the henselianization of $A$ is ind-\'{e}tale over $A$. So by Lemma~\ref{lem:cohpuretop} (ii), we find that $V \to \spec{A}$ misses the maximal ideal of $A$. Hence, by base-change, we see $R\Gamma(\cal{O}_V) \otimesl_A k \simeq 0$.
If $\mathrm{char}(k)=0$, we conclude as in \cite[0ECD]{sp}. Our novelty will be when $\mathrm{char}(k)=p>0$.\\
\indent By Theorem~\ref{thm:cohpuregenp}, we have $H^0(\cal{O}_V) \otimesl_A k$ is a discrete ring. Since $V \to \spec{A}$ is cohomologically pure in codimension $1$, 
each $H^i(\cal{O}_V)$ is supported on $\{\ideal{m}_A\}$ and thus has depth $0$ for $i\ge 1$. 
Hence, if $H^i(\cal{O}_V) \neq 0$ for $i\ge 1$, then $ H^i(\cal{O}_V) \otimesl_A k$ is non-zero in cohomological degree $-d$. Let $j\in [d-2]$ be the first index such that $H^j(\cal{O}_V)\neq 0$. 
By investigating the spectral sequence
\[E_2^{i,j} = \mathrm{Tor}_{-i}^A(H^j(\cal{O}_V), k) \Rightarrow H^{i+j}(R\Gamma(\cal{O}_V) \otimesl_A k),\]
we conclude that the $E_2^{-d,j}$-term survives to the $E_{\infty}$-page, which contradicts the fact that $R\Gamma(\cal{O}_V) \otimesl_A k \simeq 0$. 
Hence, $H^i(\cal{O}_V) = 0$ for $i\ge 1$, and thus $V$ is affine as claimed.
\end{proof}
In the above proof, we used two critical properties of perfect rings:
\begin{enumerate}
\item For any diagram of perfect rings $B \leftarrow A \to C$, the natural map $B \otimesl_A C \to B \otimes_A C$ is an isomorphism.
\item There exists a faithfully flat perfect cover $A \to A_{\text{perf}}$ such that $H^0(\cal{O}_{V'})=H^0(\cal{O}_V) \otimes_A A_{\text{perf}}$ is a perfect ring, where $V' = V \times_{\spec{A}} \spec{A_{\text{perf}}}$.
\end{enumerate}
Mixed characteristic analogues of the first property are usually stated in the $p$-completed setting, which would not be suitable for our purposes.
Instead, we will use the following.
\begin{proposition}\label{prop:torindep}
Let $B \leftarrow A \to C$ be a diagram of commutative rings, and assume there exists an element $\varpi=pu \in A$ admitting all $p$-power roots where $u \in A$ is a unit. 
Suppose that $C$ is perfect and both $A$ and $B$ are $p$-torsionfree rings such that $A/\varpi^{1/p^{\infty}}$ and $B/\varpi^{1/p^{\infty}}B$ are perfect rings. 
Then $B \otimesl_A C$ is a discrete ring.
\end{proposition}
\begin{proof}
The map $A \to C$ factors over the ring $A/\varpi^{1/p^{\infty}}$.
Note that $B \otimesl_{A} A/\varpi^{1/p^{\infty}} \simeq B/\varpi^{1/p^{\infty}}B$ is a discrete ring since $B$ is $p$-torsionfree (so $\varpi$ is a nonzerodivisor in $B$). Therefore, we have
\[B \otimesl_A C \simeq (B \otimesl_A A/\varpi^{1/p^{\infty}}) \otimesl_{A/\varpi^{1/p^{\infty}}} C \simeq (B/\varpi^{1/p^{\infty}}B) \otimesl_{A/\varpi^{1/p^{\infty}}} C.\]
Since all the rings in the right-term of the above expression are perfect, we conclude that $B \otimesl_A C$ is a discrete ring as desired.
\end{proof}
As for the second property, in mixed characteristic, perfectoid rings are generally considered the correct replacement for perfect rings, 
but finding perfectoid covers is more subtle. Moreover, we
are not sure if the following analogue of the second property holds for perfectoid rings. Namely:
\begin{question}
If $V \to \spec{A}$ is an \'{e}tale morphism of schemes with $A$ perfectoid, is $H^0(\cal{O}_V)$ formally perfectoid?
\end{question}
We introduce particular kinds of schemes where some control over the global sections of arbitrary open subsets is guaranteed.
\begin{proposition}\label{prop:perfglobalsections}
Consider the following class of commutative rings having the following property (*): $A$ is a normal domain such that every element of $A$ has a $p$-th root in $A$.
\begin{enumerate}
\item[(i)] If $A$ has property (*), then $A/(p^{1/p^{\infty}})$ is a perfect ring.
\item[(ii)] If $X$ is an integral scheme such that, for every affine open $\spec{A} \subset X$, $A$ has property (*), then $\Gamma(X, \cal{O}_X)$ is also a ring with property (*).
\item[(iii)] If $X$ is an integral scheme such that, for every affine open $\spec{A} \subset X$, $A$ has property (*), then $\Gamma(U, \cal{O}_U)$ has property (*) for every open subset $U \subset X$.
\end{enumerate}
\end{proposition}
\begin{proof}
(i): If $p$ is invertible then $A/(p^{1/p^{\infty}})=0$ and the result is trivial. Otherwise, it is sufficient to show that the Frobenius map $A/p^{1/p^{n+1}} \to A/p^{1/p^{n}}$ is an isomorphism 
for any $n\ge 1$. 
Surjectivity follows because $A$ has all $p$-th roots. For injectivity,
if $x \in A$ has the property that $x^p = p^{1/p^{n}}y$ for some $y \in A$, then $z=\frac{x}{p^{1/p^{n+1}}} \in K(A)$ satisfies the $A$-integral equation $z^p-y=0$.
Since $A$ is normal, we conclude that $z \in A$, so $x = p^{1/p^{n+1}}z \in p^{1/p^{n+1}}A$ as desired.\\
\indent For (ii), by \cite[0358]{sp}, we know that $\Gamma(X, \cal{O}_X)$ is a normal domain. Thus, we need to show that every element $x \in \Gamma(X, \cal{O}_X)$ has a $p$-th root in $\Gamma(X, \cal{O}_X)$. 
Let $\eta \in X$ be the generic point. Then $x|_{\eta} \in k(\eta)$ has a $p$-th root in $k(\eta)$, say $y$. Now $y$ is integral over $\Gamma(X, \cal{O}_X)$, and therefore also integral 
over any affine open $U=\spec{A} \subset X$.
Since $A$ is normal with fraction field $k(\eta)$, we conclude that there exists an element $y_U \in A$ such that $y_U|_{\eta} = y$ and $y_U^p = x|_U$. It is clear
that for any smaller affine open $V \subset U$ that $y_V = y_U|_V$, hence, we obtain a global section $y \in \Gamma(X, \cal{O}_X)$ such that $(y^p - x)|_U = 0$ for all 
affine opens $U \subset X$, so $y^p - x=0$ in $\Gamma(X, \cal{O}_X)$ as desired.\\
\indent (iii) follows from (ii) by replacing $X$ with $U$.
\end{proof}
There is an obvious choice of schemes satisfying the condition in Proposition~\ref{prop:perfglobalsections}, namely, the spectra of absolute integral closures of local domains.
\begin{proposition}\label{prop:flatcovermixed}
Let $X$ be an integral, normal scheme whose fraction field is algebraically closed.
\begin{enumerate}
\item[(i)] For every affine open $\spec{A} \subset X$, $A$ has property (*) of Proposition~\ref{prop:perfglobalsections}. 
\item[(ii)] For every open subset $U \subset X$, the ring $A'=H^0(\cal{O}_U)$ has property (*) of Proposition~\ref{prop:perfglobalsections}. In particular,
$A'/p^{1/p^{\infty}}$ is a perfect ring.
\item[(iii)] Let $U \to X$ be any separated \'{e}tale morphism. Then $U$ is a disjoint union of open subsets of $X$.
\end{enumerate}
\end{proposition}
\begin{proof}
(i): Every such $A$ is normal with algebraically closed fraction field. Therefore, it suffices to observe that $K(A)$ contains all $p$-th roots.\\
(ii): This follows from (i) and Proposition~\ref{prop:perfglobalsections} (iii).\\
(iii): This is the content of (1) of \cite[09Z9]{sp}.
\end{proof}
To proceed, we will use a critical result of Bhatt in his seminal paper \cite{bhattcohen}.
\begin{theorem}\label{thm:bhattcohen}
Let $(A,\ideal{m}_A,k)$ be an excellent local Noetherian domain with $p \in \ideal{m}_A$. Let $A^{+}$ be the normalization of $A$ in an algebraic closure of its fraction field. 
Then $R\Gamma_{\ideal{m}_A}(A^{+})$ is concentrated in cohomological degree $d$ where $d=\mathrm{dim}(A)$.
\end{theorem}
\begin{theorem}\label{Theorem:torindepmixed}
Let $(A,\ideal{m}_A,k)$ be an excellent regular local Noetherian ring of dimension $d$ with $p \in \ideal{m}_A$, $A^{+}$ the normalization of $A$ in the algebraic
closure of $K(A)$, and $i: A \to A^{+}$ the induced integral map.
\begin{enumerate}
\item[(i)] Let $M$ be a module over $A$ that is supported on $\{\ideal{m}_A\}$. Then $M \otimesl_A A^{+}$ is discrete.
\item[(ii)]Let $V \to \spec{A}$ be a flat morphism of schemes such that $V$ is cohomologically pure in codimension $1$ over $\spec{A}$.
 Then $H^0(\cal{O}_V) \otimesl_A A^{+}$ is a discrete ring equal to $H^0(\cal{O}_{V'})$ where $V'= V \times_{\spec{A}} \spec{A^{+}}$.
 \item[(iii)] In the situation of (ii), if $V \to \spec{A}$ is further assumed to be \'{e}tale, then $H^0(\cal{O}_V) \otimesl_A k$ is a discrete ring.
\end{enumerate}
\end{theorem}
\begin{proof}
(i): Since $M$ is supported on $\{\ideal{m}_A\}$, we have
\[A^{+} \otimesl_A M \simeq R\Gamma_{\ideal{m}_A}(A^{+}) \otimesl_A M.\]
Since $A$ is regular, $A$ has global dimension $d$, and since $R\Gamma_{\ideal{m}_A}(A^{+})$ is concentrated in cohomological degree $d$ by Theorem~\ref{thm:bhattcohen}, 
the right-hand side is concentrated in cohomological degrees $\ge 0$. But the left-hand side is concentrated in cohomological degrees $\le 0$,
so we conclude that $A^{+} \otimesl_A M$ is discrete.\\
(ii): By flat base-change, we know that
\[R\Gamma(\cal{O}_{V'}) \simeq R\Gamma(\cal{O}_V) \otimesl_A A^{+}.\]
We have a spectral sequence 
\[E_2^{i,j} = \mathrm{Tor}_{-i}^A(H^j(\cal{O}_V), A^{+}) \Rightarrow H^{i+j}(R\Gamma(\cal{O}_V) \otimesl_A A^{+}) = H^{i+j}(\cal{O}_{V'}).\]
Since $V \to \spec{A}$ is cohomologically pure in codimension $1$, each $H^j(\cal{O}_V)$ is supported on $\{\ideal{m}_A\}$ for $j \ge 1$, so by (i), $E_2^{-i,j} = 0$ for all $i \ge 1$ and $j \ge 1$.
Therefore, the terms $E_2^{-i,0} = \Tor_{i}^A(H^0(\cal{O}_V), A^{+})$ survive to the $E_{\infty}$-page, contributing to negative cohomological degrees of the complex $R\Gamma(\cal{O}_{V'})$.
We therefore conclude that $H^0(\cal{O}_V) \otimesl_A A^{+}$ is discrete and equal to $H^0(\cal{O}_{V'})$.\\
(iii): By (ii), $H^0(\cal{O}_{V}) \otimesl_A A^{+}$ is a discrete ring equal to $H^0(\cal{O}_{V'})$. Moreover,
by Proposition~\ref{prop:flatcovermixed}, $V'=\coprod_{i \in [n]} V_i, n \in \bb{N}$ where each $V_i$ is an open subset of $\spec{A^{+}}$,
so $H^0(\cal{O}_{V_i})$ is a normal domain such that $H^0(\cal{O}_{V_i})/p^{1/p^{\infty}}$ is a perfect ring for each $i \in [n]$.\\
\indent Let $k_{A^{+}}$ be the residue field of $A^{+}$. Note that the induced map 
\[k=A/(\ideal{m}_{A^{+}} \cap A) \to A^{+}/\ideal{m}_{A^{+}}=k_{A^{+}}\]
identifies $k_{A^{+}}$ as an algebraic closure of $k$, so $k \to k_{A^{+}}$ is faithfully flat and $k_{A^{+}}$ is perfect. 
Therefore, by faithfully flat descent, it suffices to show that $H^0(\cal{O}_V) \otimesl_A k_{A^{+}}$ is a discrete ring. We have
\[H^0(\cal{O}_V) \otimesl_A k_{A^{+}} \simeq H^0(\cal{O}_V) \otimesl_A A^{+} \otimesl_{A^{+}} k_{A^{+}} \simeq H^0(\cal{O}_{V'}) \otimesl_{A^{+}} k_{A^{+}}.\]
Since $V'=\coprod_{i \in [n]} V_i$, we have
\[H^0(\cal{O}_{V'}) \otimesl_{A^{+}} k_{A^{+}} \simeq \prod_{i \in [n]} H^0(\cal{O}_{V_i}) \otimesl_{A^{+}} k_{A^{+}}.\]
By Proposition~\ref{prop:torindep}, $H^0(\cal{O}_{V_i}) \otimesl_{A^{+}} k_{A^{+}}$ 
is a discrete ring for each $i \in [n]$. Hence, $H^0(\cal{O}_V) \otimesl_A k$ is a discrete ring as claimed.\\
\end{proof}
We are now ready to prove Theorem~\ref{thm:cohpuregen}.
\begin{proof}[Proof of Theorem~\ref{thm:cohpuregen}]
We've already resolved the equicharacteristic case in Theorem~\ref{thm:cohpuregenequi}, so we may assume 
that $p \in \ideal{m}_A$. The dimension $2$ case is clear, so we may assume that $d=\mathrm{dim}(A) \ge 3$. 
We may also assume that pair $(A, \ideal{m}_A)$ is henselian and $A$ is excellent, since the $\ideal{m}_A$-adic completion of $A$ is faithfully flat over $A$.\\
\indent By Lemma~\ref{lem:cohpuretop} (ii), 
$V$ misses the closed point of $\spec{A}$, so $R\Gamma(\cal{O}_V) \otimesl_A k \simeq 0$, and by Theorem~\ref{Theorem:torindepmixed} (iii), $H^0(\cal{O}_V) \otimesl_A k$ is a discrete ring. 
Since $V$ is cohomologically pure in codimension $1$, each $H^i(\cal{O}_V)$ is supported on $\{\ideal{m}_A\}$ for $i \ge 1$, 
so if $H^i(\cal{O}_V) \neq 0$ for some $i \ge 1$, then $H^i(\cal{O}_V) \otimesl_A k$ is non-zero in cohomological degree $-d$.
Let $j\in [d-2]$ be the smallest number such that $H^j(\cal{O}_V) \neq 0$. By investigating the spectral sequence
\[E_2^{i,j} = \mathrm{Tor}_{-i}^A(H^{j}(\cal{O}_V), k) \Rightarrow H^{i+j}(R\Gamma(\cal{O}_V) \otimesl_A k),\]
we conclude that the $E_2^{-d,j}$-term survives in the $E_{\infty}$-page, which contradicts the fact that $R\Gamma(\cal{O}_V) \otimesl_A k \simeq 0$.
\end{proof}
\section{Affineness of the maximal \'{e}tale locus}
We establish Theorem~\ref{thm:affine-mixed}.
\begin{theorem}
Let $X \to Y$ be a morphism of finite-type between locally Noetherian excellent schemes, where $Y$ is regular and $X$ is normal. 
Let $V \subset X$ be the maximal open subset such that $f|_V: V \to Y$ is \'{e}tale. The inclusion $V \to X$ is an affine morphism.
\end{theorem}
\begin{proof}
We first prove the result under the assumption that $Y$ is excellent following \cite[0ECD]{sp}.\\
\indent We first reduce to the situation that $Y$ has finite dimension.
For any $x \in X$, we need to find an affine open neighborhood $U$ of $x$ such that $U \cap V$ is affine. 
It is therefore necessary and sufficient to show that $V_x = V \times_X \spec{\cal{O}_{X,x}}$ is affine for every $x \in X$ \cite[01Z6]{sp}. 
Hence, it is necessary and sufficient to show the result for the base-change of the morphism
$X \times_Y \spec{\cal{O}_{Y,y}} \to \spec{\cal{O}_{Y,y}}$ for each $y \in Y$, so we may assume $Y$ is local of finite Krull-dimension. 
By induction on the dimension $d$ of $Y$, we may also assume that $V \cap (X \setminus f^{-1}(y)) \to X \setminus f^{-1}(y)$ is affine.\\
\indent For any $x \in f^{-1}(y)$, if $x \in V$ then $V_x$ is clearly affine.
Since $V \to Y$ is quasi-finite, $V \cap f^{-1}(y)$ is a finite disjoint union of closed points, so if $x \in f^{-1}(y)$ is not in $V$, 
then $x$ has an open affine neighborhood $U$ not meeting $f^{-1}(y)$. By replacing $X$ by $U$ we've reduced to the case that $X$ is affine, $V \to Y$ misses the closed point $y \in Y$, and 
$V \to \spec{\cal{O}_{Y,y}}-\{y\}$ is an affine morphism. If $d =0$ or $1$, then $V$ is affine. Hence, we assume $d \ge 2$.\\
\indent Now assume that $Y$ is excellent. Let $e(V)$ be the largest $i$ such that $H^i(\cal{O}_V) \neq 0$. As $X$ is affine, $X \setminus V$ is defined by an ideal $I \subset \cal{O}_X$, and by \cite[0DXB]{sp}, we have
\[e(V)= \max_{x \in X} e(V_x) - 1, \; \; e(V_x) = \max \{i: H^i_{I\cal{O}_{X,x}}(\cal{O}_{X,x}) \neq 0\}.\]
If $x$ is not a specialization of any point in $V$, then $V_x$ is empty and thus $e(V_x)=0$.
If $x$ is a specialization of some point in $V$, then $\mathrm{dim}(\cal{O}_{X,x}) \le d$ by the dimension formula \cite[02JU]{sp}. $V_x$ is contained in the punctured 
spectrum of $\spec{\cal{O}_{X,x}}$, so by the purity of the ramification locus $X \setminus V$, if
$\mathrm{dim}(\cal{O}_{X,x}) \ge 2$, then 
the dimension of the complement of $V_x$ in $\spec{\cal{O}_{X,x}}$ is at least $1$, and so by Hartshorne-Lichtenbaum vanishing (\cite[0EB7]{sp}), 
$e(V_x) \le \mathrm{dim}(\cal{O}_{X,x})-1 \le d-1$. If $\mathrm{dim}(\cal{O}_{X,x}) \le 1$, then $e(V_x) \le 1$, and therefore, $e(V) \le \max(0,d-2)=d-2$.\\
\indent As $V \to \spec{\cal{O}_{Y,y}} - \{y\}$ is affine and quasi-compact, for $i \ge 1$, $H^i(\cal{O}_V)$ is supported on $\{y\}$. Moreover, $V \to \spec{\cal{O}_{Y,y}}$ 
is separated, and since we know $e(V) \le d-2$, we conclude that $V \to \spec{\cal{O}_{Y,y}}$ is an \'{e}tale morphism that is cohomologically pure in codimension $1$.
By Theorem~\ref{thm:cohpuregen}, $V$ is affine as desired.\\
\end{proof}

\backmatter
\SingleSpacing
\renewcommand{\bibname}{References}
\setlength{\bibitemsep}{\baselineskip}
\bibliographystyle{ieeetr}
\bibliography{references}
\end{document}